\documentclass[11pt,a4paper, reqno]{article}
\usepackage[utf8]{inputenc}
\usepackage{amsfonts}
\usepackage{amssymb}
\usepackage{enumerate}
\usepackage{graphicx}
\usepackage{mathtools}
\usepackage{amsmath, amsthm}
\usepackage{tikz}
\usepackage{color}
\usepackage[affil-it]{authblk}
\usepackage[T1]{fontenc}
\usepackage[sort,numbers]{natbib}
\usepackage{bbm}
\usepackage[a4paper, top = 1in, bottom = 1in, left = 1in, right = 1in]{geometry}
\usepackage{amsmath}
\usepackage{amsfonts}
\usepackage{amssymb}
\usepackage{bbm}
\usepackage{mathrsfs}
\usepackage[utf8]{inputenc}
\usepackage[english]{babel}
\usepackage{hyperref}
\usepackage{relsize}
\usepackage{accents}
\usepackage[noabbrev]{cleveref}
\usepackage{appendix}
\usepackage{lipsum}
\usepackage[color = white]{todonotes}

\newcommand\blfootnote[1]{%
  \begingroup
  \renewcommand\thefootnote{}\footnote{#1}%
  \addtocounter{footnote}{-1}%
  \endgroup
}

\makeatletter
  \@addtoreset{chapter}{part}
  \@addtoreset{@ppsaveapp}{part}
\makeatother

\long\def\/*#1*/{}

\usepackage{palatino}

\makeatletter
\newsavebox\myboxA
\newsavebox\myboxB
\newlength\mylenA

\newcommand*\xoverline[2][0.75]{%
    \sbox{\myboxA}{$\m@th#2$}%
    \setbox\myboxB\null
    \ht\myboxB=\ht\myboxA%
    \dp\myboxB=\dp\myboxA%
    \wd\myboxB=#1\wd\myboxA
    \sbox\myboxB{$\m@th\overline{\copy\myboxB}$}
    \setlength\mylenA{\the\wd\myboxA}
    \addtolength\mylenA{-\the\wd\myboxB}%
    \ifdim\wd\myboxB<\wd\myboxA%
       \rlap{\hskip 0.5\mylenA\usebox\myboxB}{\usebox\myboxA}%
    \else
        \hskip -0.5\mylenA\rlap{\usebox\myboxA}{\hskip 0.5\mylenA\usebox\myboxB}%
    \fi}
\makeatother

\usepackage{environ}
\NewEnviron{eq}{%
\begin{equation}\begin{split}
  \BODY
\end{split}\end{equation}
}

\usepackage{hyperref}
\hypersetup{
  colorlinks,
  linkcolor=blue,
  citecolor=black,
  filecolor=black,
  urlcolor=black,
  pdftitle={},
  pdfauthor={S. Dhara, R. van der Hofstad, J. van Leeuwaarden, S. Sen},
  pdfcreator={S. Dhara},
  pdfsubject={},
  pdfkeywords={}
}
\numberwithin{equation}{section}

\newcommand{\cG}{\mathcal{G}}
\newcommand{\cE}{\mathcal{E}}

\def\sss{\scriptscriptstyle}
\newcommand{\ubar}[1]{\underaccent{\bar}{#1}}

\newcommand*{\Mtilde}[1]{\skew{5}{\tilde}{#1}}
\newcommand{\prob}[1]{\ensuremath{\mathbbm{P}\left(#1\right)}}
\newcommand{\expt}[1]{\ensuremath{\mathbbm{E}\left[#1\right]}}
\newcommand{\var}[1]{\ensuremath{\mathrm{Var}\left(#1\right)}}
\newcommand{\refl}[1]{\ensuremath{\mathrm{refl}(#1)}}

\newcommand{\floor}[1]{\ensuremath{\left\lfloor #1 \right\rfloor}}
\newcommand{\ind}[1]{\ensuremath{\mathbbm{1}_{\left\{#1\right\}}}}
\newcommand{\pto}{\ensuremath{\xrightarrow{\mathbbm{P}}}}
\newcommand{\dto}{\ensuremath{\xrightarrow{d}}}
\newcommand{\asto}{\ensuremath{\xrightarrow{\sss \mathrm{a.s.}}}}
\newcommand{\surp}[1]{\ensuremath{\mathrm{SP}(#1)}}

\newcommand{\diam}{\ensuremath{\mathrm{diam}}}

\newcommand{\PR}{\ensuremath{\mathbbm{P}}}
\newcommand{\tPR}{\ensuremath{\tilde{\mathbbm{P}}}}
\newcommand{\E}{\ensuremath{\mathbbm{E}}}
\newcommand{\tE}{\ensuremath{\tilde{\mathbbm{E}}}}
\newcommand{\R}{\ensuremath{\mathbbm{R}}}
\newcommand{\Q}{\ensuremath{\mathbbm{Q}}}
\newcommand{\N}{\ensuremath{\mathbbm{N}}}

\newcommand{\e}{\ensuremath{\mathrm{e}}}
\newcommand{\im}{\ensuremath{\mathrm{i}}}
\newcommand{\1}{\ensuremath{\mathbbm{1}}}
\newcommand{\OP}{\ensuremath{O_{\sss\PR}}}
\newcommand{\oP}{\ensuremath{o_{\sss\PR}}}
\newcommand{\thetaP}{\ensuremath{\Theta_{\sss\PR}}}
\newcommand{\CM}{\ensuremath{\mathrm{CM}_n(\boldsymbol{d})}}
\newcommand{\CMP}{\ensuremath{\mathrm{CM}_n(\boldsymbol{d},p_c(\lambda))}}

\newcommand{\shortarrow}{{\sss\downarrow}}

\newcommand{\dif}{\mathrm{d}}
\newcommand{\bld}[1]{\boldsymbol{#1}}

\newcommand{\rE}{\mathrm{E}}

\newcommand{\rCM}{\mathrm{CM}}

\newcommand{\fI}{\mathfrak{I}}
\newcommand{\xx}{\bld{x}}

\newcommand{\biS}{\mathbf{S}_{\infty}^\lambda}
\newcommand{\iS}{S_{\infty}^\lambda}
\newcommand{\rC}{\mathrm{C}}
\newcommand{\probc}[1]{\ensuremath{\tilde{\mathbbm{P}}\left(#1\right)}}
\newcommand{\exptc}[1]{\ensuremath{\tilde{\mathbbm{E}}\left[#1\right]}}
\newcommand{\sC}{\mathscr{C}}
\newcommand{\sCi}{\mathscr{C}_{\sss (i)}}

\newcommand{\sCil}{\mathscr{C}_{\sss (i)}}
\newcommand{\tell}{\tilde{\ell}}
\newcommand{\td}{\tilde{d}}

\newcommand{\cA}{\mathcal{A}}
\newcommand{\Unot}{\mathbb{U}^0_{\shortarrow}}
\newcommand{\SP}{\mathrm{SP}}
\newcommand{\Gp}{\mathcal{G}_n(p_n)}
\newcommand{\tV}{\tilde{V}}
\newcommand{\tS}{\tilde{S}}
\newcommand{\tA}{\tilde{A}}
\newcommand{\cI}{\mathcal{I}}
\newcommand{\sF}{\mathscr{F}}

\newcommand{\sV}{\mathscr{V}}

\newcommand{\cL}{\mathcal{L}}
\newcommand{\cR}{\mathcal{R}}

\newtheorem{theorem}{Theorem}
\newtheorem{algo}{Algorithm}
\newtheorem{lemma}[theorem]{Lemma}
\newtheorem{proposition}[theorem]{Proposition}

\newtheorem{assumption}{Assumption}
\newtheorem{remark}{Remark}
\newtheorem{fact}{Fact}
\newtheorem{defn}{Definition}

\begin{document}
\title{Critical percolation on scale-free random graphs:\\ New universality class for the configuration model}
\author{Souvik Dhara$^{1,2}$, Remco van der Hofstad$^{3}$, and Johan S.H. van Leeuwaarden$^4$}
\maketitle
\begin{abstract}
In this paper, we study the critical behavior of percolation on a configuration model with degree distribution satisfying an infinite second-moment condition, which includes power-law degrees with exponent $\tau \in (2,3)$.
It is well known that, in this regime, many canonical random graph models, such as the configuration model, are \emph{robust} in the sense that the giant component is not destroyed when the percolation probability stays bounded away from zero.
Thus, the critical behavior is observed when the percolation probability tends to zero with the network size, despite of the fact that the average degree remains bounded.

In this paper, we initiate the study of critical random graphs in the infinite second-moment regime by identifying the critical window for the configuration model.
We prove scaling limits for component sizes and surplus edges, and show that the maximum diameter the critical  components is of order $\log n$, which contrasts with the previous universality classes arising in the literature.
This introduces a third and novel universality class for the critical behavior of percolation on random networks, that is not covered by the multiplicative coalescent framework due to Aldous and Limic \cite{AL98}. 
We also prove concentration of the component sizes outside the critical window, and that a unique, complex \emph{giant component} emerges after the critical window.
This completes the picture for the percolation phase transition on the configuration model.
\end{abstract}

\blfootnote{\emph{Emails:} 
 \href{mailto:sdhara@mit.edu}{sdhara@mit.edu},
 \href{mailto:r.w.v.d.hofstad@tue.nl}{r.w.v.d.hofstad@tue.nl},
 \href{mailto:j.s.h.vanleeuwaarden@tilburguniversity.edu}{j.s.h.vanleeuwaarden@tilburguniversity.edu}} 
\blfootnote{$^1$Department of Mathematics, Massachusetts Institute of Technology}
\blfootnote{$^2$Microsoft Research}
\blfootnote{$^3$Department of Mathematics and Computer Science, Eindhoven University of Technology}
\blfootnote{$^4$Stochastic Operations Research, Tilburg University}
\blfootnote{\emph{Acknowledgment}. 
This project was supported by the Netherlands Organisation for Scientific Research (NWO) through Gravitation Networks grant 024.002.003. 
In addition, RvdH was supported by VICI grant 639.033.806.
We sincerely thank the referee for an extremely thorough review, and in particular for pointing out an error in the properties of the scaling limit. 
Also, we sincerely thank Shankar Bhamidi and Debankur Mukherjee for carefully reading the revised proof in  Section~\ref{sec:properties-exploration}. 
}

\section{Introduction}
Bond percolation, or simply percolation, refers to the random graph obtained by independently keeping each edge of a graph with some fixed probability $p$ (and deleting with probability $1-p$).
Percolation is a classical and important model in statistical physics and network science, as it serves as a canonical model for assessing robustness of a network when the edges of the underlying network are randomly damaged, and also as a basic model of vaccination for the prevention of an epidemic on networks. 
A detailed account of many of these applications can be found in \cite{Newman-book,Bar16}. 
From a theoretical perspective, percolation is one of the most  elementary models that exhibits a phase transition, i.e., there exist values $p_c = p_c(n)$ such that for $p>p_c(1+\varepsilon)$ and $\varepsilon>0$, the proportion of vertices in the largest connected component is bounded away from zero with high probability, whereas for $p<p_c(1-\varepsilon)$ this proportion becomes negligible. 
The critical behavior is observed when $p\approx p_c$, and  fascinating behavior starts to emerge for the percolation process around this critical value. 

It turns out that there is a window of values of $p$ where the component functionals show intermediate and unique behavior. For example, rescaled component functionals converge to non-degenerate scaling limits, in contrast to the fact that they always concentrate for other values of $p$.
Also, the large components in this window are structurally intermediate in the sense that neither there is a giant component with a growing number of cycles, nor do the components look like trees.
This regime is called the \emph{critical window} of the percolation phase-transition. 
Starting with the pioneering work of Aldous \cite{A97}, deriving scaling limits for critical component functionals has been the ground for an enormous literature with several interesting scaling-limit results over the past decades \cite{BHL10,BHL12,NP10a,NP10b,Jo10,DHLS15,DHLS16,R12,AP00,AL98}. 
We refer the reader to \cite[Chapter 1]{Dha18} and references therein for an elaborate discussion of the nature of this transition, and a literature overview.  

In the literature, two fundamentally different types of behavior have been proved for the scaling limits and the critical exponents associated to the critical window and component sizes depending on whether the asymptotic degree distribution satisfies a finite third-moment condition \cite{BHL10,DHLS15} or an infinite third - but a finite second-moment condition \cite{BHL12,DHLS16}.
However, the study of critical behavior in the infinite second-moment setting was an open question.

When the degree distribution is asymptotically a power-law with exponent $\tau\in (2,3)$, then the finite second-moment condition fails.
These networks are popularly known as \emph{scale-free networks} \cite{Bar16} in the literature.
Many real-world networks are observed to be scale-free~\cite{RGCN1,RTD-book,Newman-book,AB02}.
One of the well-known features of scale-free networks is that they are \emph{robust} under random edge-deletion, i.e., for any sequence $(p_n)_{n\geq 1}$ satisfying $\liminf_{n\to\infty} p_n > 0$, the graph obtained by performing percolation with probability $p_n$ is supercritical. 
This feature has been studied experimentally in~\cite{AJB00}, using heuristic arguments in~\cite{CEbAH00,CNSW00,DGM08,CbAH02} (see also \cite{braunstein2003optimal,braunstein2007optimal,Halvin05} in the context of optimal paths in the strong disorder regime), and mathematically in \cite{BR03}. 
Thus, in order to observe the percolation critical behavior, one needs to have $p_n \to 0$ with the network size, despite of the fact that the average degree of the network remains bounded.

In this paper, we initiate the study of critical behavior in the scale-free regime.
As a canonical random graph model on which percolation acts, we take the multigraph generated
by the configuration model. When the degree distribution satisfies a power-law with
exponent $\tau\in (2,3)$, it was heuristically argued in \cite{CbAH02,DGM08} that the critical value is $p_c \sim n^{-(3-\tau)/(\tau-1)}$, so that the critical window is given by the collection of values $p_c = p_c(\lambda)=\lambda n^{-(3-\tau)/(\tau-1)}$, where $\lambda>0$ indicates the location inside the critical window.
We establish that the scaling exponents from \cite{CbAH02,DGM08} are indeed true, and discuss asymptotics of component functionals inside the critical window.
We also show that $p_c = p_c(\lambda)=\lambda n^{-(3-\tau)/(\tau-1)}$ with $\lambda >0$ gives the right critical window, by showing that a giant component emerges at the end of the critical window ($\lambda\to\infty$), while components have a trivial star-like structure before the critical window ($\lambda \to 0$). 
The main contributions of this paper can be summarized as follows:

\paragraph*{Critical window.}
At criticality, we obtain scaling limits for the largest component sizes and surplus edges in a strong topology. 
The result displays a completely new universality class of scaling limits of critical components. 
The scaling limits here are different from the general multiplicative coalescent framework in \cite{AL98}. 
In particular, the limiting exploration process has bounded variation, so that the general tools from \cite{AL98} cannot be applied. 
We also study the diameter of these components and show that the maximum diameter is of order $\log n$.

\paragraph*{Near-critical behavior.} For $p_n= \lambda_n n^{-(3-\tau)/(\tau-1)}$ with $\lambda_n \to 0$,  the graph is subcritical and we show that the largest components sizes, rescaled by $n^\alpha p_n$, concentrate.
On the other hand, when $\lambda_n \to \infty$, the largest component size, rescaled by $n p_n ^{1/(3-\tau)}$, concentrates, and this is the \emph{unique giant component} in the sense that the size of the second largest component is much smaller than $n p_n ^{1/(3-\tau)}$. 
The nature of the emergence of this giant component for $p_n\gg p_c$ is markedly different compared to the universality classes in the $\tau\in (3,4)$ and the $\tau>4$ regimes, where the giant emerges when the percolation probability satisfies  $(p_n-p_c(\lambda_1)) \gg (p_c(\lambda_2) - p_c(\lambda_1))$, for some strictly positive $p_c$ and $-\infty<\lambda_1<\lambda_2<\infty$ \cite{HJL16}.

\paragraph*{Methods.}
Technically, analyzing percolation on random graphs like the configuration model is challenging, because in order to make Aldous's exploration process approach \cite{A97} work, one is required to keep track of many functionals of the unexplored part of the graph \cite{NP10b}, resulting in a high-dimensional exploration process.
This difficulty was circumvented in \cite{DHLS15,DHLS16} by using Janson's algorithm \cite{J09}. 
Unfortunately, Janson's algorithm does not work here due to the fact that the algorithm creates $n-o(n)$ degree-one vertices. 
Instead, we sandwich the percolated graph in between two configuration models, which yield the same scaling limits for the component sizes.
Also, in order to deduce scaling limits of the component sizes from that of the exploration process, we prove several properties of the limiting exploration process, which are interesting from an independent perspective.

\begin{remark}[Single-edge constraint] \normalfont 
In a parallel work \cite{BDHL18}, Bhamidi and the first two authors consider critical percolation on simple random graphs, i.e., random graphs having no multiple-edges, namely generalized random graphs. 
It turns out that the critical window there is $p_c \sim n^{-(3-\tau)/2} \gg n^{-(3-\tau)/(\tau-1)}$. 
This is a distinctive feature in the infinite second-moment case that never surfaced in the other two universality classes of critical random graphs. 
\end{remark}

\paragraph*{Organization of the paper.} 
In Section~\ref{sec:main-results}, we state our results precisely.
In Section~\ref{sec:model}, we give the precise definitions of the model and the scaling limits. Section~\ref{sec:discussion} is devoted to comments about the heuristics, and some important special cases.
In Section~\ref{sec:properties-exploration}, we study excursions of the limiting exploration process.
Section~\ref{sec:proofs} contains the proofs of the results at criticality, and in Section~\ref{sec:near-critical-proofs}, we analyze the near-critical regimes.

\section{Main results} \label{sec:main-results}
\subsection{The configuration model}\label{sec:model}
\subsubsection{Model description}
The configuration model generates random multigraphs with any given degree sequence. 
Consider $n$ vertices labeled by $[n]:=\{1,2,...,n\}$ and a non-increasing sequence of degrees $\boldsymbol{d} =\bld{d}_n = ( d_i )_{i \in [n]}$ such that $\ell_n = \sum_{i \in [n]}d_i$ is even. 
The configuration model on $n$ vertices having degree sequence $\boldsymbol{d}$ is constructed as follows \cite{B80,BC78}:
 \begin{itemize}
 \item[] Equip vertex $j$ with $d_{j}$ stubs, or \emph{half-edges}. Two half-edges create an edge once they are paired. Therefore, initially we have $\ell_n=\sum_{i \in [n]}d_i$ half-edges. Pick any half-edge and pair it with a uniformly chosen half-edge from the remaining unpaired half-edges and remove both these half-edges from the set of unpaired half-edges.
 Keep repeating the above procedure until all half-edges are paired. 
 \end{itemize}
  Let $\CM$ denote the graph constructed by the above procedure.
  Note that $\CM$ may contain self-loops and multiple edges. 
In fact, the probability that $\CM$ is a simple graph tends to zero in our setting with an infinite second-moment condition on the degree distribution \cite[Proposition 7.12]{RGCN1}. 
Before stating the main results about the configuration model, we set up some necessary notation.
\subsubsection{Notions of convergence and the limiting objects} 
\label{sec:notation}
To describe the main results of this paper, we need some definitions and notations. 
We use the Bachmann–Landau asymptotic notation $O(\cdot)$, $o(\cdot)$, $\Theta(\cdot)$ for large-$n$ asymptotics of real numbers. 
For $(a_n)_{n\geq 1}, (b_n)_{n\geq 1} \subset (0,\infty)$, we write $a_n\ll b_n$, $a_n \sim b_n$ and $a_n\gg b_n$ as a shorthand for $\lim_{n\to\infty}a_n/b_n = 0,1,\infty$, respectively. 
We often use $C$ as a generic notation for a positive constant whose value can be different in different lines.  
We also use the standard notation of $\xrightarrow{\sss\PR}$, and $\xrightarrow{\sss d}$ to denote convergence in probability and in distribution, respectively.
The topology needed for the  convergence in distribution will always be specified unless it is  clear from the context. 
 We say that a sequence of events $(\mathcal{E}_n)_{n\geq 1}$ occurs with high probability~(whp) with respect to the probability measures $(\mathbbm{P}_n)_{n\geq 1}$  when $\mathbbm{P}_n\big( \mathcal{E}_n \big) \to 1$. 
 Define $f_n = O_{\sss\mathbbm{P}}(g_n)$ when  $ ( |f_n|/|g_n| )_{n \geq 1} $ is tight; $f_n =o_{\sss\mathbbm{P}}(g_n)$ when $f_n/g_n  \xrightarrow{\sss\PR} 0 $; $f_n =\thetaP(g_n)$ if both $f_n=\OP(g_n) $ and $g_n=\OP(f_n)$. 
 Denote
\begin{equation}
\ell^p_{\shortarrow}:= \Big\{ \mathbf{x}= (x_i)_{i=1}^\infty \subset [0,\infty): \ x_{i+1}\leq x_i\ \forall i,\text{ and } \sum_{i=1}^{\infty} x_{i}^p < \infty \Big\}
\end{equation} with the $p$-norm metric $d(\mathbf{x}, \mathbf{y})= \big( \sum_{i=1}^{\infty} |x_i-y_i|^p \big)^{1/p}$. Let $\ell^2_{\shortarrow} \times \mathbbm{N}^{\infty}$  denote the product topology of $\ell^2_{\shortarrow}$ and $\mathbbm{N}^{\infty}$ with $\mathbbm{N}^{\infty}$ denoting the sequences on $\mathbbm{N}$ endowed with the product topology. Define also
\begin{equation}\mathbb{U}_{\shortarrow}:= \Big\{ ((x_i,y_i))_{i=1}^{\infty}\in  \ell^2_{\shortarrow} \times \mathbbm{N}^{\infty}: \sum_{i=1}^{\infty} x_iy_i < \infty \text{ and } y_i=0 \text{ whenever } x_i=0 \; \forall i   \Big\},
\end{equation}
endowed with the metric \begin{equation} \label{defn_U_metric}\mathrm{d}_{\mathbb{U}}((\mathbf{x}_1, \mathbf{y}_1), (\mathbf{x}_2, \mathbf{y}_2)):= \bigg( \sum_{i=1}^{\infty} (x_{1i}-x_{2i})^2 \bigg)^{1/2}+ \sum_{i=1}^{\infty} \big| x_{1i} y_{1i} - x_{2i}y_{2i}\big|. 
\end{equation} Further, let $\mathbb{U}^0_{\shortarrow} \subset \mathbb{U}_{\shortarrow}$ be given by \begin{equation}\mathbb{U}^0_{\shortarrow}:= \big\{((x_i,y_i))_{i=1}^{\infty}\in\mathbb{U}_{\shortarrow} : \text{ if } x_k = x_m, k \leq m,\text{ then }y_k \geq y_m\big\}.
\end{equation}  Let $(\mathbb{U}^0_{\shortarrow})^k$ denote the $k$-fold product space of $\mathbb{U}^0_{\shortarrow}$.

Throughout, we write $\mathbb{D}[0,\infty)$ to denote the space of c\`adl\`ag functions $[0,\infty)\mapsto \R$ equipped with the Skorohod $J_1$-topology. Also, let $\mathbb{D}_+[0,\infty) \subset \mathbb{D}[0,\infty)$ be the collection of functions with positive jumps only,  and $\mathbb{C}[0,\infty)\subset \mathbb{D}[0,\infty)$ be the collection of continuous functions. 
For any fixed $T>0$, $\mathbb{D}[0,T], \mathbb{D}_+[0,T], \mathbb{C}[0,T]$ are defined similarly for functions $[0,T]\mapsto \R$. 
For any function $f\in \mathbb{D}[0,\infty)$,  define $\ubar{f}(t)=\inf_{s\leq t}f(s)$.  
Note that 
$\ubar{f}$ is non-increasing. Moreover, 
\begin{eq} \label{continuity-ubar-f}
\ubar{f} \in \mathbb{C}[0,\infty), \quad \text{whenever} \quad f\in \mathbb{D}_+[0,\infty).
\end{eq} 
Indeed, if $\ubar{f}$ is discontinuous at some point $t$, then $\ubar{f}(t-) > \ubar{f}(t)$, but that would mean that $f$ has a negative jump of size $ \ubar{f}(t-) - \ubar{f}(t)$ at $t$. Thus \eqref{continuity-ubar-f} holds. 
Next, for any $f\in \mathbb{D}_+[0,\infty)$, define the zero set of $f$ by $\mathscr{Z}_f = \{t\geq 0: f(t) - \ubar{f}(t)=0\}$, and let $\mathrm{cl}(\mathscr{Z}_f)$ denote the closure of $\mathscr{Z}_f$. 
An interval $(l,r)$ is called an excursion above the past minimum of $f$,  or simply excursion of $f$ (see \cite[Section IV.2]{Ber01}) if 
\begin{eq}\label{defn:excursion}
f(t) - \ubar{f}(t)>0, \quad \forall t\in (l,r), \text{ where } l\in \mathrm{cl}(\mathscr{Z}_f)\text{ and } r\in  \mathrm{cl}(\mathscr{Z}_f) \cup \{\infty\}. 
\end{eq}
For $f\in \mathbb{D}_+[0,T]$, we consider $(l,r) \subset [0,T]$, and define an excursion similarly as in~\eqref{defn:excursion}.

 We often use boldface notation $\mathbf{X}$ for the stochastic process $( X(s) )_{s \geq 0}$, unless stated otherwise. 
Consider a decreasing sequence $ \boldsymbol{\theta}=(\theta_1,\theta_2,\dots)\in \ell^2_{\shortarrow}\setminus \ell^1_{\shortarrow}$. Denote by  $\mathcal{I}_i(s):=\ind{\xi_i\leq s }$ where $\xi_i\sim \mathrm{Exp}(\theta_i/\mu)$ independently, and $\mathrm{Exp}(r)$ denotes the exponential distribution with rate $r$.  Consider the process \begin{equation}\label{defn::limiting::process}
\iS(t) =  \frac{\lambda \mu}{\|\bld{\theta}\|_2^2} \sum_{i=1}^{\infty} \theta_i\mathcal{I}_i(t)-  t,
\end{equation}
for some $\lambda, \mu >0$.
Note that, for all $t>0$,  $\E[\iS(t)]<\infty$  since $\sum_{i}\theta_i^2<\infty$, and consequently $\iS(t)<\infty$, almost surely. 
Also, for any $u<t$, 
\begin{eq}
\E\big[|\iS(t) - \iS(u)|\big] \leq \frac{\lambda \mu}{\|\bld{\theta}\|_2^2} \sum_{i=1}^{\infty} \theta_i \e^{-\theta_i u}(1-\e^{-\theta_i (t-u)})+ |t-u| \leq (\lambda\mu+1)|t-u|,
\end{eq}
so that $\biS$ has bounded variation almost surely.
However, since $\sum_i\theta_i = \infty$, the process experiences infinitely many jumps in any bounded interval of time.
Define the reflected version of $\iS(t)$ by
\begin{equation} \label{defn::reflected-Levy}
 \refl{ \iS(t)}= \iS(t) - \min_{0 \leq u \leq t} \iS(u).
\end{equation}
We will show that, for any $\lambda>0$, the excursion lengths of the process $\biS = (\iS(t))_{t\geq 0}$ can be ordered almost surely as an element of $\ell^2_{\shortarrow}$. We denote this ordered vector of excursion lengths by $(\gamma_i(\lambda))_{i\geq 1}$.
For $v,t> 0$, define 
$M_t(v) := \sum_{j: v\theta_j\leq 1,\  t\theta_j\leq 1 } \theta_j^3.$
We will assume that for any $t>0$,
\begin{eq}\label{density-assumption}
\int_{0}^\infty  \e^{- tv^2M_t(v)} \dif v<\infty. 
\end{eq}
The technical condition in \eqref{density-assumption} on top of $\bld{\theta}\in \ell^2_{\shortarrow}\setminus\ell^1_{\shortarrow} $ will be used to ensure that the distribution of $\iS(t)$ is non-atomic for all $t>0$ (see Lemma~\ref{lem:exc} below), which in turn implies that we have strict ordering between excursion lengths, i.e., $\gamma_{i+1}(\lambda) <\gamma_i(\lambda)$ for all $i\geq 1$ almost surely. 
The condition \eqref{density-assumption} is relatively weak, and is, for example, satisfied for $\theta_j = j^{-\alpha}$ for $\alpha \in (1/2,1)$. To see this, note that $v^2M_t(v) $ is of the same order as $v^{-1+1/\alpha}$. 
However, this also shows that \eqref{density-assumption} is not satisfied for the extreme case $\alpha = 1$, i.e.,  $\theta_j = j^{-1}$.

 Also, define the counting process $\mathbf{N}^\lambda = (N^\lambda(t))_{t\geq 0}$ to be the Poisson process that has intensity $(\lambda\mu^2)^{-1}\|\bld{\theta}\|_2^2 \times\refl{ \iS(t)}$ at time $t$, conditionally on $( \iS(u) )_{u \leq t}$. Formally, $\mathbf{N}^\lambda$ is characterized as the counting process for which 
\begin{equation} \label{defn::counting-process}
N^\lambda(t) - \frac{\|\bld{\theta}\|_2^2}{\lambda\mu^2} \int\limits_{0}^{t}\refl{ \iS(u)}\dif u
\end{equation} is a martingale.  We use  the notation $N_i(\lambda)$ to denote the number of marks of $\mathbf{N}^\lambda$ in the $i$-th largest excursion of $\biS$.
Define 
\begin{eq}\label{eq:Z-limit}
\mathbf{Z}(\lambda):= ((\gamma_i(\lambda),N_i(\lambda)))_{i\geq 1}, \text{ ordered as an 
 element of }\Unot.
\end{eq}

\subsubsection{Results for the critical window}
Fix $\tau \in (2,3)$. 
Throughout this paper, we denote 
\begin{equation}\label{eqn:notation-const}
 \alpha= 1/(\tau-1),\qquad \rho=(\tau-2)/(\tau-1),\qquad \eta=(3-\tau)/(\tau-1). 
\end{equation}
Also, let $D_n$ be the degree of a vertex chosen uniformly at random from $[n]$.
We start by stating our assumptions on the degree sequences:
\begin{assumption}
\label{assumption1}
\normalfont  For each $n\geq 1$, let $\bld{d}=\boldsymbol{d}_n=(d_1,\dots,d_n)$ be a degree sequence satisfying $d_1\geq d_2\geq\ldots\geq d_n$. 
We assume the following about $(\boldsymbol{d}_n)_{n\geq 1}$ as $n\to\infty$:
\begin{enumerate}[(i)] 
\item \label{assumption1-1} (\emph{High-degree vertices}) For any  $i\geq 1$, 
$ n^{-\alpha}d_i\to \theta_i,$
where $\boldsymbol{\theta}:=(\theta_i)_{i\geq 1}\in \ell^2_{\shortarrow}\setminus \ell^1_{\shortarrow}$ is such that \eqref{density-assumption} holds. 
\item \label{assumption1-2} (\emph{Moment assumptions}) 
$(D_n)_{n\geq 1}$ is uniformly integrable,  $\lim_{n\to\infty}\frac{1}{n}\sum_{i\in [n]}d_i= \mu$ for some $\mu>0$, and 
 \begin{eq}\label{eq:unif-int-2nd-moment}
 \lim_{K\to\infty}\limsup_{n\to\infty}n^{-2\alpha} \sum_{i=K+1}^{n} d_i^2=0.
 \end{eq}
\end{enumerate}
\end{assumption}
\noindent 
In Section~\ref{sec:discussion}, we discuss the generality of Assumption~\ref{assumption1} and show that  power-law degrees satisfy these assumptions. 
For $\CM$, the \emph{criticality parameter} $\nu_n$ is defined as
\begin{equation}
\nu_n = \frac{\sum_{i\in [n]}d_i(d_i-1)}{\sum_{i\in [n]}d_i}.
\end{equation} 
Molloy and Reed~\cite{MR95}, and Janson and Luczak~\cite{JL09} showed that, under some regularity conditions, $\CM$ has a unique giant component (a component of size $\Theta(n)$) with high probability precisely when $\nu_n \to \nu>1$. 
Under Assumption~\ref{assumption1}, $\nu_n\to\infty$, as $n\to\infty$ since $\sum_{i\in [n]} d_i^2 \geq d_1^2= \Theta(n^{2\alpha})\gg n$, and $\CM$ always contains a giant component (see the remark below \cite[Theorem 4.5]{Hof17} and consider $\pi =1$).

We study \emph{percolation}, which refers to deleting each edge of a graph independently with probability $1-p$.  
In case of percolation on random graphs, the deletion of edges is  also independent from the underlying graph.
The percolation probability is allowed to depend on the network size, i.e., $p = p_n$.
 Let $\rCM_n(\bld{d},p_n)$ denote the graph obtained from percolation with probability $p_n$ on the graphs $\mathrm{CM}_n(\boldsymbol{d})$. 
Fountoulakis \cite{F07} showed that $\rCM_n(\bld{d},p_n)$ is distributed as $\rCM_n(\bld{d}^p)$, where $\bld{d}^p$ is the degree sequence of the percolated graph. 
Note that the degrees in $\bld{d}^p$ could be correlated, so later Janson~\cite{J09} gave an explicit construction which is simpler to analyze.  
This construction was used to identify the percolation phase transition in \cite{J09} and to study the critical window in \cite{DHLS15,DHLS16}. 
An interested reader is also referred to \cite[Algorithm 4]{DHLS15} where a construction of the whole percolation process $(\rCM_n(\bld{d},p))_{p\in [0,1]}$ is provided.

Now, under Assumption~\ref{assumption1}, if $\liminf_{n\to\infty}p_n>0$, then $\rCM_n(\bld{d},p_n)$ retains a giant component with high probability, i.e.,~$\rCM_n(\bld{d},p_n)$ is always supercritical;  see the remark below \cite[Theorem 4.5]{Hof17}.
Thus, in order to see the critical behavior, one must take $p_n \to 0$, as $n\to\infty$. 
For $p_n\to 0$, the graph always contains $n-\oP(n)$ degree-zero or isolated vertices, which makes Janson's construction inconvenient  to work with. 

For a sequence of finite graphs, the critical behavior is where we see intermediate behavior in the sense that it inherits some features from the subcritical (such as the absence of the giant component) and the supercritical regimes (the largest component is not a tree). 
 The collection of such values of $p$ is called the critical window. 
However, due to our lack of knowledge about the subcritical phase and the structural propeties therein, it is not a priori evident here how to define the critical window.
One way to define the subcritical regime and the critical window would be to say that inside the critical window, the rescaled vector of ordered component sizes converge to some \emph{non-degenerate} random vector, whereas the component sizes concentrate in the subcritical regime. 
This property has been observed quite universally for the percolation critical window.
In this paper, we take this as our definition of the critical window. 
It is worthwhile to mention that there is a substantial literature on how to define the critical value. 
See \cite{NP08,JW18,BCHSS05,HvdH17,Hof17} for different definitions of the critical probability and related discussions.

We will show that the critical window for percolation on $\CM$ is given  by 
\begin{equation}\label{eq:crit-window-CM}
p_c=p_c(\lambda):= \frac{\lambda}{\nu_n}(1+o(1)), \quad \lambda \in (0,\infty).
\end{equation}  
Notice that, under Assumption~\ref{assumption1}, $ p_c \sim n^{-2\alpha+1} \sim  n^{-\eta}$, where $\eta = (3-\tau)/(\tau-1)>0$.
The case where $p\ll p_c$ will be called the barely subcritical regime and the case $p_c\ll p \ll 1$ will be called the barely supercritical regime. 
We will show that a unique giant component emerges in the barely supercritical regime.
We first state the results about the component sizes and the complexity in the critical window, and then discuss the barely sub-/supercritical regimes.

We will always write $\sC_{\sss (i)}(p)$ to denote the $i$-th largest component in the percolated graph. The random graph on which percolation acts will always be clear from the context. 
A vertex is called isolated if it has degree zero in the graph $\CMP$. 
We define the component size corresponding to an isolated vertex to be zero (see Remark~\ref{rem:isolated} below). 
For any component $\mathscr{C}\subset\CMP$, let $\SP(\mathscr{C})$ denote the number of surplus edges given by $\# \{\text{edges in }\mathscr{C} \}-|\mathscr{C}|+1$. 
Finally, let 
\begin{eq}\label{defn:Z}
\mathbf{Z}_n(\lambda) := \big(n^{-\rho}|\sCi(p_c(\lambda))|,\SP(\sCi(p_c(\lambda)))\big)_{i\geq 1}, \text{ ordered as an element of }\Unot.
\end{eq}
The following theorem gives the asymptotics for the critical component sizes and the surplus edges of  $\CMP$: 
\begin{theorem}[Critical component sizes and surplus edges]\label{thm:main} Under {\rm Assumption~\ref{assumption1}}, as $n\to\infty$,
\begin{equation}
\mathbf{Z}_n(\lambda) \dto \mathbf{Z}(\lambda)
\end{equation}with respect to the $\Unot$ topology, where $\mathbf{Z}(\lambda)$ is defined in \eqref{eq:Z-limit}.
\end{theorem}
\begin{remark}[Ignoring isolated components]\label{rem:isolated} \normalfont
Note that $2\rho<1$ for $\tau\in (2,3)$. 
When percolation is performed with probability $p_c$, there are of the order $n$ isolated vertices and thus $n^{-2\rho}$ times the number of isolated vertices tends to infinity. 
This is the reason why we must ignore the contributions due to isolated vertices, when considering the convergence of the component sizes in the $\ell^2_{\shortarrow}$-topology.
Note that an isolated vertex with self-loops does not create an isolated component. 
\end{remark}


For a connected graph $G$, $\diam(G)$ denotes the diameter of the graph, i.e., the maximum graph distance between any pair of vertices. 
For an arbitrary graph~$G$, $\diam(G):=\max \diam(\mathscr{C})$, where the maximum is taken over all connected components.
Our next result shows that the diameter of the largest connected components is of order $\log n$:
\begin{theorem}[Diameter of largest critical  clusters]\label{thm:diameter-large-comp} Under {\rm Assumption~\ref{assumption1}}, $\diam (\mathrm{CM}_n(\bld{d},p_c(\lambda))) = \OP(\log n)$.
\end{theorem}
Thus, the maximum diameter scales logarithmically in the $\tau \in (2,3)$, in contrast to the other universality classes in the $\tau\in (3,4)$ and $\tau>4$ regimes, where graph distances scale as a positive power of $n$ \cite{ABG09,BHS15}.

\subsubsection{Behavior in the near-critical regimes}
We now discuss asymptotic results for the component sizes in the barely subcritical ($p_n\ll p_c(\lambda)$) and barely supercritical ($p_n\gg p_c(\lambda)$) regimes.
The next two theorems summarize the behavior outside the critical window:
\begin{theorem}[Barely subcritical regime]\label{thm:barely-subcrit} For $\rCM_n(\bld{d},p_n)$, suppose that $n^{-\alpha}\ll p_n\ll p_c(\lambda)$ and that {\rm Assumption~\ref{assumption1}} holds. Then, as $n\to\infty$,
\begin{equation}\label{subcrit-asmp-comp-sp}
\big((n^\alpha p_n)^{-1}|\sCi(p_n)|\big)_{i\geq 1} \pto (\theta_i)_{i\geq 1},
\end{equation} in $\ell^{2}_{\shortarrow}$ topology, and  $\PR(\surp{\sCi(p_n)} =0) \to 1$, for all $i\geq 1$.
\end{theorem}

\begin{remark}[Components and hubs] \normalfont 
In the barely subcritical regime, we show that the $i$-th largest component is essentially the component containing the $i$-th largest degree vertex, or the $i$-th hub. 
Since the hubs have degree $\Theta(n^{\alpha})$, we need the assumption that $p_n\gg n^{-\alpha}$ in Theorem~\ref{thm:barely-subcrit}, as otherwise the hubs become isolated, in which case components are likely to be extremely small. 
\end{remark}
For the result in the barely supercritical regime, let $p_c(\lambda) \ll p_n\ll 1$. 
The exact asymptotics of the high-degree vertices and the tail behavior in \eqref{eq:unif-int-2nd-moment} will not be required. 
Below, we state the sufficient conditions for the concentration of the size of the giant component. 
 In Section~\ref{sec:discussion}, we will see that these conditions are satisfied when the degrees are sampled from a power-law distribution: 
\begin{assumption}\label{assumption-supercrit} \normalfont
For each $n\geq 1$, let $\bld{d}=\boldsymbol{d}_n=(d_1,\dots,d_n)$ be a degree sequence satisfying $d_1\geq d_2\geq\ldots\geq d_n$. 
We assume the following about $(\boldsymbol{d}_n)_{n\geq 1}$:
\begin{enumerate}[(i)]
    \item $d_1 = O(n^\alpha)$.   
    \item $(D_n)_{n\geq 1}$ is uniformly integrable, and   $\lim_{n\to\infty}\frac{1}{n}\sum_{i\in [n]}d_i= \mu$ for some $\mu>0$. 
\item Let $D_n^\star$ denote the degree of a vertex chosen in a size-biased manner with the sizes being $(d_i/\ell_n)_{i\in [n]}$.
Then, there exists a constant $\kappa>0$ such that 
\begin{equation}\label{eq:asymp-laplace}
1 - \E[\e^{-tp_n^{1/(3-\tau)}D_n^{\star}}] = \kappa p_n^{(\tau-2)/(3-\tau)}(t^{\tau-2}+o(1)).
\end{equation}
\end{enumerate}
\end{assumption}
Let $\rE(G)$ denote the number of edges in the graph $G$.
\begin{theorem}[Barely supercritical regime]\label{thm:barely-supercrit} 
For $\rCM_n(\bld{d},p_n)$, suppose that $p_c(\lambda) \ll p_n\ll 1$ and that {\rm Assumption~\ref{assumption-supercrit}} hold.
 Then, as $n\to\infty$,
\begin{equation}
\frac{|\sC_{\sss (1)}(p_n)|}{np_n^{1/(3-\tau)}} \pto \mu\kappa^{1/(3-\tau)}, \quad \frac{\rE(\sC_{\sss (1)}(p_n))}{np_n^{1/(3-\tau)}} \pto \mu\kappa^{1/(3-\tau)},
\end{equation}and for all $i\geq 2$, $|\sCi(p_n)| = \oP(np_n^{1/(3-\tau)})$, $\rE(\sCi(p_n)) = \oP(np_n^{1/(3-\tau)})$.
\end{theorem}

\begin{remark}[Relation to Abel-Tauberian theorem] \normalfont The infinite second-moment assumption is captured by \eqref{eq:asymp-laplace}.
The identity \eqref{eq:asymp-laplace} is basically a version of the celebrated Abel-Tauberian theorem \cite[Chapter XIII.5]{Fel91} (see also \cite[Chapter 1.7]{BGT89}). 
However, since both $D_n^{\star}$ and $p_n$ depend on $n$, the joint asymptotics needs to be stated as an assumption.
In Section~\ref{sec:discussion}, we discuss how this assumption is satisfied when (i) $d_i = (1-F)^{-1}(i/n)$ (ii) $d_i$ is the $i$-th order statistic of an i.i.d sample, where $F$ is a power-law  distribution with $\tau\in (2,3)$. 
\end{remark}

\subsection{Discussion}
\label{sec:discussion}

\paragraph*{Critical window: emergence of hub connectivity.}
The critical window is the regime in which  hubs start getting connected. 
Hubs are the high-degree vertices, whose asymptotic degree is determined by Assumption~\ref{assumption1}(i).
To understand the above remark more precisely, let us denote the probability that $i$ and  $j$ are in the same component in the $p$-percolated graph by $\pi(i,j,p)$. 
Then, for any fixed $i, j\geq 1$,
\begin{gather}
\limsup_{n\to\infty} \pi(i,j,p_n)  = 0 \quad \text{ for }p_n \ll p_c,\label{eq:remark-sub}\\
0< \liminf_{n\to\infty} \pi (i,j,p_n)\leq  \limsup_{n\to\infty} \pi (i,j,p_n) <1  \quad \text{ for }p_n = \Theta( p_c),\label{eq:remark-crit}\\
\limsup_{n\to\infty} \pi(i,j,p_n)  = 1 \quad \text{ for }p_n \gg p_c, \label{eq:remark-super}
\end{gather}
Indeed, any two vertices $i$ and $j$ share $p_n d_id_j/(\ell_n-1)$ edges in expectation.
This expectation is $o(1)$, $\Theta(1)$, or $\omega (1)$ depending on whether $p_n\ll p_c$, $p_n\sim p_c$, or $p_n\gg p_c$.
In the subcritical regime, this observation and a simple union bound yields \eqref{eq:remark-sub}. 
For the critical case, a method of moment computation shows that the number of edges between hubs $i$ and $j$ converges in distribution to Poisson$(\lambda \theta_i \theta_j/\mu)$. 
We don't prove this here, but instead refer the reader to \cite[Proposition 7.13]{RGCN1} where similar Poisson approximation computations have been done for the configuration model. 
This shows \eqref{eq:remark-crit}. 
In the super-critical regime, 
\begin{eq}
\PR((i,j) \text{ don't share any edge}) = \prod_{l=1}^{d_j} \bigg(1-\frac{p_nd_i}{\ell_n-2l+1}\bigg) \leq \e^{-p_n d_i d_j/2\ell_n} \to 0,
\end{eq}
 so that $1- \pi(i,j,p_n) \to 0$ which yields \eqref{eq:remark-super}. 
Intuitively, in the barely subcritical regime, all the hubs are in different components.
Hubs start getting connected to each other directly, forming the critical components as the $p$ varies over the critical window. 
Finally in the barely super-critical regime the giant component, which contains all the hubs, is formed. 
The features \eqref{eq:remark-sub}, \eqref{eq:remark-crit} and \eqref{eq:remark-super} are also observed in the $\tau \in (3,4)$ case~\cite{BHL12}. 
However, the key distinction between $\tau \in (3,4)$ and $\tau\in (2,3)$ is that for $\tau \in (3,4)$ the paths between the hubs have lengths that grow as $n^{(\tau-3)/(\tau-1)}$, whereas they are directly connected in the $\tau\in(2,3)$ regime.

\paragraph*{Intuitive explanation for the exploration process. }
Suppose that we explore the critically percolated configuration model sequentially in a breadth-first manner.
The reflected version of the stochastic process in \eqref{defn::limiting::process} turns out to be the limit of the process that counts the number of unpaired half-edges incident to the discovered vertices. 
This limiting process can be intuitively understood as follows. 
When we explore hubs, the exploration process increases drastically, causing the jumps in the first term in \eqref{defn::limiting::process}. 
The negative linear drift is an accumulation of two effects. 
(1) Because we explore two vertices at each time, we get a negative drift $-2t$. 
(2) The exploration of the low-degree vertices cumulatively causes a linear positive drift $+t$. 
The main contribution in the latter case comes due to the degree-one vertices in the system.
Thus in total, we get a drift of $-t$ in the exploration process \eqref{defn::limiting::process}.

\paragraph*{Assumption on the degrees.} 
Assumptions~\ref{assumption1}, \ref{assumption-supercrit} hold for two interesting special cases of power-law degrees that have received special attention in the literature: Case~(I)  $d_i = (1-F)^{-1}(i/n)$, Case~(II) $d_i$'s are the order statistics of an i.i.d sample from $F$.
Here $F$ is some distribution function supported on non-negative integers and $(1-F)(x) = c_{\sss F} k^{-(\tau-1)},$ for  $k \leq x< k+1$, and we recall that the inverse of a bounded non-increasing function $f:\R\mapsto \R $ is defined as  
\begin{eq}
f^{-1}(x) := \inf\{y:f(y)\leq x\}.
\end{eq}
We add a dummy half-edge to vertex 1 if necessary to make $\sum_{i\in [n]}d_i$ even. However, we ignore this contribution since this does not change any asymptotic calculation below. 
Recall that we use $C$ as a generic notation for a constant whose value can be different between expressions, and $a_n \sim b_n$ denotes  $\lim_{n\to\infty}a_n/b_n = 1$.

For Case (I), 
$d_i \sim (c_{\sss F}n/i)^\alpha $ for all $i = o(n)$ and $d_i \leq C(n/i)^\alpha$ for all $i\in [n]$.
Consequently, Assumption~\ref{assumption1}(i) is satisfied with $\theta_i = c_{\sss F}^\alpha i^{-\alpha}$. 
To see Assumption~\ref{assumption1}(ii), note that
\begin{eq}\label{D-exp-conv-case-1}
\frac{1}{n} \sum_{i\in [n]} d_i \sim \int_0^1 (1-F)^{-1}(x) \dif x = \E[D],
\end{eq}where $D$ has distribution function $F$, 
and 
\begin{eq}\label{eq:second-moment}
n^{-2\alpha} \sum_{i>K} d_i^2 \leq C \sum_{i>K} i^{-2\alpha} \sim CK^{1-2\alpha} \to 0 \quad \text{ as } K\to\infty.
\end{eq}
Also, $D_n\dto D$, and $\E[D_n]\to \E[D]$ implies that $(D_n)_{n\geq 1}$ is uniformly integrable. 
To see Assumption~\ref{assumption-supercrit}, 
with the above computations, we have already verified all the conditions in Assumption~\ref{assumption-supercrit}(i),(ii).
To verify Assumption~\ref{assumption-supercrit}(iii), we now show that, for $t_n = tp_n^{1/(3-\tau)}$ with fixed $t>0$, 
\begin{eq} \label{eq:tauberian-to-prove}
&1 - \E[\e^{-t_nD_n^{\star}}] = \frac{1}{\ell_n} \sum_{k\in [n]} d_k \big(1-\e^{-t_nd_k}\big) \sim t_n^{\tau - 2} \int_0^\infty c_{\sss F} z^{-\alpha} (1-\e^{-c_{\sss F}z^{-\alpha}}) \dif z,
\end{eq}
and thus \eqref{eq:asymp-laplace} holds as well. 
Let us split the last sum in three parts by restricting to the set $\{k: d_k <\varepsilon(t_n)^{-1}\}$, $\{k: d_k \in [\varepsilon(t_n)^{-1}, (\varepsilon t_n)^{-1}]\}$, and $\{k: d_k >(\varepsilon t_n)^{-1}\}$ and denote them by $(I)$, $(II)$ and $(III)$ respectively.
Using the fact that $1-\e^{-x} \leq x$, it follows that
\begin{eq} \label{eq:tauberian-I}
\frac{(I)}{t_n^{\tau-2}} &\leq \frac{t_n^{3-\tau}}{\ell_n} \sum_{k: d_k < \varepsilon(t_n)^{-1}} d_k^2 \sim Cn^{2\alpha - 1}t_n^{3-\tau}  \sum_{k\geq Cn (t_n/\varepsilon)^{1/\alpha}}k^{-2\alpha} \\
& \sim C n^{2\alpha - 1}t_n^{3-\tau} \int_{Cn (t_n/\varepsilon)^{\tau-1}}^\infty x^{-2\alpha} \dif x \sim  C n^{2\alpha - 1}t_n^{3-\tau} (Cn (t_n/\varepsilon)^{\tau-1})^{1-2\alpha} \sim C\varepsilon^{3-\tau},
\end{eq}
and 
\begin{gather}\label{eq:tauberian-II}
\frac{(III)}{t_n^{\tau-2}} \leq  \frac{C}{t_n^{(\tau-2)}\ell_n}
\sum_{k: d_k > (\varepsilon t_n)^{-1}} d_k \leq \frac{C n^{\alpha-1}}{t_n^{\tau-2}}
\int_{1}^{Cn (t_n \varepsilon)^{\tau-1}} \frac{\dif x}{x^{\alpha}} \sim C \varepsilon^{\tau-2}.
\end{gather}
Also, we compute $(II)$ by 
\begin{eq}
\frac{(II)}{t_n^{\tau-2}} & = \frac{1}{t_n^{\tau-2} \ell_n } \sum_{k: d_k \in [\varepsilon t_n^{-1}, (\varepsilon t_n)^{-1}]} d_k (1-\e^{-t_n d_k}) \\
&\sim \frac{n^{\alpha - 1}}{\mu t_n^{\tau-2}} \sum_{k\in [c_0 n (t_n \varepsilon)^{\tau-1}, c_1 (t_n/\varepsilon)^{\tau -1}]} c_{\sss F} k^{-\alpha} \big(1-\e^{-t_n (c_{\sss F}n/k)^{\alpha}}\big)\\
& = \frac{1}{n t_n^{\tau -1}} \sum_{z\in [c_0\varepsilon^{\tau -1}, c_1/\varepsilon^{\tau-1}]} c_{\sss F} z^{-\alpha} (1-\e^{-c_{\sss F}z^{-\alpha}}),
\end{eq}where we have put $k = n t_n^{\tau -1} z$, so that the  $z$ values increase by $1/(n t_n^{\tau -1})$ in the final sum. 
Thus, in the iterated limit $\lim_{\varepsilon \to 0} \limsup_{n\to\infty}$,  
\begin{eq}
\frac{(II)}{t_n^{\tau-2}} \to \int_0^\infty c_{\sss F} z^{-\alpha} (1-\e^{-c_{\sss F}z^{-\alpha}}) \dif z = \kappa,
\end{eq}
which yields \eqref{eq:asymp-laplace} by combining it with \eqref{eq:tauberian-I} and \eqref{eq:tauberian-II}.

Let us now consider Case (II), i.e., the i.i.d degree setup. 
We have assumed that the degree sequence is ordered in a non-decreasing manner, 
i.e., $d_i$ is the $i$-th order statistic of the i.i.d samples. 
We use the following construction from \cite[Section 13.6]{B68}. 
   Let $(E_1,E_2,\dots )$ be an i.i.d sequence of unit-rate exponential random variables and let $\Gamma_i:= \sum_{j=1}^iE_j$. Let
  \begin{equation}\label{eq:iid-degee-coupling}
   \bar{d}_i=(1-F)^{-1}(\Gamma_i/\Gamma_{n+1}).
  \end{equation} 
Then $(d_1,\dots,d_n) \stackrel{d}{=} (\bar{d}_1,\dots,\bar{d}_n)$.  
Now, $\Gamma_i$'s follow a Gamma distribution with shape parameter $n$ and scale parameter 1. 
Note that, by the stong law of large numbers, $\Gamma_{n+1}/n \asto 1$. 
Thus,  for each fixed $i\geq 1$, $\Gamma_{n+1}/(n\Gamma_i)\asto 1/\Gamma_i$. 
Using \eqref{eq:iid-degee-coupling}, we see that $\boldsymbol{d}$ satisfies Assumption~\ref{assumption1}(i) almost surely with $\theta_i=(C_F/\Gamma_i)^{\alpha}$. 
To see that $(\theta_i)_{i\geq 1} \in \ell^2_{\shortarrow} \setminus \ell^1_{\shortarrow}$, observe that $\Gamma_i/i\asto 1$, and $\alpha\in (1/2,1)$.
Next, the first condition in Assumption~\ref{assumption1}(ii) follows from the strong law of large numbers.  
To see the second condition, we note that $\sum_i \Gamma_i^{-2\alpha} <\infty$ almost surely. Now using the fact that $\Gamma_{n+1}/n\asto 1$, we can use arguments identical to \eqref{eq:second-moment} to show that $\lim_{K\to\infty}\limsup_{n\to\infty} n^{-2\alpha}\sum_{i>K}d_i^2=0$ on the event $\{\sum_{i=1}^{\infty}\Gamma_i^{-2\alpha}<\infty\}\cap \{\Gamma_{n+1}/n\to 1\}$. Thus, we have shown that the third condition of Assumption~\ref{assumption1}\eqref{assumption1-2} holds almost surely. 
The verification of Assumption~\ref{assumption-supercrit} is also identical to Case-(I) if we do the computations conditionally on the Gamma random variables and use the above asymptotics.

\paragraph{Extension to the Norros-Reittu model.}
A related model where one would expect the same behavior as the configuration model is the multigraph version of the Norros–Reittu model or the Poisson graph process~\cite{NR06}.
Given a weight sequence $(w_i)_{i\in [n]}$, the Norros-Reittu multigraph is the multipgraph generated by putting $\mathrm{Poisson} (w_iw_j/L_n)$ many edges between vertices $i$ and $j$, where $L_n = \sum_{i\in [n]} w_i$.
If Assumptions~\ref{assumption1},~\ref{assumption-supercrit} holds with $(d_i)_{i\in [n]}$ replaced by $(w_i)_{i\in [n]}$, then we expect the same results for percolation on the Norros-Reittu multigraph about the critical and near critical regimes as described above. 
We do not pursue the Norros-Reittu multigraph here.

\paragraph*{Open Problems.} We next state some open problems: \\ 

\noindent{\bf Open Problem 1.} Theorem~\ref{thm:main} studies convergence of $\mathbf{Z}_n(\lambda)$ for each fixed $\lambda$. 
It will be interesting to study the distribution of $(\mathbf{Z}_n(\lambda))_{\lambda>0}$ as a stochastic process, when the percolated graphs are coupled through the Harris coupling. 
In the $\tau>4$ and $\tau\in (3,4)$ regimes, such evolution of critical components is described by the so-called augmented multiplicative coalescent process.
However, we do not expect the limit to be the augmented multiplicative coalescent here. 
This is clear from the fact that the scaling limit in \eqref{defn::limiting::process} is not related to the general characterization of exploration processes that arise in relation to multiplicative coalescent  in \cite{AL98}. 
Heuristically, one would expect that if $\sum_{i\in \sC}d_i\1\{i\text{ is hub}\}$ denotes the mass of a component, then the components would merge at rate proportional to their masses, but additionally, there are immigrating vertices of degree-one that keep on increasing the component sizes as well. 
The description of the process, and proving its Feller properties and entrance boundary conditions, are interesting open challenges. \\

\noindent{\bf Open Problem 2.}
Is it possible to prove that the metric structure of components converges in a suitable topology? 
This question is motivated by a strong notion of structural convergence of critical components that was first established in \cite{ABG09} ($\tau >4$) and \cite{BHS15} ($\tau\in (3,4)$).  
Since the components have small distances, 
it may be natural to consider the local-weak convergence framework.
However, the hubs within components have unbounded degrees, which is not covered directly in the local-weak convergence framework.

\section{Properties of the excursions of the limiting process}\label{sec:properties-exploration}
In this section, we prove some good properties of the process \eqref{defn::limiting::process} that allows us to conclude the convergence of largest excursion lengths from the stochastic process convergence. 
In Section~\ref{sec:cnt-det-funct}, we identify these good properties for functions in $\mathbb{D}_+[0,\infty)$ that ensure continuity of the largest excursion map. Then, we prove in Section~\ref{sec:limit-good} that $\biS$ satisfies these good properties almost surely. 

\subsection{Continuity of the largest excursion map} \label{sec:cnt-det-funct}
Recall the definitions of excursions from \eqref{defn:excursion}. Also, recall from Section~\ref{sec:notation} that $\ubar{f}(t) = \inf_{u\leq t} f(u)$ and $\mathscr{Z}_f = \{t: f(t) = \ubar{f}(t)\}$.
Define the set of excursions of $f$ as
\begin{equation}
\mathcal{E}_f:= \{(l,r): (l,r) \text{ is an excursion of }f\}.
\end{equation} 
We denote  the set of excursion begin-points (or left-points) and end-points (or right-points) by $\cL_f$ and $\cR_f$ respectively, i.e.,
\begin{equation}
\cL_f:= \{l\geq 0: (l,r)\in \mathcal{E}_f \text{ for some }r\}\quad \text{and}\quad \cR_f:= \{r\geq 0: (l,r)\in \mathcal{E}_f\text{ for some }l\}.
\end{equation}
We will use the following elementary fact: 
\begin{fact}\label{fact:cont-r}
Let $f\in \mathbb{D}_+[0,\infty)$. Then, for all $r\in \cR_f\setminus \{\infty\}$, $f$ is continuous at $r$. Consequently, $r\in \mathscr{Z}_f$. 
\end{fact}
\begin{proof}
Using the right-continuity of $f$, it suffices to show that $f(r) = f(r-)$. 
Suppose that is not the case. 
Since $f$ has positive jumps only, we must have that $f(r-) < f(r)$. 
Since $r$ is an excursion ending point, there exists $\varepsilon>0$ such that $f(t) - \ubar{f}(t) > 0 $ for all $t\in (r-\varepsilon,r)$. 
On the other hand, using the right-continuity of $f$ and the fact that $f(r) >f(r-)$, we obtain that $f(t) - \ubar{f}(t) > 0 $ for all $t\in [r,r+\varepsilon)$ for some  $\varepsilon>0$. Thus, there exists a sufficiently small $\varepsilon>0$ such that $f(t) - \ubar{f}(t) > 0 $ for all $t\in (r-\varepsilon,r+\varepsilon)$. 
This contradicts the fact that $r \in \mathrm{cl}(\mathscr{Z}_f)\setminus \{\infty\}$. 
\end{proof}
For $f\in \mathbb{D}_+[0,\infty)$, let $\phi_i(f)$ be the length of the $i$-th largest excursion of $f$. 
Also, let $\mathcal{A}_i(f)$ denote the area under $i$-th largest excursion of $f$.  
We will show that if $f_n\to f$ in $\mathbb{D}[0,\infty)$ then $\phi_i$ and $\mathcal{A}_i$ converge when the limiting function has some good properties. 
Let us start by describing these \emph{good} properties: 
\begin{defn}[Good functions]\label{defn::good_function}\normalfont A  function $f\in \mathbb{D}_+[0,\infty)$ is said to be \emph{good} if the following holds:
\begin{enumerate}[(a)]
\item For all $r\in \cR_f\setminus \{\infty\}$, $r$ is not a local minimum of $f$.

\item There does not exist any interval $(q_1,q_2)$ with $q_1,q_2\in \Q_+$ such that $(q_1,q_2) \subset \mathscr{Z}_f$.

\item For all $(l,r) \in \cE_f$ with $r<\infty$, there exists $\varepsilon_0=\varepsilon_0(l,r)>0$ such that the following holds for all $\varepsilon \in (0,\varepsilon_0)$: 
There exists $\delta = \delta(\varepsilon,l,r)>0$ such that 
 \begin{equation}\label{f-large-interval}
 f(t)> f(r)+\delta\quad \forall t\in (l+\varepsilon,r-\varepsilon).
 \end{equation}
\item $f$ does not have any infinite excursion, i.e., $\phi_1(f)<\infty$.

\item For any $\delta>0$, $f$ has only finitely many excursions of length at least~$\delta$.

\item For all $i\geq 1$, $\phi_{i+1} (f)< \phi_i(f)$. 
\end{enumerate} 
\end{defn}

 \begin{lemma}\label{lem::good:function:continuity} Suppose that $f\in \mathbb{D}_+[0,\infty)$ is good. 
 Further, let $(f_n)_{n\geq 1} \subset \mathbb{D}[0,\infty)$ be such that $f_n\to f$ in $\mathbb{D}[0,\infty)$. 
 Moreover, let $\limsup_{n\to\infty} \phi_1(f_n) < \infty$, and if $z_n(T)$ denotes the length of the largest excursion of $f_n$ starting after $T$, then $\lim_{T\to\infty}\limsup_{n\to\infty} z_n(T) = 0$. 
 Then, for all $m\geq 1$, as $n\to\infty$,
 \begin{eq}
 (\phi_i(f_n))_{i\in [m]} \to(\phi_i(f))_{i\in [m]}, \quad \text{and} \quad  (\mathcal{A}_i(f_n))_{i\in [m]} \to(\mathcal{A}_i(f))_{i\in [m]}.
 \end{eq}
 \end{lemma}
\begin{proof}
 The proof  here is for $m=1$, and for $m>1$, we can proceed inductively. 
 Using Definitions~\ref{defn::good_function}(d),(e), we can take $T>0$ and $n_0\geq 1$ large so that 
 the largest excursions of $f_n$ and $f$ end before $T$ for all $n\geq n_0$. 
 Let $\mathfrak{L}$ denote the set of continuous functions $\Lambda:[0,\infty)\to [0,\infty)$  that are strictly increasing and satisfy $\Lambda(0)=0, \Lambda(T)=T$.  
 Suppose $(l,r)$ is the longest excursion of $f$ on $[0,T]$, and thus $\phi_1(f)=r-l$. 
 We will first show that $\lim_{n\to\infty}\phi_1(f_n)=\phi_1(f)$.

 Fix $\varepsilon,\delta>0$ such that \eqref{f-large-interval} holds. 
 Let $||\cdot||_{\sss T}$ denote the sup-norm on $[0,T]$.  
 Recall the definition of the metric for Skorohod $J_1$-topology from \cite[(12.13)]{Bil99}.
 Since $f_n\to f$ in $\mathbb{D}[0,T]$, there exists $(\Lambda_n)_{n\geq 1} \subset \mathfrak{L}$, and $n_1\geq n_0$ such that for all $n\geq n_1$, 
 \begin{equation}\label{f_n-f-close}
 ||f_n\circ \Lambda_n-f||_{\sss T}< \frac{\delta}{2}\quad  \text{and}\quad  ||\Lambda_n-I||_{\sss T}< \varepsilon,
 \end{equation}where $I$ is the identity function. 
Using \eqref{f-large-interval} and \eqref{f_n-f-close}, for all $t\in (l+\varepsilon,r-\varepsilon)$  and $n\geq n_1$,
 \begin{eq}
 f_n\circ \Lambda_n (t) > f(t) - \frac{\delta}{2} > f(r) + \frac{\delta}{2} = \ubar{f}(r) + \frac{\delta}{2}, 
 \end{eq}where the last equality is due to $r\in \mathscr{Z}_f$ from Fact~\ref{fact:cont-r}.
 Thus, using $||\Lambda_n-I||_{\sss T}< \varepsilon$ from \eqref{f_n-f-close}, 
 \begin{equation} \label{liminf_f_excursion}
  f_n(t)> \ubar{f}(r)+\frac{\delta}{2}\quad \forall t \in (l+2\varepsilon, r-2\varepsilon).
 \end{equation}
Next, note that the infimum operation is continuous in the Skorohod $J_1$-topology \cite[Theorem 13.4.1]{W02}, and thus $\ubar{f}_n \to \ubar{f}$ in $\mathbb{D}[0,T]$. Moreover, using \eqref{continuity-ubar-f},  $\ubar{f} \in \mathbb{C}[0,T]$, and therefore, there exists $n_2\geq n_0$,  such that for all $n\geq n_2$ 
\begin{eq}\label{inf-conv}
\|\ubar{f}_n - \ubar{f}\|_{\sss T} < \frac{\delta}{4}.
\end{eq}
Using $\ubar{f}(t) = \ubar{f}(r)$ for all $t\in [l,r]$, this implies that, for all $n\geq n_2$,
\begin{equation} \label{liminf_f_excursion-2}
  \ubar{f}(r) = \ubar{f}(t)> \ubar{f}_n(t)-\frac{\delta}{4}\quad \forall t \in (l+2\varepsilon, r-2\varepsilon),
 \end{equation}
and consequently \eqref{liminf_f_excursion} yields that for all $n\geq \max\{n_1,n_2\}$
\begin{eq}\label{eq:excursion-f-n}
 f_n(t)-\ubar{f}_n(t) > \frac{\delta}{4}\quad \forall t \in (l+2\varepsilon, r-2\varepsilon).
\end{eq}
 Thus,
\begin{equation}\label{eq:liminf-exc}
\liminf_{n\to\infty}\phi_1(f_n)\geq r-l-4\varepsilon= \phi_1(f)-4\varepsilon,
\end{equation}which provides the required lower bound. 
We now turn to a suitable upper bound on the quantity  $\limsup_{n\to\infty}\phi_1(f_n)$. 
We claim that, using Definition~\ref{defn::good_function}(b), one can find $r_1,\dots,r_k\in \cR_f$ such that $r_1\leq \phi_1(f)+\varepsilon, T-r_k< \phi_1(f)+\varepsilon, $ and $r_i-r_{i-1}\leq \phi_1(f)+\varepsilon, \forall i=2, \dots,k$. 
Indeed, since $\phi_1(f)$ is the largest excursion length of $f$, if there is no excursion end-point in between 0 and $\phi_1(f)+\varepsilon$, then there is no excursion begin-point in $[0,\varepsilon)$. The latter shows that the interval $[0,\varepsilon)$ is contained in $\mathscr{Z}_f=\{t: f(t) = \ubar{f}(t)\}$, which contradicts Definition~\ref{defn::good_function}(b). The existence of the points $r_2,\dots,r_k$ can be shown inductively using similar argument as above. Let $l_1,\dots,l_k$ be the excursion begin-points corresponding to the endpoints $r_1,\dots,r_k$.  
We will show that, for all $i$, $f_n$ will have an  excursion within $(l_i-4\varepsilon, r_i+2\varepsilon)\subset (r_{i-1}-4\varepsilon, r_i+2\varepsilon)$, so that the largest excursion of $f_n$ is contained inside one of these intervals.

Using Definition \ref{defn::good_function}(a), $r_i$ is not a local minimum, and thus for any $\varepsilon >0$ (sufficiently small), there exists $\delta >0$ and $t_i\in (r_i,r_i+\varepsilon)$ such that $f(r_i)-f(t_i)> \delta$. 
 We also let $\delta>0$ be sufficiently small such that \eqref{f_n-f-close} holds. 
 Thus, using \eqref{f_n-f-close}, for all $n\geq n_1$,
 \begin{equation}
  f(r_i)-f_n(\Lambda_n(t_i))\geq f(r_i)-f(t_i) -\frac{\delta}{2} > \frac{\delta}{2}.
 \end{equation}
 Since $t_i\in (r_i,r_i+\varepsilon)$, we have that $t_i^n = \Lambda_n(t_i)\in (r_i-\varepsilon,r_i+2\varepsilon)$. 
 Thus, for all $n\geq n_1$, there exists a point $t_i^n\in (r_i - \varepsilon,r_i+2\varepsilon)$ such that
 \begin{equation}\label{exc-id-1}
  f(r_i)-f_n(t_i^n)> \frac{\delta}{2}.
 \end{equation}
 Next, using  \eqref{inf-conv}, 
 \begin{eq}\label{exc-id-2}
 \ubar{f}_n(r_i-3\varepsilon)\to \ubar{f}(r_i-3\varepsilon) \geq \ubar{f}(r_i)= f(r_i),
 \end{eq}since $r_i\in \mathscr{Z}_f$.
Combining \eqref{exc-id-1} and \eqref{exc-id-2}, we see that $\ubar{f}_n(r_i-3\varepsilon) > \ubar{f}_n(t_i^n)$, and by \eqref{eq:excursion-f-n}, we also have that $f_n(r_i-3\varepsilon) - \ubar{f}_n(r_i-3\varepsilon)>0$. 
Thus  $f_n$ must have an excursion end-point in $(r_i-3\varepsilon,r_i+2\varepsilon)$. 
 Also, using Definition~\ref{defn::good_function}(b), $f$ has an excursion end-point  $r^0_i\in(l_i-\varepsilon,l_i)$. The previous argument shows that $f_n$ has to have an excursion end-point in $(r^0_i-3\varepsilon, r^0_i+2\varepsilon)$ and thus in $(l_i-4\varepsilon,l_i+2\varepsilon)$. Therefore, 
 \begin{equation}\label{eq:limsup-exc}
 \limsup_{n\to\infty} \phi_1(f_n)\leq \max_{i\in [k]} (r_i-l_i) + 6\varepsilon \leq \max_{i\in [k]} (r_i-r_{i-1}) + 6\varepsilon \leq \phi_1(f)+7\varepsilon.
\end{equation}  Hence, the convergence of the largest excursion length follows from  \eqref{eq:liminf-exc} and \eqref{eq:limsup-exc}.

Next, we show that $\lim_{n\to\infty} \mathcal{A}_1(f_n) = \mathcal{A}_1(f)$. 
Let $e= (l,r)$ be the largest excursion of~$f$.
Using \eqref{eq:excursion-f-n}, the interval $(l-2\varepsilon,r+2\varepsilon)$ is part of some excursion of $f_n$. Let us denote this excursion by $e_n = (L_n(e),R_n(e))$. We will show that $e_n$ is the largest excursion of $f_n$ when $n$ is large. Indeed, the arguments above already show that 
\begin{eq}\label{sandwich-excursion}
l-4\varepsilon\leq L_n (e)\leq l+2\varepsilon, \quad \text{and} \quad r-2\varepsilon \leq R_n(e) \leq r+2\varepsilon,
\end{eq}
and thus $R_n(e) - L_n(e) \geq r-l-4\varepsilon$. 
Now, using Definition~\ref{defn::good_function}(f), we can take $\varepsilon>0$ sufficiently small such that $\phi_2(f_n) < r-l -4\varepsilon$ for all sufficiently large $n$. 
Thus, $e_n = (L_n(e), R_n(e))$ must be the largest excursion of $f_n$. 
 The convergence of $\mathcal{A}_1(f_n)$ follows by using $f_n\to f$ in $\mathbb{D}[0,\infty)$ together with $L_n(e) \to l$ and $R_n(e) \to r$ as $n\to\infty$. 
\end{proof}
\begin{remark} \normalfont
We emphasize that the strict ordering between excursion lengths in  Definition~\ref{defn::good_function}(f) is only used in the convergence of $\mathcal{A}_i(f_n)$. This ensures that the location of largest excursions of $f_n$ and $f$ approximately coincide, which is strictly stronger than requiring the convergence of excursion lengths. 
\end{remark}

Next, we define what it means for a stochastic process $\mathbf{X} \in \mathbb{D}_+[0,\infty)$ to be good:
\begin{defn}[Good stochastic process]\label{defn:good_function-infty}\normalfont A stochastic process $\mathbf{X}$ with sample paths in  $\mathbb{D}_+[0,\infty)$ is said to be \emph{good} if the sample path satisfies all the conditions of Definition~\ref{defn::good_function} almost surely. 
\end{defn}
The following is a direct consequence of Lemma~\ref{lem::good:function:continuity}: 
\begin{proposition}\label{prop:conv-exc-length-area}
Consider a sequence of stochastic processes $(\mathbf{X}_n)_{n\geq 1}$ and a good stochastic process $\mathbf{X}$ such that $\mathbf{X}_n \xrightarrow{\sss d} \mathbf{X}$. Also, let $(\phi_1(\mathbf{X}_n))_{n\geq 1}$ be tight, and if $Z_n(T)$ denotes the length of the largest excursion of $\mathbf{X}_n$ starting after time $T$, then for any $\varepsilon>0$,  $\lim_{T\to\infty}\limsup_{n\to\infty} \PR(Z_n(T) >\varepsilon)= 0$. 
Then, for all $m\geq 1$,
\begin{eq}
\big(\phi_i(\mathbf{X}_n), \mathcal{A}_i(\mathbf{X}_n)\big)_{i\in[m]}\dto \big(\phi_i(\mathbf{X}), \mathcal{A}_i(\mathbf{X})\big)_{i\in[m]}.
\end{eq}
\end{proposition}

\subsection{The limiting process is good almost surely} \label{sec:limit-good}
In this section, we will show that the sample paths of $\biS$ are good almost surely.  
Throughout this section, we assume without loss of generality  that $\mu=1$ and $\sum_i \theta_i^2 = 1$ to simplify writing. 
An identical proof works for the general $\mu$ and $\bld{\theta}$ by replacing $\lambda$ with $\lambda'=\lambda \mu/\sum_{i}\theta_i^2$. 
Consider the sigma-field $\mathscr{F}_t = \sigma(\{\xi_i\leq s\}:s\leq t, i\geq 1)$, where $(\xi_i)_{i\geq 1}$ are the exploration random variables used in the definition of $\biS$ in \eqref{defn::limiting::process}, and, for a collection of sets $\cA$, $\sigma (\cA)$ denotes the minimum sigma-algebra containing all the sets in $\cA$. 
Then $(\mathscr{F}_t)_{t\geq 0}$ is a filtration and $\biS$ is adapted to $(\mathscr{F}_t)_{t\geq 0}$. 
Our goal is stated formally in the following proposition: 
\begin{proposition} \label{prop:limit-properties}
The sample paths of $\biS$ satisfy the conditions of {\rm Definition~\ref{defn::good_function}} almost surely.
\end{proposition}

\noindent \emph{Proof of Proposition~\ref{prop:limit-properties}.} The verification of each of the conditions in Definition~\ref{defn::good_function} are given separately below. 

\paragraph*{Verification of Definition~\ref{defn::good_function}(a).}
Let $q\in \Q_+$ and define the random time $T_q = \inf\{t\geq q: \iS(t) = \inf_{u\leq q} \iS(u)\}$. 
We will show that, almost surely,
\begin{eq}\label{eq:non-perfect-levy}
\inf \{t>0: \iS(T_q+t) - \iS(T_q)<0\}=0, \quad \text{ on }\{T_q<\infty\}, \text{ for all } q\in \Q_+.
\end{eq}
Note that if $q$ lies in some finite-length excursion then $T_q<\infty$, and also $T_q$ is the end-point of that excursion. Now, \eqref{eq:non-perfect-levy} ensures that $T_q$ is not a local minimum because we can find $u$ arbitrarily close to $T_q$ such that $\iS (u) < \iS(T_q)$. Hence,
Definition~\ref{defn::good_function}(a) holds for $\biS$ almost surely.

Thus it suffices to prove \eqref{eq:non-perfect-levy}. 
Since $\Q_+$ is countable, it is enough to prove   \eqref{eq:non-perfect-levy} for each fixed $q\in \Q_+$.
Let $V_q = \{i: \cI_i (T_q) = 1\}$.
Note that $T_q$ is a stopping time. Moreover, conditionally on the sigma-field $\mathscr{F}_{T_q}$, the process $(\iS(T_q+t) - \iS(T_q))_{t\geq 0}$ is distributed as $\hat{\mathbf{S}}_{\infty}^\lambda$ given by  
\begin{eq}\label{S-hat-defn}
\hat{S}_{\infty}^\lambda(t) =  \lambda\sum_{i\notin V_q} \theta_i\mathcal{I}_i(t)-  t. 
\end{eq}
Define
$L(t) = \lambda\sum_{i=1}^{\infty} \theta_i \mathcal{N}_i(t) -t,$
where $(\mathcal{N}_i(t))_{t\geq 0}$ is a rate-$\theta_i$ Poisson process, independently for different $i$. 
We assume that $\hat{\mathbf{S}}_{\infty}^\lambda$ and $\bld{L}$ are coupled by taking $\cI_i(s) = 
\mathbbm{1}\{\mathcal{N}_i(s)\geq 1\}$, so that $\hat{S}_{\infty}^{\lambda}(t)\leq L(t)$ for all $t\geq 0$ almost surely. 
Thus, if $R_0 = \inf\{t> 0: L(t)<0\}$, then
it suffices to show that
\begin{equation}\label{defn:perfect-levy-1}
  \PR(R_0 =0) = 1, 
 \end{equation}and \eqref{eq:non-perfect-levy} follows.
 Fix $\varepsilon>0$ and $K\geq 1$. 
Then, 
\begin{eq}\label{computation-levy-regular}
\PR(R_0 \leq \varepsilon)  &\geq  \PR(L(\varepsilon)<0 ) \\
&\geq \PR\bigg(\lambda\sum_{i=K+1}^{\infty} \theta_i \mathcal{N}_i(\varepsilon) < \varepsilon, \text{ and } \mathcal{N}_i(\varepsilon) = 0, \ \forall i\in [K] \bigg)\\
& =\prod_{i=1}^K\PR(\mathcal{N}_i(\varepsilon) = 0) \times \PR\bigg(\lambda\sum_{i=K+1}^{\infty} \theta_i \mathcal{N}_i(\varepsilon) < \varepsilon \bigg) 
\\
& =\e^{-\varepsilon \sum_{i = 1}^K\theta_i } \bigg(1- \PR\bigg(\lambda\sum_{i=K+1}^{\infty} \theta_i \mathcal{N}_i(\varepsilon) \geq  \varepsilon \bigg) \bigg)
\\
&\geq  \e^{-\varepsilon \sum_{i = 1}^K\theta_i }\bigg(1-\frac{\lambda}{\varepsilon} \E\bigg[\sum_{i=K+1}^{\infty} \theta_i \mathcal{N}_i(\varepsilon)\bigg]\bigg) =  \e^{-\varepsilon \sum_{i = 1}^K\theta_i }\bigg(1-\lambda \sum_{i=K+1}^{\infty} \theta_i^2 \bigg),
\end{eq}where the one-but-last step follows from Markov's inequality. 
Thus, using the fact that $\{R_0 \leq \varepsilon\}\searrow \{R_0 = 0\}$, as $\varepsilon\searrow 0$, 
\begin{eq}
\PR(R_0 =0) = \lim_{\varepsilon\searrow 0}\PR(R_0 \leq \varepsilon) \geq 1-\lambda \sum_{i=K+1}^{\infty} \theta_i^2,
\end{eq}and since the above holds for any $K\geq 1$, and $\sum_{i}\theta_i^2<\infty$, we have proved \eqref{defn:perfect-levy-1}.
\paragraph*{Verification of Definition~\ref{defn::good_function}(b).}
Next, we verify that Definition~\ref{defn::good_function}(b) holds almost surely for~$\biS$. 
Since $\Q_+$ is countable, we may again work with fixed $q_1,q_2\in\Q_+$, i.e., it suffices to prove that $(q_1,q_2) \not\subset \{t: \iS(t) = \inf_{u\leq t} \iS(u)\}$ almost surely. 
By the description of our thinned L\'evy process, it has positive jumps only, and if there is a jump of size $\theta_i$ at time $t$, then $\iS(t+\theta_i/2) > \inf_{u\leq t} \iS(u)= \inf_{u\leq t+\theta_i/2} \iS(u)$. Therefore, if $(q_1,q_2) \subset \{t: \iS(t) = \inf_{u\leq t} \iS(u)\}$, then there is no $\xi_i$ such that $\xi_i\in(q_1,q_2)$.
We compute  
\begin{eq}
\PR(\forall i\geq 1: \xi_i \notin (q_1,q_2)) &= \prod_{i=1}^\infty \PR(\xi_i \notin (q_1,q_2))=\prod_{i=1}^\infty (1- \e^{-\theta_i q_1} + \e^{-\theta_i q_2}) \\
& = \exp\bigg(\sum_{i=1}^\infty \log \Big(1- \e^{-\theta_i q_1} (1- \e^{-\theta_i (q_2-q_1)}\Big) \bigg) \\
&\leq \exp\bigg( - \e^{-\theta_1q_1}\sum_{i=1}^\infty(1- \e^{-\theta_i (q_2-q_1)}) \bigg) = 0,
\end{eq}
where the one-but-last step follows using $\log (1-x) \leq -x$ for all $x\in (0,1)$ and $\e^{-\theta_iq_1}\geq \e^{-\theta_1q_1}$ for all $i\geq 1$, and the last step uses the fact that $\sum_{i=1}^\infty(1- \e^{-\theta_i (q_2-q_1)}) = \infty$, which follows by applying the limit comparison test together with $(1- \e^{-\theta_i (q_2-q_1)})/\theta_i \to q_2-q_1$ as $i\to\infty$, and $\sum_{i=1}^\infty \theta_i = \infty$.
Thus we have verified that Definition~\ref{defn::good_function}(b) holds almost surely for $\biS$.

\paragraph*{Verification of Definition~\ref{defn::good_function}(c).}
Similarly as above, for any $q\in\Q_+$, define the stopping time $T_q(\varepsilon) = \inf\{t\geq q: \iS(t) \leq  \inf_{u\leq q} \iS(u)+ \varepsilon\}$.
Let $\mathcal{C}_q$ denote the event that $q$ lies in some finite-length excursion.
Observe that $\mathcal{C}_q$ implies $T_q(\varepsilon)<\infty$.
We claim that it is sufficient to prove
\begin{eq}\label{no-approximate-end}
\lim_{\varepsilon\searrow 0} \PR(\biS \text{ has an excursion end-point in }(T_q(\varepsilon),T_q(\varepsilon)+2\varepsilon), \text{ and }\mathcal{C}_q \text{ occurs}) =1. 
\end{eq}
Let $T_q^-:=\inf\{t> q: \iS(t-) =  \inf_{u\leq q} \iS(u)\}$. 
Indeed, if $\mathcal{C}_q$ occurs, then $T_q(\varepsilon) \nearrow T_q^-$ as $\varepsilon \searrow 0$, and  \eqref{no-approximate-end} shows that $T_q^{-}$ must be an excursion end-point with probability 1. 
Thus, none of the excursions of $\biS$ contain a point $t$ such that $\iS(t-) = \inf_{u\leq q} \iS(u) = \inf_{u\leq t} \iS(u)$, where we have used the fact that $\inf_{u\leq t} \iS(u)$ is constant on an excursion interval. 
This completes the verification of Definition~\ref{defn::good_function}(c). 

It remains to prove \eqref{no-approximate-end}. As before, let $L(t) = \lambda\sum_{i=1}^{\infty} \theta_i \mathcal{N}_i(t) -t$, and let us also work under the coupling under which $\iS (T_q(\varepsilon) +t) - \iS (T_q(\varepsilon)) \leq L(t)$ for all $t\geq 0$ almost surely. 
On the event $\mathcal{C}_q$, we have $\iS(T_q(\varepsilon)) \leq \inf_{u\leq q}\iS(u) + \varepsilon$, since the process has only positive jumps. 
Also, on $\mathcal{C}_q$, if $L(2\varepsilon) < \varepsilon$, then $\iS (T_q(\varepsilon) +2\varepsilon) - \iS (T_q(\varepsilon)) < \varepsilon$, and consequently the event in \eqref{no-approximate-end} holds. 
Thus, using identical computations as \eqref{computation-levy-regular}, it follows that 
\begin{eq}
&\PR(\biS \text{ has an excursion end-point in }(T_q(\varepsilon),T_q(\varepsilon)+2\varepsilon), \text{ and } \mathcal{C}_q\text{ occurs }) \\
&\geq \PR(L(2\varepsilon) < \varepsilon) \geq \e^{-2\varepsilon \sum_{i = 1}^K\theta_i }\bigg(1-\frac{3\lambda}{2} \sum_{i=K+1}^{\infty} \theta_i^2 \bigg),
\end{eq}
and \eqref{no-approximate-end} follows by taking the iterated limit  $\lim_{K\to\infty}\lim_{\varepsilon\to 0}$, and using $\sum_i\theta_i^2<\infty$.

\paragraph*{Verification of Definition~\ref{defn::good_function}(d).}
We start by providing the martingale decomposition for $\biS$:
\begin{lemma}\label{lem:mart-decomp-Sinfty} The process $\biS$ admits the Doob-Meyer decomposition $\iS(t) = M(t)+A(t)$ with the drift term $A(t)$ and the quadratic variation for the martingale term $\langle M\rangle (t)$  given by 
\begin{equation}
A(t) = \lambda\sum_{i=1}^\infty \theta_i^2\min\{\xi_i,t\}-t, \qquad \langle M\rangle (t) = \lambda^2\sum_{i=1}^\infty \theta_i^3\min\{\xi_i,t\}.
\end{equation}
\end{lemma}
\begin{proof}
Define $M_i(t) = \ind{\xi_i\leq t}-\theta_i\min\{\xi_i,t\}$. Then 
 \begin{eq}\label{eq:martingale-statement}
 (M_i(t))_{t\geq 0} \quad \text{is a martingale.}
 \end{eq}
Indeed, note that $M_{i}(t+s) - M_i(t) = 0$ if $\xi_i \leq t$. 
Thus, 
\begin{eq}\label{eq:martingale-comp}
\E[M_i(t+s) - M_i(t) \mid \mathscr{F}_t] &= \E[\ind{t<\xi_i\leq t+s}-\theta_i(\min\{\xi_i,t+s\} - \min\{\xi_i,t\}) \mid \xi_i>t] \\
&=\E[\ind{t<\xi_i\leq t+s}-\theta_i \min\{\xi_i-t,s\} \mid \xi_i>t] \\
& = \PR(0<\xi_i \leq s) - \theta_i\E[\min\{\xi_i,s\}],
\end{eq}
where the last step follows from the memoryless property of the exponential distributions. Now, using the fact that $\int x\e^{-ax} \dif x= -\e^{-ax}(ax+1)/a^2$, one can verify that $\theta_i\E[\min\{\xi_i,s\}] = 1- \e^{-\theta_i s}$.
Applying this to \eqref{eq:martingale-comp}, we can conclude that $\E[M_i(t+s) - M_i(t) \vert \mathscr{F}_t] =0$, thus verifying~\eqref{eq:martingale-statement}.
Moreover, the quadratic variation of $(M_i(t))_{t\geq 0}$  is given by
\begin{equation} \label{eq:QV-statement}
\langle M_i\rangle(t) = \theta_i\min\{\xi_i,t\}.
\end{equation}
This follows from the characterization of unit-jump  processes given in \cite[Lemma 3.1]{PTW07}, together with the fact that $\theta_i\min\{\xi_i,t\}$, the compensator of $\ind{\xi_i\leq t}$, is continuous in $t$.
Then \eqref{eq:martingale-statement} and \eqref{eq:QV-statement} completes the proof of Lemma~\ref{lem:mart-decomp-Sinfty}.
\end{proof} 
We are now ready to verify  Definition~\ref{defn::good_function}(d). 
In order to prove that $\biS$ does not have an  excursion of infinite length almost surely, it suffices to show that 
\begin{eq}\label{s-t-large-infty}
\lim_{t\to\infty}\iS(t) = -\infty \quad \text{almost surely. }
\end{eq}
Fix $K\geq 1$ such that $\lambda \sum_{i>K} \theta_i^2 <1/2$.
Such a choice of $K$ is always possible as $\bld{\theta}\in \ell^2_{\shortarrow}$.
Further define the stopping time $T:= \inf\{t: \xi_i\leq t,\ \forall i\in [K]\} = \max_{i\leq K} \xi_i$.
Thus, $T<\infty$ almost surely.
Note that $\min\{\xi_i,t\} \leq t$ and thus, 
\begin{eq}
\frac{1}{t}\lambda\sum_{i>K}\theta_i^2 \min\{\xi_i,t\} < \frac{1}{2}, \quad \text{almost surely.}
\end{eq}
Therefore, for any $t>T$,
\begin{eq}
A(t) = \lambda \sum_{i\in [K]} \theta_i^2 \xi_i  + \lambda\sum_{i>K}\theta_i^2 \min\{\xi_i,t\} -t < \lambda \sum_{i\in [K]} \theta_i^2 \xi_i -\frac{t}{2}, \quad \text{almost surely.}
\end{eq}
We conclude that, for any $r\in (0,1)$,
$t^{-r}A(t) \asto -\infty.$
For the martingale part we will use the exponential concentration inequality \cite[Inequality 1, Page 899]{SW86}, which is stated below:
\begin{lemma}\label{lem:concentration-inequality}
If $M$ is any continuous time local martingale such that $M(0) = 0$, and $\sup_{t\in [0,\infty)} |M(t) - M(t-)| \leq c$, almost surely, then for any $t>0$, $a>0$ and $b>0$, 
\begin{eq}
\PR\Big(\sup_{s\in [0,t]} M(s) > a, \text{ and } \langle M \rangle (t) \leq b  \Big) \leq \exp \bigg(-\frac{a^2}{2b} \psi \Big( \frac{a c}{b}\Big) \bigg), 
\end{eq}where $\psi(x) = ( (1+x)\log(1+x) - x ) / x^2$. 
\end{lemma}
\noindent In particular, $\psi(x) \geq 1/(2 (1+x/3))$ (see \cite[Page 27]{JLR00}).
Note that $ \langle M\rangle(t)\leq \lambda^2 t \sum_{i=1}^\infty \theta_i^3.$
We apply Lemma~\ref{lem:concentration-inequality} with $a = \varepsilon t^{r}$, $b = \lambda^2 t \sum_{i=1}^\infty \theta_i^3$, and $c= \theta_1$. 
Using Lemma~\ref{lem:mart-decomp-Sinfty}, $\langle M\rangle (t) \leq b$ almost surely.
Now, $\psi (ac/b) \geq C/(1+t^{r-1})$, and thus for any $\varepsilon >0$, and $r\in (1/2,1)$
\begin{eq}
\PR\Big(\sup_{s\in [0,t]} |M(s)| > \varepsilon t^{r}  \Big) \leq 2 \exp (-C t^{2r-1}), 
\end{eq}for some constant $C>0$, where the bound on the absolute value of $M$ follows from the fact that $-M$ is also a martingale, so Lemma~\ref{lem:concentration-inequality} applies to $-M$ as well. 
Now an application of the Borel-Cantelli lemma proves that $ t^{-r}| M(t) |\asto 0,$ for any $r\in (1/2,1)$.
This fact, together with the asymptotics of the drift term, completes the proof of \eqref{s-t-large-infty}. \qed

\paragraph*{Verification of Definition~\ref{defn::good_function}(e).}
Fix $\delta >0$.
 Let $t_k = (k-1)\delta/2$ and define the event
 \begin{equation}
  \rC_k^\delta:=\bigg\{\sup_{t\in (t_{k-1},t_k]}\iS(t_{k+1}) - \iS(t)>0\bigg\}.
 \end{equation}
Suppose that there is an excursion $(l,r)$ with $r-l > \delta$ and $l\in (t_{k-1},t_k]$ for some $k$. 
Since $r> t_{k+1}$ and $l\in (t_{k-1},t_k]$, we have that  $\inf_{u\leq t_{k+1}}\iS(u) = \inf_{u\in (t_{k-1},t_k]}\iS(u) $.
Consequently, 
$\iS(t_{k+1})> \inf_{t\in (t_{k-1},t_k]} \iS(t)$, and therefore $\rC_k^\delta$ must occur.
Therefore, if $\biS$ has infinitely many excursions of length at least $\delta$, then $\rC_k^\delta$ must occur infinitely often.
Using the Borel-Cantelli lemma, the proof follows if we can show that
\begin{equation}
\sum_{k=1}^\infty\PR(\rC_k^\delta) <\infty.
\end{equation}
As before, fix $K\geq 1$ such that $\lambda \sum_{i>K} \theta_i^2 <1/2$, and let $T:= \inf\{t: \xi_i\leq t,\ \forall i\in [K]\} = \max_{i\leq k} \xi_i$.
Notice that for each $K\geq 1$,
\begin{eq}
\sum_{k=1}^\infty \prob{T>t_{k-1}} &= \sum_{k=1}^\infty \prob{\exists i\in [K]: \xi_i>t_{k-1}}\leq \sum_{k=1}^\infty K \e^{-\theta_K(k-1)\delta/2}<\infty,
\end{eq} and therefore it is enough to show that 
\begin{equation}\label{suff-finite-exc}
\sum_{k=1}^\infty\PR(\rC_k^\delta\cap \{T\leq t_{k-1}\}) <\infty.
\end{equation}
Now,
\begin{eq}
&\sup_{t\in [t_{k-1},t_k]}\big[\iS(t_{k+1})-\iS(t)\big] \leq M(t_{k+1})+\sup_{t\in [t_{k-1},t_k]} - M(t) + \sup_{t\in [t_{k-1},t_k]} [A(t_{k+1}) - A(t)]\\
& \hspace{1cm}\leq M(t_{k+1}) - M(t_{k-1})+ \sup_{t\in [t_{k-1},t_k]} [M(t_{k-1})- M(t)] \\
&\hspace{2cm}+ \sup_{t\in [t_{k-1},t_k]}\bigg[\lambda\sum_{i=1}^\infty \theta_i^2 (\min\{\xi_i,t_{k+1}\}-\min\{\xi_i,t\}) -(t_{k+1}-t)\bigg]\\
& \hspace{1cm}\leq 2\sup_{t\in [t_{k-1},t_{k+1}]}  |M(t)-M(t_{k-1})| \\
& \hspace{2cm}+ \sup_{t\in [t_{k-1},t_k]}\bigg[\lambda\sum_{i=1}^\infty \theta_i^2 (\min\{\xi_i,t_{k+1}\}-\min\{\xi_i,t\}) -(t_{k+1}-t)\bigg].
\end{eq}
On the event $\{T\leq t_{k-1}\}$, the second term inside the supremum above reduces to 
\begin{eq}\label{drift-bound}
\lambda\sum_{i>K} \theta_i^2 (\min\{\xi_i,t_{k+1}\}-\min\{\xi_i,t\}) -(t_{k+1}-t) \leq (t_{k+1}-t)\lambda\sum_{i>K} \theta_i^2 - (t_{k+1}-t)< -\frac{\delta}{2},
\end{eq}using $\lambda \sum_{i>K} \theta_i^2 <1/2$.
Thus we only need to estimate
\begin{equation}
\PR\bigg(\sup_{t\in [t_{k-1},t_{k+1}]}  |M(t)-M(t_{k-1})| >\frac{\delta}{4}\bigg).
\end{equation}
Note that $(M(t)-M(t_{k-1}))_{t\geq t_{k-1}}$ is a martingale with respect to the filtration $(\mathscr{F}_t)_{t\geq t_{k-1}}$ 
starting from zero. Moreover, using an identical argument as Lemma~\ref{lem:mart-decomp-Sinfty} yields that the quadratic variation of $(M(t)-M(t_{k-1}))_{t\geq t_{k-1}}$ is given by 
\begin{equation}
\lambda^2\sum_{i=1}^\infty \theta_i^3\big(\min\{\xi_i,t\}-\min\{\xi_i,t_{k-1}\}\big).
\end{equation}
Further, $\E[\min\{\xi_i,t\}] = \theta_i^{-1}(1-\e^{-\theta_i t})$. 
Therefore, Doob's martingale inequality \cite[Theorem~1.9.1.3]{LS89} implies
\begin{eq}
&\sum_{k=1}^\infty\PR\bigg(\sup_{t\in [t_{k-1},t_{k+1}]}  |M(t)-M(t_{k-1})| >\frac{\delta}{4}\bigg)\\
&\leq \sum_{k=1}^\infty\frac{16 \lambda^2}{\delta^2}\sum_{i=1}^\infty \theta_i^2 (\e^{-\theta_i t_{k-1}}-\e^{-\theta_i t_{k+1}})
= \frac{16\lambda^2}{\delta^2} \sum_{i=1}^\infty \theta_i^2(1-\e^{-\theta_i\delta}) \sum_{k=1}^\infty \e^{-\theta_i t_{k-1}}<\infty,
\end{eq}
and the proof of \eqref{suff-finite-exc} now follows using \eqref{drift-bound}.

\paragraph*{Verification of Definition~\ref{defn::good_function}(f).}
We first prove the following: 
\begin{lemma}\label{lem:exc} 
The distribution of $\iS(t)$ has no atoms for all $t>0$.
\end{lemma}
\begin{proof}
Let $\phi_t (v) = \E[\e^{\im v S(t)}]$ for $v\in\R$. Using the sufficient condition for random variables to have non-atomic distribution stated in \cite[Page 189]{GS01}, it suffices to prove that 
\begin{eq}
\int_{-\infty}^\infty |\phi_t (v)| \dif v<\infty. 
\end{eq}
Note that  
\begin{eq}
\phi_t(v) &= \e^{- \im v t}\prod_{j=1}^{\infty}  \E[\e^{\im v \lambda \theta_j\1\{\xi_j\leq t\}}] = \e^{- \im v  t}\prod_{j=1}^{\infty} (\e^{\im v \lambda \theta_j}(1-\e^{- \theta_jt}) + \e^{- \theta_jt})\\
 &=\e^{- \im v t}\prod_{j=1}^{\infty}\big( (1-\e^{-t\theta_j }) \cos (v \lambda \theta_j )  + \e^{-t\theta_j }+ \im (1-\e^{-t\theta_j})\sin(v\lambda \theta_j)\big).
\end{eq}
Therefore, 
\begin{eq}
|\phi_t (v)|^2 
&= \prod_{j=1}^\infty \Big(\big((1-\e^{-t\theta_j }) \cos (v \lambda \theta_j )  + \e^{-t\theta_j }\big)^2+ (1-\e^{-t\theta_j})^2 \sin^2(v\lambda \theta_j)\Big) \\
& = \prod_{j=1}^\infty \Big(\e^{-2t\theta_j} + 2\cos(v\lambda \theta_j) \e^{-t\theta_j}(1-\e^{-t\theta_j}) + (1-\e^{-t\theta_j})^2 \Big) \\ 
& = \prod_{j=1}^\infty \Big(1- 2 \e^{-t\theta_j}(1- \e^{-t\theta_j}) (1-\cos(v\lambda \theta_j))\Big)\\
&\leq \e^{-\sum_{j=1}^\infty 2 \e^{-t\theta_j}(1- \e^{-t\theta_j}) (1-\cos(v\lambda \theta_j))},
\end{eq}where in the last step we have used the fact that $1-x\leq \e^{-x}$ for all $x>0$.
Recall \eqref{density-assumption}.
Let $j_0(v,t)\geq 1$ be such that $ \max \{|v|\theta_j,t\theta_j \} \leq 1$ for all $j\geq j_0 (v,t)$.
Now, for $j\geq j_0(v,t)$, we have that $\e^{-t\theta_j} \geq \e^{-1}$, $(1-\e^{-t\theta_j}) \geq t\theta_j/2$ and $1-\cos(v\lambda \theta_j) \geq \frac{2\lambda^2}{\pi} v^2\theta_j^2$.
Thus, using \eqref{density-assumption}, 
\begin{eq}\label{eq:ub-char-funct}
\int_{-\infty}^\infty |\phi_t (v)| \dif v \leq \int_{-\infty}^\infty  \e^{-\frac{2\lambda^2t}{\e\pi} v^2M_t(v)} \dif v<\infty,
\end{eq}
and the proof now follows.
\end{proof}

In order to prove the strict ordering between excursion lengths, it is enough to show that no two excursions of $\biS$ have the same length almost surely. 
For any $q\in \Q_+$, let $e(q)$ be the excursion containing $q$. 
Thus it is enough to show that for any $q_1,q_2 \in \Q_+$,
\begin{eq}
\PR(e(q_1) \neq e(q_2), \text{ but }|e(q_1)| = |e(q_2)|) =0.
\end{eq}
Without loss of generality, let $q_1<q_2$. Thus, if $e(q_1)\neq e(q_2)$, then  $e(q_1) $ appears earlier than $e(q_2)$. 
Let $V_{q_2} = \{i: \mathcal{I}_i (q_2) = 1\}$. As before, conditionally on  $\mathscr{F}_{q_2}$, the process $(\iS(q_2+t) - \iS(q_2))_{t\geq 0}$ is distributed as $\hat{\mathbf{S}}_{\infty}^\lambda$ given by  
\begin{eq}\label{eq:conditional-process}
\hat{S}_{\infty}^\lambda(t) =  \sum_{i\notin V_{q_2}} \theta_i\left(\mathcal{I}_i(t)- (\theta_i/\mu)t\right)+\lambda t.
\end{eq}
Therefore, the process in \eqref{eq:conditional-process} again has the form \eqref{defn::limiting::process} (see \eqref{S-hat-defn}). 
Now, for any $x>0$, the probability that $|e(q_2)| = x$, conditionally on $\mathscr{F}_{q_2}$ and $|e(q_1)| = x$, is zero using Lemma~\ref{lem:exc} together with the fact that $|V_{q_2}^c| = \infty$. 
This concludes the verification of Definition~\ref{defn::good_function}(f).

\section{The critical window} \label{sec:proofs}
In this section, we prove our results related to critical percolation on $\CM$. 
In Section~\ref{sec:sandwich-CM}, we start by describing a way to approximate  percolation on a configuration model by a suitable alternative configuration model.
In Section~\ref{sec:CM-critical-window}, we analyze the latter graph. 
The first step is to set up an exploration process that approximately encodes the component sizes  in terms of excursion lengths above past minima. This exploration process is shown to converge to $\biS$ (Section~\ref{sec:expl-process}). 
We must also ensure that the exploration process does not have large excursions appearing beyond the time scale of the exploration process, which allows us to prove that the largest component sizes converge to largest excursion lengths of $\biS$ (Section~\ref{sec:large-early}). 
Next we analyze the surplus edges (Section~\ref{sec:surplus-fd}) and the proof of Theorem~\ref{thm:main} is completed in Section~\ref{sec:conv-surp-comp-size}. 
Finally, we analyze the diameter of the critical components in Section~\ref{sec:diameter} and complete the proof of Theorem~\ref{thm:diameter-large-comp}.

\subsection{Sandwiching the percolated configuration model}\label{sec:sandwich-CM}
Following the pioneering work of Aldous~\cite{A97}, the main tool to prove scaling limits of the component sizes is to set up an appropriate exploration process. 
The idea is to explore the graph sequentially, and the exploration process keeps track of some functional of vertices that have been discovered but their neighborhoods have not been explored. For percolation on the configuration model, this could be the number of unpaired half-edges of those vertices. 
Now, for random graphs with independent connection probabilities, the exploration process is usually Markovian, but not for the configuration model.
Indeed, one has to keep track of the degree-profile outside the explored graph in order to know the distribution of the degree of a newly discovered vertex. 
For $d$-regular graphs, Nachmias and Peres \cite{NP10b} used the above approach, but this becomes difficult in the unbounded degree case. 
In earlier papers with Sen \cite{DHLS15,DHLS16}, we have used a construction by Janson~\cite{J09} which says that the percolated configuration model can be viewed as a configuration model satisfying some criticality condition, so that it is enough to analyze the behavior of these critical configuration models. 
However, in the $\tau\in (2,3)$ regime, this construction does not work because it gives rise to $n - o(n)$ many degree-one vertices. 
As a remedy to this problem, we use a result of Fountoulakis~\cite{F07} to show that the critical configuration model can be sandwiched between two approximately equal configuration models,  as stated in Proposition~\ref{prop:coupling-lemma} below. 
We emphasize that Proposition~\ref{prop:coupling-lemma} holds for percolation on the configuration model without any specific assumption on the degree distribution, as long as $\ell_n p_n \gg \log(n)$, and this will be used in the proofs for the near-critical results as well.
We start by describing the approximating configuration model below:
\begin{algo}\label{algo-alt-cons-perc}
\begin{itemize}\normalfont
 \item[(S0)] Keep each half-edge with probability $p_n$, independently, and delete the half-edges otherwise. If the total number of retained half-edges is odd, then attach a \emph{dummy} half-edge to vertex 1.
 \item[(S1)] Perform a uniform perfect matching among the \emph{retained half-edges}, i.e., within the retained half-edges, pair unpaired half-edges sequentially with a uniformly chosen unpaired half-edge until all half-edges are paired. 
 The paired half-edges create edges in the graph, and we call the resulting graph $\mathcal{G}_n(p_n)$.
\end{itemize}
\end{algo}
\noindent The following proposition formally states that $\cG_n(p_n)$ approximates $\rCM_n(\bld{d},p_n)$:
\begin{proposition} \label{prop:coupling-lemma}
Let $p_n$ be such that $\ell_np_n \gg \log(n)$. 
There exists $(\varepsilon_n)_{n\geq 1} \subset (0,\infty)$ with $\varepsilon_n\to 0 $, and a coupling such that, with high probability,
\begin{equation}
\mathcal{G}_n(p_n(1-\varepsilon_n))\subset\mathrm{CM}_n(\bld{d},p_n)\subset\mathcal{G}_n(p_n(1+\varepsilon_n)).
\end{equation}
\end{proposition}
\begin{proof} 
The proof relies on an exact construction of $\rCM_n(\bld{d},p_n)$ by Fountoulakis~\cite{F07} which goes as follows:
\begin{algo}\label{algo-fountoulakis}\normalfont
\begin{itemize}
\item[(S0)] Perform a binomial trial $X\sim \mathrm{Bin}(\ell_n/2,p_n)$ and choose $2X$ half-edges uniformly at random from the set of all half-edges.
\item[(S1)] Perform a perfect matching of these $2X$ chosen half-edges. The resulting graph is distributed as $\rCM_n(\bld{d},p_n)$.
\end{itemize}
\end{algo}
Notice the similarity between Algorithm~\ref{algo-alt-cons-perc}~(S1) and Algorithm~\ref{algo-fountoulakis}~(S1). 
In both algorithms, given the number of retained  half-edges, the choice of the half-edges can be performed sequentially uniformly at random without replacement. 
Thus, given the number of half-edges in the two algorithms, we can couple the choice of the half-edges, and their pairing (the restriction of a uniform matching to a subset of half-edge remains uniform matching on that subset).
Let $\mathcal{H}_1$, $\mathcal{H}_2^-$ and $\mathcal{H}_2^+$,  respectively, denote the number of half-edges in $\mathrm{CM}_n(\bld{d},p_n)$, $\mathcal{G}_n(p_n(1-\varepsilon_n))$ and $\mathcal{G}_n(p_n(1+\varepsilon_n))$.
From the above discussion, the proof is complete if we can show that, as $n\to\infty$,
\begin{equation}\label{eq:coup-reduc}
 \PR\big(\mathcal{H}_2^-\leq \mathcal{H}_1\leq \mathcal{H}_2^+\big) \to  1.
\end{equation} 
We ignore the contribution due to the possible addition of  \emph{only one} dummy edge in Algorithm~\ref{algo-expl}~(S0), as it does not affect asymptotic computations.
Notice that $\mathcal{H}_1=2X$, where $X\sim\mathrm{Bin}(\ell_n/2,p_n)$, and $\mathcal{H}_2^{+/-}\sim\mathrm{Bin} (\ell_n,p_n(1\pm\varepsilon_n))$.
Using standard concentration inequalities \cite[Corollary 2.3]{JLR00}, it follows that
\begin{subequations}
\begin{equation}
\mathcal{H}_1 = \ell_np_n +\oP(\sqrt{\ell_np_n\log(n)}), 
\end{equation}and
\begin{equation}
\mathcal{H}_2^+ = \ell_np_n + \ell_np_n \varepsilon_n +\oP(\sqrt{\ell_np_n\log(n)}).
\end{equation}
\end{subequations}
If we choose $\varepsilon_n$ such that $\varepsilon_n\gg (\log(n)/(\ell_np_n))^{1/2}$ and $\varepsilon_n\to 0$, then, with high probability, $\mathcal{H}_1\leq \mathcal{H}_2^+$.
Similarly we can conclude that $\mathcal{H}_2^-\leq \mathcal{H}_1$ with high probability, and the proof of Proposition~\ref{prop:coupling-lemma} follows. 
\end{proof}

We conclude this section by stating some properties of the degree sequence of the graph $\cG_n(p_n)$ that will be crucial in the analysis below. 
Let $\Mtilde{\bld{d}}=(\tilde{d}_1,\dots,\tilde{d}_n)$ be the degree sequence induced by Algorithm~\ref{algo-alt-cons-perc}~(S1), and let $\tilde{\ell}_n = \sum_{i} \tilde{d}_i$ be the number of retained half-edges.
Then the following result holds for $\Mtilde{\bld{d}}$: 
\begin{lemma}[Degrees of $\mathcal{G}_n(p_n)$]\label{lem:perc-degrees}
Suppose that $p_n \gg n^{-\alpha}$, and {\rm Assumption~\ref{assumption1}} holds.
For each fixed $i\geq 1$, $\tilde{d}_i = d_ip_n(1+\oP(1))$,  $\tilde{\ell}_n = \ell_np_n (1+\oP(1))$, and $\sum_{i\in [n]} \tilde{d}_i(\tilde{d}_i-1) = p_n^2\sum_{i\in [n]} d_i(d_i-1) (1+\oP(1))$. 
Consequently, for $p_n\ll p_c(\lambda)$, $\sum_{i\in [n]}\tilde{d}_i^2= \tilde{\ell}_n(1+\oP(1))$, whereas for $p_n= p_c(\lambda)$, \begin{equation}\label{eq:perc-tail-sum-square}
\tilde{\nu}_n = \frac{\sum_{i\in [n]}\tilde{d}_i(\tilde{d}_i-1)}{\sum_{i\in [n]}\tilde{d}_i} = \lambda (1+\oP(1)), \quad and \quad   
\lim_{K\to\infty}\limsup_{n\to\infty} \PR\bigg(\sum_{i>K}\tilde{d}_i(\tilde{d}_i-1)>\varepsilon \tilde{\ell}_n\bigg) = 0,
\end{equation}for any $\varepsilon >0$.
\end{lemma} 
\begin{proof}
Note that $\Mtilde{d}_i \sim \mathrm{Bin}(d_i,p_n)$, independently for $i\in [n]$. 
For each fixed $i\geq 1$, $d_i p_n \to\infty$, as $p_n\gg n^{-\alpha}$. 
Thus the first fact follows using \cite[Theorem 2.1]{JLR00}. 
Since, $\tell_n\sim \mathrm{Bin}(\ell_n,p_n)$, the second fact also follows using the same bound.
To see the asymptotics for $\tilde{m}_2:=\sum_{i\in [n]} \tilde{d}_i(\tilde{d}_i-1)$, 
note that $\E[\tilde{m}_2] = p_n^2 m_2$, where $m_2 = \sum_{i\in [n]}d_i(d_i-1)$.
Also, 
$\mathrm{Var}(\tilde{d}_i(\tilde{d}_i-1)) = 2d_i(d_i-1)p_n^2(1-p_n) (1+(2d_i-3)p_n)$.
Thus, 
\begin{eq}
\frac{\mathrm{Var}\big(\sum_{i\in [n]}\tilde{d}_i(\tilde{d}_i-1)\big) }{\big(\E[\sum_{i\in [n]}\tilde{d}_i(\tilde{d}_i-1)]\big)^2} \leq \frac{4d_1 p_n^3 m_2}{p_n^4 m_2^2} = O\Big(\frac{1}{p_n n^\alpha}\Big) = o(1),
\end{eq}where the penultimate step uses the fact that $m_2 = \Theta(n^{2\alpha})$, $d_1 = \Theta(n^{\alpha})$, and in the last step we have again used the fact that $p_n \gg n^{-\alpha}$. 
Using Chebyshev's inequality, it now follows that $\tilde{m}_2 = p_n^2 m_2(1+
\oP(1))$.
Thus, 
\begin{eq}
\tilde{\nu}_n = (1+\oP(1))p_n \frac{\sum_{i\in [n]}d_i(d_i-1)}{\sum_{i\in [n]}d_i} = (1+\oP(1))p_n\nu_n. 
\end{eq}
For $p_n\ll p_c(\lambda)$, $p_n\nu_n =o(1)$. Thus, $\sum_{i\in [n]}\tilde{d}_i^2= \tilde{\ell}_n(1+\oP(1))$. 
For $p_n = p_c(\lambda)$, the first equality in \eqref{eq:perc-tail-sum-square} follows using \eqref{eq:crit-window-CM}.

We now prove the second inequality in \eqref{eq:perc-tail-sum-square}. 
 For any $\varepsilon>0$, the required probability is at most
\begin{eq}
&\PR\bigg(\sum_{i>K}\tilde{d}_i(\tilde{d}_i-1)>\varepsilon \tilde{\ell}_n, \frac{\ell_np_n}{2}\leq \tilde{\ell}_n\leq 2\ell_np_n \bigg) +o(1)\\
&\hspace{1cm}\leq \PR\bigg(\sum_{i>K}\tilde{d}_i(\tilde{d}_i-1)> \frac{\varepsilon\ell_np_n}{2}\bigg)+o(1)
\\
&\hspace{1cm}\leq \frac{4p_n^2\sum_{i>K}d_i(d_i-1)}{\varepsilon\ell_n p_n}+o(1) = \frac{4 p_n \sum_{i>K}d_i^2}{\varepsilon\ell_n} +o(1),
\end{eq}
where the penultimate step follows from Markov's inequality. 
The proof now follows using~\eqref{eq:unif-int-2nd-moment} and $p_n = \Theta(n^{1-2\alpha})$ for $p_n = p_c(\lambda)$. 
\end{proof}

\subsection{Scaling limits of critical components}\label{sec:CM-critical-window}
\subsubsection{Convergence of the exploration process}\label{sec:expl-process}
  Let $\Mtilde{\bld{d}}=(\tilde{d}_1,\dots,\tilde{d}_n)$ be the degree sequence induced by Algorithm~\ref{algo-alt-cons-perc}~(S1) with $p_n=p_c(\lambda)$, and consider $\mathcal{G}_n(p_c(\lambda))$. 
  Note that $\mathcal{G}_n(p_c(\lambda))$ has the same distribution as $\mathrm{CM}_{n}(\Mtilde{\boldsymbol{d}})$. 
  We start by describing how the connected components in the graph can be explored while generating the random graph simultaneously:
\begin{algo}[Exploring the graph]\label{algo-expl}\normalfont  The algorithm carries along vertices that can be alive, active, exploring and killed, and half-edges that can be alive, active or killed. 
Alive and killed half-edges correspond to unpaired and paired half-edges respectively, whereas active half-edges correspond to half-edges that have been found during the exploration, but have not been paired yet. 
Thus a half-edge can be alive and active simultaneously.
Similarly, a vertex is killed when all its half-edges have been explored, otherwise the vertex is alive. An active vertex is an alive vertex that has been found already during the exploration, whereas an exploring vertex is currently being explored. 
We sequentially explore the graph as follows:
\begin{itemize}
\item[(S0)] At stage $i=0$, all the vertices and the half-edges are \emph{alive} but none of them are \emph{active}. Also, there are no \emph{exploring} vertices. 
\item[(S1)]  At each stage $i$, if there is no active half-edge at stage $i$, choose a vertex $v$ proportional to its degree among the alive (not yet killed) vertices and declare all its half-edges to be \emph{active} and declare $v$ to be \emph{exploring}. 
Proceed to step $i+1$.
\item[(S2)] At each stage $i$, if the set of active half-edges is non-empty, then take an active half-edge $e$ of an exploring vertex $v$ and pair it with a  half-edge $f$ chosen uniformly among the alive half-edges. 
 Kill $e,f$. If $f$ is incident to a vertex $v'$ that has not been discovered before, then declare all the half-edges incident to $v'$ active (if any), except $f$. 
If $\mathrm{degree}(v')=1$ (i.e. the only half-edge incident to $v'$ is $f$) then kill $v'$. Otherwise, declare $v'$ to be active and \emph{larger} than all other vertices that are active. After killing $e$, if $v$ does not have another active half-edge, then kill $v$ also, and declare the \emph{smallest} vertex to be exploring.

\item[(S3)] Repeat from (S1) at stage $i+1$ if not all half-edges are already killed.
\end{itemize}
\end{algo}
Algorithm~\ref{algo-expl} gives a \emph{breadth-first} exploration of the connected components of $\mathrm{CM}_n(\Mtilde{\boldsymbol{d}})$. Define the exploration process by
   \begin{equation}\label{defn:exploration:process}
    S_n(0)=0,\quad
     S_n(l)=S_n(l-1)+\tilde{d}_{(l)}J_l-2,
    \end{equation} where $J_l$ is the indicator that a new vertex is discovered at time $l$ and $\tilde{d}_{(l)}$ is the degree of the new vertex chosen at time $l$ when $J_l=1$.  
    The $-2$ in \eqref{defn:exploration:process} takes into account the fact that two half-edges are killed whenever two half-edges are paired at some step. However, at the beginning of exploring a component when Algorithm~\ref{algo-expl}~(S1) is carried out, we do not pair half-edges but the exploration process subtracts $-2$ nonetheless. For this reason, there is an additional $-2$ in \eqref{defn:exploration:process} at the beginning of exploring each component, and thus  the first component is explored when the exploration process hits $-2$, the second component is explored when the process hits $-4$ and so on. More formally, suppose that $\mathscr{C}_{k}$ is the $k$-th connected component explored by the above exploration process and define
$\tau_{k}=\inf \big\{ i:S_{n}(i)=-2k \big\}$. 
Then  $\mathscr{C}_{k}$ is discovered between the times $\tau_{k-1}+1$ and $\tau_k$, and  $\tau_{k}-\tau_{k-1}-1$ gives the total number of edges in $\mathscr{C}_k$.
 Call a vertex \emph{discovered} if it is either active or killed. Let $\mathscr{V}_l$ denote the set of vertices discovered up to time $l$ and $\mathcal{I}_i^n(l):=\ind{i\in\mathscr{V}_l}$. 
 Note that 
   \begin{equation}\label{eq:expl-process-crit}
    S_n(l)= \sum_{i\in [n]} \tilde{d}_i \mathcal{I}_i^n(l)-2l.
   \end{equation} 
In the rest of this section, we often use the asymptotics in Lemma~\ref{lem:perc-degrees} even if it is not stated explicitly. 
Recall that we write $\mathscr{F}_l^n = \sigma (\mathcal{I}^n_i(l): i\in [n]) $.    
All the martingales and related computations will be done with respect to the filtration $(\mathscr{F}_l^n )_{l\geq 0}$.

  Define the re-scaled version $\bar{\mathbf{S}}_n$ of $\mathbf{S}_n$ by $\bar{S}_n(t)= n^{-\rho}S_n(\lfloor tn^{\rho} \rfloor)$. Then,
   \begin{equation} \label{eqn::scaled_process}
    \bar{S}_n(t)= n^{-\rho} \sum_{i\in [n]}(\tilde{d}_i-1) \mathcal{I}_i^n(tn^{\rho})+n^{-\rho} \sum_{i\in [n]}\mathcal{I}_i^n(tn^{\rho}) -2t + o(1),
   \end{equation}where we have used the convention that $\mathcal{I}_i^n(tn^{\rho}) = \mathcal{I}_i^n(\floor{tn^{\rho}})$ when $tn^{\rho}$ is not an integer.
   The following theorem describes the scaling limit of this rescaled process:
   \begin{theorem} \label{thm::convegence::exploration_process} Consider the process $\bar{\mathbf{S}}_n:= (\bar{S}_n(t))_{t\geq 0}$ defined in \eqref{eqn::scaled_process} and recall the definition of  $\biS$ in \eqref{defn::limiting::process}. Then, under {\rm Assumption~\ref{assumption1}}, as $n\to\infty,$
 \begin{equation}
  \bar{\mathbf{S}}_n \dto \biS
 \end{equation} with respect to the Skorohod $J_1$-topology.
\end{theorem}
To prove Theorem~\ref{thm::convegence::exploration_process}, we need to obtain asymptotics of the first two terms in \eqref{eqn::scaled_process}. 
The first term accounts for the contribution due to the non-degree-one vertices during the exploration.
The first term is dominated by the contributions due to hubs, which allows us to use a truncation argument. 
The convergence of the truncated sum is given by the following lemma:
\begin{lemma} \label{lem::convergence_indicators}
   Fix any $K\geq 1$, and $\mathcal{I}_i(s):=\ind{\xi_i\leq s }$ where $\xi_i\sim \mathrm{Exp}(\theta_i/\mu)$ independently for $i\in [K]$. 
   Under {\rm Assumption~\ref{assumption1}}, as $n\to\infty$,
   \begin{equation}
    \left( \mathcal{I}_i^n(tn^\rho) \right)_{i\in[K],t\geq 0} \dto \left( \mathcal{I}_i(t) \right)_{i\in[K],t\geq 0}
   \end{equation}with respect to the Skorohod $J_1$-topology.
 \end{lemma}
The second term in \eqref{eqn::scaled_process} describes the proportion of time when a new vertex is found. 
Since we see a new vertex of degree one in most steps of the exploration process, this term is shown to converge to the constant function $t$, which is proved using martingale arguments. 
This is summarized in the next lemma:
\begin{lemma}\label{lem:asymp-second-term} 
For any $u>0$, as $n\to\infty$, $\sup_{u\leq t}n^{-\rho} \big|\sum_{i\in [n]}\mathcal{I}_i^n(un^{\rho}) -un^{\rho}\big| \pto 0.$
\end{lemma}
We first prove Theorem~\ref{thm::convegence::exploration_process} using Lemmas~\ref{lem::convergence_indicators}~and~\ref{lem:asymp-second-term}. 
The lemmas will be proved subsequently.
Let $\tell_n(u)$ denote the number of unpaired half-edges at time $\lfloor un^{\rho}\rfloor$. 
Thus, $\tell_n(u) = \tell_n -2(\floor{un^{\rho}} - c_{\floor{un^{\rho}}})+ 1$, where $c_l$ is the number of components explored up to time $l$.
Note that $\tell_n- 2un^{\rho}+1\leq \tell_n(u) \leq \tell_n$.
Since $\tell_n = \Theta_{\sss \PR} (n^{2\rho})$, we have  $\tell_n(u) = \tell_n(1+\oP(1))$ uniformly over $u\leq t$.
 Let $\tilde{\PR}(\cdot)$ (respectively $\tilde{\E}[\cdot]$) denote the conditional probability (respectively expectation) conditionally on $(\td_i)_{i\in [n]}$.
\begin{proof}[Proof of Theorem~\ref{thm::convegence::exploration_process}]
 Note that, $\mathcal{I}_i^n(l) =0$ for all $l\geq 1$ if $\tilde{d}_i=0$. Now, if $\tilde{d}_i\geq 1$, then for any $t\geq 0$, uniformly over $l\leq tn^{\rho}$,
\begin{equation}\label{prob-ind-lb}
    \exptc{\mathcal{I}_i^n(l)}= \probc{\mathcal{I}_i^n(l)=1}\leq \frac{l\tilde{d}_i}{\tilde{\ell}_n - 2un^{\rho}+1}.
  \end{equation}
Let  $X_{n,K}(t):=n^{-\rho}\sup_{u\leq t}\sum_{i>K}(\tilde{d}_i-1)\mathcal{I}_i^n(un^{\rho})$. 
Note that $\mathcal{I}_i^n(un^{\rho}) \leq \mathcal{I}_i^n(tn^{\rho})$. 
Also, using $\mathcal{I}_i^n(un^{\rho}) = 0$ whenever $\tilde{d}_i=0$, it follows that $(\tilde{d}_i-1)\mathcal{I}_i^n(un^{\rho}) \geq 0$ for all $i\in [n]$ and $u>0$. Thus, 
\begin{eq}\label{eq:expt-ub-sup}
\tilde{\E}[X_{n,K}(t)] &\leq  n^{-\rho}\tilde{\E}\bigg[\sum_{i>K}(\tilde{d}_i-1)\mathcal{I}_i^n(tn^{\rho})\bigg]\leq t \frac{\sum_{i>K}\tilde{d}_i(\tilde{d}_i-1)}{\tilde{\ell}_n(t)} := \varepsilon_{n,K}(t),
\end{eq}where $\lim_{K\to\infty}\limsup_{n\to\infty}\PR(\varepsilon_{n,K}(t) >\delta) = 0$ for any $\delta>0$, due to Lemma~\ref{lem:perc-degrees}.
Therefore, for any $\varepsilon,\delta>0$, using Markov's inequality, 
\begin{eq}
\lim_{K\to\infty}\limsup_{n\to\infty}\PR \Big(\tilde{\PR}(X_{n,K}(t) > \varepsilon) > \delta \Big) \leq \lim_{K\to\infty}\limsup_{n\to\infty}\PR \Big(\tilde{\E}[X_{n,K}(t)]  > \delta \varepsilon \Big) =0.
\end{eq}
Let $\mathcal{B}_{n,K}:= \{\tilde{\PR}(X_{n,K}(t) > \varepsilon) > \delta  \}$. It follows that  
\begin{eq}
\PR(X_{n,K}(t) > \varepsilon) = \E \big[\tilde{\PR}(X_{n,K} (t) > \varepsilon) \big] \leq \PR(\mathcal{B}_{n,K}) + \delta .  
\end{eq}
Taking the iterated limit $\lim_{\delta \to 0}\limsup_{K\to\infty}\limsup_{n\to\infty} $ yields, for any $\varepsilon>0$,
\begin{eq}\label{tail-iterated-limit}
\lim_{K\to\infty}\limsup_{n\to\infty}\PR(X_{n,K}(t) > \varepsilon) = 0. 
\end{eq}
Using \eqref{tail-iterated-limit} and Lemma~\ref{lem:asymp-second-term}, it is now enough to deduce the scaling limit, as $n\to\infty$, for 
\begin{equation}\label{eq:truncated-process}
\bar{S}_n^K(t) = n^{-\rho}\sum_{i=1}^K \tilde{d}_i \mathcal{I}_i^n(tn^{\rho})- t
\end{equation} and then taking $K\to\infty$. 
But for any fixed $K\geq 1$, Lemma~\ref{lem::convergence_indicators} yields the limit of $S_n^K$, and the proof of Theorem~\ref{thm::convegence::exploration_process} follows.
\end{proof}

\begin{proof}[Proof of Lemma~\ref{lem::convergence_indicators}] 
By noting that $(\mathcal{I}_i^n(tn^\rho))_{t\geq 0}$ are indicator processes, for any $m_1\leq m_2\leq m_3$, it follows that  $\min\{\mathcal{I}_i^n(m_2)- \mathcal{I}_i^n(m_1), \mathcal{I}_i^n(m_3) - \mathcal{I}_i^n(m_2)\} = 0$, and thus  \cite[Theorem 13.5]{Bil99} implies tightness of  $(\mathcal{I}_i^n(tn^\rho) )_{t\geq 0, n\geq 1}$ for each fixed $i\geq 1$.
Thus, it is enough to show that 
\begin{equation}
 \probc{\mathcal{I}_i^n(t_in^{\rho})=0,\ \forall i\in [K]} \pto \probc{\mathcal{I}_i(t_i)=0,\ \forall i\in [K]} = \exp \Big( -\mu^{-1}\sum_{i=1}^{K} \theta_it_i\Big),
\end{equation} for any $t_1,\dots,t_K\in [0,\infty)$. Now, 
\begin{equation} \label{lem::eqn::expression1}
 \probc{\mathcal{I}_i^n(m_i)=0,\ \forall i\in [K]}=\prod_{l=1}^{\infty}\Big(1-\sum_{i\leq K:l\leq m_i}\frac{\td_i}{\tell_n-\Theta(l)} \Big).
\end{equation}Taking logarithms on both sides of \eqref{lem::eqn::expression1} and using the fact that $l\leq \max m_i=\Theta(n^{\rho})$ we get 
\begin{equation}\label{lem::eqn::ex1}
 \begin{split}
  \probc{\mathcal{I}_i^n(m_i)=0\, \forall i\in [K]}&= \exp\Big( - \sum_{l=1}^{\infty}\sum_{i\leq K:l\leq m_i} \frac{\td_i}{\tell_n}+o(1) \Big)= \exp\Big( -\sum_{i\in [K]} \frac{\td_im_i}{\tell_n} +o(1) \Big).
 \end{split}
\end{equation} Putting $m_i=t_in^{\rho}$, Assumption~\ref{assumption1}~\eqref{assumption1-1},~\eqref{assumption1-2} give
\begin{equation} \label{lem::eqn::expression2}
 \frac{m_i\td_i}{\tell_n}= \frac{\theta_it_i}{\mu} (1+\oP(1)).
\end{equation}
Hence \eqref{lem::eqn::ex1} and \eqref{lem::eqn::expression2}  complete the proof of Lemma \ref{lem::convergence_indicators}.
\end{proof}
\begin{proof}[Proof of Lemma~\ref{lem:asymp-second-term}]
Define $W_n(l) = \sum_{i\in [n]} \cI_i^n(l) -l$. 
Recall that $\sV_l$ denotes the set of vertices discovered up to time $l$, $\tau_k$ is the time when the $k$-th component has been explored, and $c_l$ is the number of components explored up to time $l$. 
Observe that
\begin{equation}\label{superm-deduction}
 \begin{split}
  \tilde{\E}[W_n(l+1)-W_n(l) \mid  \mathscr{F}_l]&= \sum_{i\in [n]} \tilde{\E}\big[\mathcal{I}^n_i(l+1)\mid  \mathscr{F}_l\big]\ind{i\notin \sV_l} -1 \\
  &= \sum_{i\notin \sV_l} \frac{\tilde{d}_i}{\tilde{\ell}_n-2l+2c_l+1} - 1 = \frac{2l-1- \sum_{i\in \sV_l} \td_i - 2c_l}{\tilde{\ell}_n-2l+2c_l+1}.
  \end{split}
 \end{equation} 
 To see that the final term in \eqref{superm-deduction} is negative, note that if $l= \tau_k$ for some $k$, then $\sum_{i\in \sV_{\tau_k}} \td_i -2\tau_k= 2k $, and $c_{\tau_k} = k$ so that  
 \begin{eq}\label{eq:ub-tauk-ub}
 2\tau_k - 1-\sum_{i\in \sV_{\tau_k}} \td_i - 2c_{\tau_k} = -1<0.
 \end{eq}
 If $\tau_{k}<l<\tau_{k+1}$, then $\sum_{i\in \sV_l\setminus\sV_{\tau_k}} \td_i -2(l-\tau_k)\geq -1$, and also $c_l = c_{\tau_k}+1$. Therefore, using \eqref{eq:ub-tauk-ub}, we conclude that the final term in \eqref{superm-deduction} is negative for all $l\geq 1$, and consequently,  $(W_n(l))_{l\geq 1}$ is a super-martingale.
We will use the martingale-inequality \cite[Lemma 2.54.5]{RW94} stating that for any sub/super-martingale $(M(t))_{t\geq 0}$, with $M(0)=0$, 
 \begin{equation}\label{eqn:supmg:ineq}
  \varepsilon \prob{\sup_{s\leq t}|M(s)|>3\varepsilon}\leq 3\expt{|M(t)|}\leq 3\left(|\expt{M(t)}|+\sqrt{\var{M(t)}}\right).
 \end{equation}
Using Taylor expansion, 
\begin{eq}
\tilde{\PR}(\cI_i^n(l)=1) \geq 1- \Big(1-\frac{\tilde{d}_i}{\tilde{\ell}_n}\Big)^l \geq \Big(\frac{l\tilde{d}_i}{\tilde{\ell}_n} - \frac{l^2\tilde{d}_i^2}{\tilde{\ell}_n^2}\Big)\ind{l\tilde{d}_i < \tell_n},
\end{eq}
and thus, using Lemma~\ref{lem:perc-degrees}, and $l = tn^{\rho}$, 
\begin{eq}\label{expt-bound-1}
n^{-\rho} |\tilde{\E}[W_n(tn^{\rho})]| &= t- n^{-\rho}\sum_{i\in [n]} \tilde{\PR}(\cI_i^n(tn^{\rho})=1) \\
&\leq t\sum_{i\in [n]} \frac{\td_i\ind{\td_i > \tell_n/tn^{\rho}}}{\tell_n}  + \frac{t^2n^{\rho} \sum_{i\in [n]}\tilde{d}_i^2}{\tilde{\ell}_n^2}. 
\end{eq} 
Let $\cE_n$ denote the good event that $\ell_n p_c(\lambda)\leq \tell_n\leq 2\ell_n p_c(\lambda)$ and $\td_i\leq 2p_c(\lambda) d_i$ for all $i$ such that $d_i>C_0n^{\rho}$ for some $C_0$ (sufficiently small). 
Using standard concentration inequalities for the binomial distribution \cite[Theorem 2.1]{JLR00}, $\PR(\cE_n^c) < 2\e^{-n^{\varepsilon}}$ for some $\varepsilon>0$. On the event $\cE_n$, $\td_i> C n^{\rho}$, and thus $d_i> C n^{\rho}$. 
We can bound 
\begin{eq}
\sum_{i\in [n]} \frac{\td_i\ind{\td_i > \tell_n/tn^{\rho}}}{\tell_n} \leq \frac{C_1}{\ell_n} \sum_{i\in [n]} d_i  \ind{d_i>C n^{\rho}} = o(1),
\end{eq}
where the final step follows using the uniform integrability from Assumption~\ref{assumption1}. 
The second term in \eqref{expt-bound-1} is $\oP(1)$ using Lemma~\ref{lem:perc-degrees}. Thus, 
\begin{eq}\label{expt-bound}
n^{-\rho} |\tilde{\E}[W_n(tn^{\rho})]| & = \oP(1). 
\end{eq} 
Next, note that for any $(x_1,x_2,\dots)$, $0\leq a+b \leq x_i$ and $a,b>0$ one has $\prod_{i=1}^R(1-a/x_i)(1-b/x_i)\geq \prod_{i=1}^R (1-(a+b)/x_i)$. Thus, for all $l\geq 1$ and $i\neq j$, 
  \begin{equation}\label{neg:correlation}
  \tilde{\PR}(\mathcal{I}_i^n(l)=0, \mathcal{I}_j^n(l)=0)\leq \tilde{\PR}(\mathcal{I}_i^n(l)=0)\tilde{\PR}(\mathcal{I}_j^n(l)=0),
  \end{equation} and thus
\begin{eq}\label{neg:correlation-2}
\tilde{\PR}(\mathcal{I}_i^n(l)=1, \mathcal{I}_j^n(l)=1) &= 1- \tilde{\PR}(\mathcal{I}_i^n(l)=0)-\tilde{\PR}(\mathcal{I}_j^n(l)=0)+ \tilde{\PR}(\mathcal{I}_i^n(l)=0, \mathcal{I}_j^n(l)=0) \\
&\leq 1- \tilde{\PR}(\mathcal{I}_i^n(l)=0)-\tilde{\PR}(\mathcal{I}_j^n(l)=0)+ \tilde{\PR}(\mathcal{I}_i^n(l)=0)\tilde{\PR}(\mathcal{I}_j^n(l)=0)\\
& =\tilde{\PR}(\mathcal{I}_i^n(l)=1)\tilde{\PR}( \mathcal{I}_j^n(l)=1).
\end{eq}  
Therefore $\mathcal{I}_i^n(l)$ and $\mathcal{I}^n_j(l)$ are negatively correlated. 
  Using \eqref{prob-ind-lb}, it follows that
\begin{eq}
   \mathrm{Var}(\mathcal{I}_i^n(l)\vert (\tilde{d}_i)_{i\in [n]})\leq  \tilde{\PR}(\mathcal{I}_i^n(l)=1) \leq \frac{l\td_i }{\tilde{\ell}_n(t)},
\end{eq}   
    uniformly over $l\leq tn^{\rho}$.
  Therefore, using  the negative correlation in \eqref{neg:correlation-2}, 
  \begin{equation} \label{variance-bound}
   \begin{split}
    n^{-2\rho}\mathrm{Var}\big(W_n(tn^{\rho})\big\vert (\tilde{d}_i)_{i\in [n]}\big)&\leq n^{-2\rho} \sum_{i\in [n]} \var{\mathcal{I}_i^n(tn^{\rho})\vert (\tilde{d}_i)_{i\in [n]}} \\
    &= n^{-2\rho} tn^{\rho} \frac{\sum_{i\in [n]}\tilde{d}_i}{\tilde{\ell}_n(t)}= \thetaP(n^{-\rho}) =\oP(1).
   \end{split}
  \end{equation}
Using \eqref{expt-bound} and \eqref{variance-bound}, the proof now follows by an application of \eqref{eqn:supmg:ineq}.
\end{proof}

\subsubsection{Large components are explored early} \label{sec:large-early}
In this section, we prove two key results that allow us to deduce the convergence of the component sizes. 
Firstly, we show that the rescaled vector of component sizes is tight in~$\ell^2_{\shortarrow}$ (see Proposition~\ref{prop:l2-tight}). 
This result is then used to show that the largest components of~$\mathcal{G}_n(p_c(\lambda))$ are explored before time $\Theta(n^{\rho})$ (Proposition~\ref{prop:large-comp-expl-early}). The latter allows us to apply Proposition~\ref{prop:conv-exc-length-area}.
Let $\sCil$ denote the $i$-th largest component for~$\mathcal{G}_n(p_c(\lambda))$.
Recall that our convention is to take $|\sC| =0$, if the component consists of one vertex and no edges.
\begin{proposition}\label{prop:l2-tight}
Under {\rm Assumption~\ref{assumption1}}, for any $\varepsilon>0$, 
\begin{equation}
\lim_{K\to\infty}\limsup_{n\to\infty}\PR\bigg(\sum_{i>K}|\sCil|^2>\varepsilon n^{2\rho}\bigg) = 0.
\end{equation}
\end{proposition}
Let $\mathcal{G}^{\sss K}_n$ be the random graph obtained by removing all edges attached to vertices $1,\dots,K$ and let $\boldsymbol{d}'$ be the obtained degree sequence. 
Further, let $\sC^{\sss K}(v)$ and $\sCi^{\sss K}$ denote the connected component containing $v$ and the $i$-th largest component respectively in $\mathcal{G}^{\sss K}_n$.
Let $D^{\sss K}(v) = \sum_{k\in \sC^{\sss K}(v)}\td_k$ and $D^{\sss K}_i = \sum_{k\in \sCi^{\sss K}}\td_k$. 
Let $V_n^{*,{\sss K}}$ be chosen according to the following size-biased distribution:
\begin{eq}
\PR(V_n^{*,{\sss K}} = i) = \frac{\td_i}{\tell_n - \sum_{i=1}^K \td_i}, \quad\text{for }i\in [n]\setminus[K].  
\end{eq}
Also, denote the criticality parameter of $\mathcal{G}^{\sss K}_n$ by $\nu_n^{\sss K}$.
\begin{lemma} \label{lem::tail_sum_squares} Suppose that  {\rm Assumption~\ref{assumption1}} holds.
Then, for any $\varepsilon>0$, 
\begin{eq}
\lim_{K\to\infty}\limsup_{n\to\infty}\PR\bigg(\tilde{\E}\bigg[\sum_{k\in \sC^K(V_n^{*,{\sss K}})}(\td_k-1)\bigg] > \varepsilon \bigg) = 0.
\end{eq}
\end{lemma}
\begin{proof}
Note that the criticality parameter of  $\mathcal{G}_n(p_c(\lambda))$ is $\tilde{\nu}_n = \lambda(1+\oP(1))$, by Lemma~\ref{lem:perc-degrees}.
Now, conditionally on the set of removed half-edges, $\mathcal{G}^{\sss K}_n$ is still a configuration model with some degree sequence $\boldsymbol{d}'$ with $d_i'\leq \td_i$ for all $i\in [n]\setminus [K]$ and $d_i'=0$ for $i\in [K]$.
Further, the criticality parameter of $\mathcal{G}^{\sss K}_n$ satisfies 
 \begin{equation}\label{eqn:nu-K}
  \begin{split}
   \nu^{\sss K}_n&= \frac{\sum_{i\in [n]} d_i'(d'_i-1)}{\sum_{i\in [n]} d_i'}\leq \frac{\sum_{i>K}\tilde{d}_i(\tilde{d}_i-1)}{\tilde{\ell}_n-2\sum_{i=1}^K\tilde{d}_i} = \lambda \frac{\sum_{i>K}\tilde{d}_i(\tilde{d}_i-1)}{\sum_{i\in [n]}\tilde{d}_i(\tilde{d}_i-1)}(1+\oP(1)),
  \end{split}
 \end{equation}
 where we have used $\tilde{\nu}_n = \lambda(1+\oP(1))$ in the last step. 
 Now, by Assumption~\ref{assumption1} and Lemma~\ref{lem:perc-degrees}, it is possible to choose $K_0$ large such that for all  $K\geq K_0$
 \begin{equation} \label{eq:fact-K-large-nu}
 \nu_n^{\sss K}<1 \quad \text{with high probability}.
 \end{equation} 
This yields
 \begin{eq}\label{path-count-ub-main}
 \tilde{\E} \bigg[\sum_{k\in \sC^K(V_n^{*,{\sss K}})}(\td_k-1)\bigg] \leq \tilde{\E} [\td_{V_n^{*,{\sss K}}}-1]\Big(1+  \frac{\tilde{\E}[\td_{\sss V_n^{*,{\sss K}}}]}{(1-\nu_n^{\sss K})}+\oP(1)\Big),
 \end{eq}
where $\td_{\sss V_n^{*,{\sss K}}}$ is the degree of the vertex $V_n^{*,{\sss K}}$ in $\cG_n(p_c(\lambda))$.
 The proof of \eqref{path-count-ub-main} uses path-counting techniques for the configuration model \cite{J09b}.  
Since the arguments are adaptations of \cite{DHLS16}, we move the proof to Appendix~\ref{sec:appendix-path-counting}.
We now use Lemma~\ref{lem:perc-degrees} to compute the asymptotics of the different terms in \eqref{path-count-ub-main}. 
Note that $\tilde{\E}[\td_{\sss V_n^{*,{\sss K}}}] \leq (1+\oP(1))\sum_{i>K}\tilde{d}_i^2/\tilde{\ell}_n = \OP(1)$, and 
\begin{equation}
\tilde{\E}[\td_{\sss V_n^{*,{\sss K}}}-1] = 
\frac{\sum_{i>K}\td_i(\td_i-1)}{\tilde{\ell}_n-\sum_{i=1}^K\tilde{d}_i} = (1+\oP(1)) \frac{p_n\sum_{i>K}d_i(d_i-1)}{\sum_{i\in [n]}d_i} \pto 0,
\end{equation}in the iterated limit $\lim_{K\to\infty} \lim_{n\to\infty}$.
Thus the proof of Lemma~\ref{lem::tail_sum_squares} follows.
\end{proof}

\begin{proof}[Proof of Proposition~\ref{prop:l2-tight}]
Recall that  $\sCi^{\sss K}$ denotes the $i$-th largest component in $\mathcal{G}^{\sss K}_n$ and $D^{\sss K}_i = \sum_{k\in \sCi^{\sss K}}\td_k$. 
Denote by $\mathscr{S}_K$, the squared sum of the component sizes after removing components containing $1,\dots,K$. 
Note that 
\begin{eq}\label{tight-ub-K}
\sum_{i>K} |\sCi|^2 &= \sum_{i\geq 1} |\sCi|^2-\sum_{i=1}^K |\sCi|^2 \leq \mathscr{S}_K \leq \sum_{i\geq 1} |\sCi^{\sss K}|^2 \leq 4\sum_{i\geq 1} D_i^{\sss K}\sum_{k\in \sC_{\sss (i)}^{\sss K}}(\td_k-1),
\end{eq}
where the last step uses $d_i'\leq \td_i$ and the fact that for any connected component  $\sC$ with total degree $D$, we must have $D-|\sC|\geq |\sC|/4$. 
The last fact can be seen for $|\sC|\geq 2$ by 
$D-|\sC| \geq 2(|\sC| -1) -|\sC| = |\sC| -2 \geq |\sC|/4$, and for $|\sC| = 1$ and $D\geq 2$, this follows trivially.
Note here that we do not consider components with $|\sC| = 1$ and $D=0$; see Remark~\ref{rem:isolated}.
Thus it is enough to bound the final term in \eqref{tight-ub-K}. 
Now, 
\begin{eq}\label{eq:ub-degre-moment}
\tilde{\PR}\bigg(\sum_{i\geq 1} D_i^{\sss K}\sum_{k\in \sC_{\sss (i)}^{\sss K}}(\td_k-1) >\varepsilon n^{2\rho}\bigg)&\leq \frac{1}{\varepsilon n^{2\rho}}\tilde{\E}\bigg[\sum_{i\geq 1} D_i^{\sss K}\sum_{k\in \sC_{\sss (i)}^{\sss K}}(\td_k-1)  \bigg] \\
&= \frac{\tell_n-\sum_{i\in [K]} \td_i}{\varepsilon n^{2\rho} }\tilde{\E}\bigg[\sum_{k\in \sC^K(V_n^{*,{\sss K}})}(\td_k-1)\bigg].
\end{eq}
Thus, the proof follows using Lemma~\ref{lem::tail_sum_squares}, and the fact that $\tell_n-\sum_{i\in [K]} \td_i \leq \tilde{\ell}_n = \OP(n^{2\rho})$.
\end{proof}
The next proposition shows that, in Algorithm~\ref{algo-expl}, the large components are explored before time $\Theta(n^{\rho})$.
Let $\sC_{\max}^{\sss \geq T}$ denote the size of the largest component whose exploration is started by Algorithm~\ref{algo-expl} after time $Tn^{\rho}$, and let $D_{\max}^{\sss \geq T} = \sum_{k\in \sC_{\max}^{\sss \geq T}}\td_k$. 
\begin{proposition}\label{prop:large-comp-expl-early}
 Under {\rm Assumption~\ref{assumption1}}, for any $\varepsilon>0$, 
\begin{equation}
\lim_{T\to\infty}\limsup_{n\to\infty}\PR\big(|\sC_{\max}^{\sss \geq T}|>\varepsilon n^{\rho}\big) = 0 \quad\text{and} \quad  \lim_{T\to\infty}\limsup_{n\to\infty}\PR\big(D_{\max}^{\sss \geq T}>\varepsilon n^{\rho}\big) = 0.
\end{equation}
\end{proposition}
\begin{proof}
Define $\mathscr{A}_{\sss K,T}^n:= \{\text{all the vertices of }[K] \text{ are explored before time }Tn^{\rho}\}.$ Let $\mathscr{C}_{\sss (i)}^{\sss K}$  denote the $i$-th largest component of $\cG^{\sss K}_n$ so that 
 \begin{eq}\label{eq:CgeqT1}
  \probc{|\mathscr{C}_{\max}^{\sss \geq T}|>\varepsilon n^\rho, \ \mathscr{A}_{\sss K,T}^n}&\leq \tilde{\PR}\bigg(\sum_{i\geq 1}\big|\mathscr{C}_{\sss (i)}^{\sss K}\big|^2> \varepsilon^2 n^{2\rho}\bigg) \\
  &\leq \tilde{\PR}\bigg(\sum_{i\geq 1} D_i^{\sss K}\sum_{k\in \sC_{\sss (i)}^{\sss K}}(\td_k-1) >\frac{\varepsilon^2 n^{2\rho}}{4}\bigg).
 \end{eq}
The final term tends to zero in probability in the iterated limit $\lim_{K\to\infty}\limsup_{n\to\infty}$, as shown in \eqref{eq:ub-degre-moment}. 
 Next, using the fact that $\tilde{d}_jn^{\rho}=\Theta(\tilde{\ell}_n)$, we get 
 \begin{equation}\label{eq:CgeqT2}
  \begin{split} &\probc{(\mathscr{A}_{\sss K,T}^{n})^c}=\probc{\exists j\in [K]: j \text{ is not explored before }Tn^{\rho}}\\
  &\leq \sum_{j=1}^K \probc{j \text{ is not explored before }Tn^{\rho}}\leq \sum_{j=1}^K\left(1-\frac{\tilde{d}_j}{\tilde{\ell}_n-\Theta(Tn^{\rho})} \right)^{Tn^{\rho}}\leq \sum_{j=1}^K \e^{-CT},
  \end{split}
 \end{equation} where $C>0$ is a constant that may depend on $K$, and the final step holds with high probability. 
 Now, by \eqref{eq:CgeqT1},
 \begin{equation}
  \probc{|\mathscr{C}_{\max}^{\sss \geq T}|>\varepsilon n^{\rho}}\leq \tilde{\PR}\bigg(\sum_{i\geq 1}\big|\mathscr{C}_{\sss (i)}^{\sss K}\big|^2> \varepsilon^2 n^{2\rho}\bigg) + \probc{(\mathscr{A}_{\sss K,T}^{n})^c}.
 \end{equation} The proof for $\PR\big(|\sC_{\max}^{\sss \geq T}|>\varepsilon n^{\rho}\big)$ follows by taking the iterated limit $\lim_{K\to\infty}\lim_{T\to\infty}\limsup_{n\to\infty}$.

For the upper bound on $\probc{D_{\max}^{\sss \geq T}>\varepsilon n^\rho, \ \mathscr{A}_{\sss K,T}^n}$, note that 
\begin{eq}
\probc{D_{\max}^{\sss \geq T}>\varepsilon n^\rho, |\mathscr{C}_{\max}^{\sss \geq T}|\leq \varepsilon n^{\rho}/2, \ \mathscr{A}_{\sss K,T}^n}&\leq \probc{D_{\max}^{\sss \geq T}(D_{\max}^{\sss \geq T}-|\mathscr{C}_{\max}^{\sss \geq T}|)>\varepsilon^2 n^{2\rho}/2, \ \mathscr{A}_{\sss K,T}^n} \\
  &\leq \tilde{\PR}\bigg(\sum_{i\geq 1} D_i^{\sss K}\sum_{k\in \sC_{\sss (i)}^{\sss K}}(\td_k-1) >\frac{\varepsilon^2 n^{2\rho}}{2}\bigg).
\end{eq}
Hence, the proof for $\PR\big(D_{\max}^{\sss \geq T}>\varepsilon n^{\rho}\big)$ also follows.
\end{proof}
\subsubsection{Counting process that counts surplus}\label{sec:surplus-fd}
Let $N_n^\lambda(k)$ be the number of surplus edges discovered up to time $k$ and $\bar{N}^\lambda_n(u) = N_n^\lambda(\lfloor un^\rho \rfloor)$. 
Below, we prove the asymptotics for the process $\bar{\mathbf{N}}^\lambda_n$: 
\begin{lemma} \label{lem:surp:poisson-conv} 
Under {\rm Assumption~\ref{assumption1}}, as $n\to\infty$,
 \begin{equation}\label{eq:limit-joint-comp-sp}
 (\bar{\mathbf{S}}_n,\bar{\mathbf{N}}_n^\lambda)\dto (\mathbf{S}_{\infty}^\lambda,\mathbf{N}^\lambda),
 \end{equation} where $\mathbf{N}^\lambda$ is defined in \eqref{defn::counting-process}.
 \end{lemma}
\begin{proof}
  We write $
N_n^{\lambda}(l)=\sum_{i=2}^l \xi_i$,
where $\xi_i=\ind{\mathscr{V}_i=\mathscr{V}_{i-1}}$. 
Let $A_i$ denote the number of active half-edges after stage $i$ while implementing Algorithm~\ref{algo-expl}. Note that 
\begin{equation}\label{prob-increment-surplus}
 \probc{\xi_i=1\mid \mathscr{F}_{i-1}}=\frac{A_{i-1}-1}{\tilde{\ell}_n-2i-1}= \frac{A_{i-1}}{\tilde{\ell}_n}(1+O(i/\tell_n))+O(\tell_n^{-1}),
\end{equation} uniformly for $i\leq Tn^{\rho}$ for any $T>0$. 
By Lemma~\ref{lem:perc-degrees}, $\tilde{\ell}_n = \ell_n p_c(\lambda) (1+\oP(1))= n^{2\rho} \lambda \mu^2/ \sum_i \theta_i^2(1+\oP(1))$.
Therefore, the instantaneous rate of change of the re-scaled process $\bar{\mathbf{N}}^{\lambda}$ at time $t$, conditional on the past, is 
\begin{equation}\label{eqn:intensity}
 n^{\rho}\frac{A_{\floor{tn^\rho}}}{n^{2\rho}\frac{\lambda\mu^2}{\sum_{i\geq 1}\theta_i^2}}\left( 1+\oP(1)\right) +\oP(1)= \frac{\sum_{i\geq 1}\theta_i^2}{\lambda \mu^2}\refl{\bar{S}_n(t)}\left( 1+\oP(1)\right) +\oP(1).
\end{equation}
Since the reflection of a process is continuous in Skorohod $J_1$-topology (see \cite[Lemma 13.5.1]{W02}), we can use Theorem~\ref{thm::convegence::exploration_process}  to conclude that $\mathrm{refl}(\bar{\mathbf{S}}_n)  \xrightarrow{\sss d} \mathrm{refl} (\mathbf{S}^\lambda_\infty )$, so that the compensator of 
$\bar{\mathbf{N}}_n^\lambda$ converges. 
The convergence of the compensators is usually enough for convergence of Poisson processes.
Indeed, for Erd\H{o}s-R\'enyi random graphs \cite{A97} or rank-one inhomogeneous random graphs \cite{BHL10,BHL12}, showing the convergence of compensators suffices using \cite[Theorem 1]{Bro81}. This is because the surplus edges can be added independently after we have observed the whole exploration process. However, this is not true for the configuration model because the surplus edges occur precisely at places with jumps $-2$.
This difficulty was circumvented in \cite{DHLS16} for the $\tau\in (3,4)$ regime.
In Appendix~\ref{appendix-surplus}, we adapt the arguments from \cite{DHLS16} in the $\tau\in (2,3)$ setting, which completes the proof of Lemma~\ref{lem:surp:poisson-conv}. 
 \end{proof}

\subsubsection{Convergence of the component sizes and the surplus edges} \label{sec:conv-surp-comp-size}
We first show the asymptotics of the component sizes and surplus edges of $\mathcal{G}_n(p_c(\lambda))$ generated by  Algorithm~\ref{algo-alt-cons-perc}. 
Recall that $\mathrm{SP}(\mathscr{C})$ denotes the number of  surplus of $\mathscr{C}$. 
The following lemma states the tightness of the  vector of component sizes and surplus edges of $\mathcal{G}_n(p_c(\lambda))$ in the $\Unot$-topology:
\begin{lemma}\label{lem-surp-u-0} For any $\varepsilon >0$,
 \begin{equation}
 \lim_{\delta\to 0}\limsup_{n\to\infty} \PR\bigg(\sum_{i: |\mathscr{C}_{(i)}|\leq \delta n^{\rho} }|\mathscr{C}_{\sss(i)}|\times \surp{\mathscr{C}_{\sss(i)}}> \varepsilon n^{\rho}\bigg)=0.
 \end{equation}
\end{lemma}

\noindent The proof of Lemma~\ref{lem-surp-u-0} is an adaptation of  \cite[Proposition 19]{DHLS16} in this setting. 
We provide a proof of Lemma~\ref{lem-surp-u-0} in Appendix~\ref{appendix-surp-unot}. 
Next, let $\mathbf{Z}_n'(\lambda)$ denote the vector $(n^{-\rho}|\sCi|,\SP(\sCi))_{i\geq 1}$, ordered as an element in~$\Unot$. Below, we prove the scaling limit of $\mathbf{Z}_n'(\lambda)$: 
\begin{proposition}\label{lem:component-sizes}
Under {\rm Assumption~\ref{assumption1}}, as $n\to\infty$, 
\begin{equation}\label{eq:l2-conv}
\mathbf{Z}_n'(\lambda) \dto \mathbf{Z}(\lambda)
\end{equation}with respect to the $\Unot$ topology, where $\mathbf{Z}(\lambda)$ is defined in \eqref{eq:Z-limit}. 
\end{proposition}
\begin{proof}
Recall from Proposition~\ref{prop:limit-properties} that the limiting process $\biS$ is good in the sense that all the conditions in Definition~\ref{defn::good_function} are satisfied. Also, Proposition~\ref{prop:large-comp-expl-early} ensures that the additional restriction on the pre-limit process in Proposition~\ref{prop:conv-exc-length-area} is satisfied. 
Thus, using Theorem~\ref{thm::convegence::exploration_process},  an application of Proposition~\ref{prop:conv-exc-length-area} yields the finite-dimensional convergence in \eqref{eq:l2-conv}. Finally, the convergence in the $\Unot$-topology follows using the tightness in Lemma~\ref{lem-surp-u-0}. 
\end{proof}

We now provide a proof of Theorem~\ref{thm:main}:

\begin{proof}[Proof of Theorem~\ref{thm:main}] 
Throughout the proof, we ignore the $\lambda$ in a predefined notation to simplify writing. 
We will work under the coupling under which  Proposition~\ref{prop:coupling-lemma} holds, i.e.,
$\mathcal{G}_n(p_c(1-\varepsilon_n))\subset\mathrm{CM}_n(\bld{d},p_c)\subset\mathcal{G}_n(p_c(1+\varepsilon_n))$, where $\varepsilon_n\to 0$. 
We write $\sCi^-$, $\sCi$ and $\sCi^+$ to denote the $i$-th largest component of 
$\mathcal{G}_n(p_c(1-\varepsilon_n))$, $\mathrm{CM}_n(\bld{d},p_c)$ and $\mathcal{G}_n(p_c(1+\varepsilon_n))$ respectively, and let $\mathbf{Z}_n^-$, $\mathbf{Z}_n$ and $\mathbf{Z}_n^+$ be the corresponding vectors, rearranged as elements of $\Unot$. 
Then,  
\begin{eq} \label{identical-limit}
\mathbf{Z}_n^+ \text{ and } \mathbf{Z}_n^- \text{ have identical scaling limits as Proposition~\ref{lem:component-sizes}.}
\end{eq}
Let $\mathrm{d}_{\sss \mathbb{U}} $ denote the metric for the $\Unot$ topology defined in \eqref{defn_U_metric}. The proof is complete if we can show that, as $n\to\infty$, 
\begin{eq}\label{eq:dist-mod}
\mathrm{d}_{\sss \mathbb{U}} (\mathbf{Z}_n^+,\mathbf{Z}_n) \pto 0.
\end{eq}
First, we prove that, for any $K\geq 1$, 
 \begin{eq}\label{eq:sandwich-perc-12}
 \lim_{n\to\infty} \PR(\mathscr{C}_{\sss (i)} ^{-} \subset \mathscr{C}_{\sss (i)} ^{+}, \ \forall i\leq K) = 1.
 \end{eq}
If $\mathscr{C}_{\sss (1)} ^{-} $ is not contained in $ \mathscr{C}_{\sss (1)} ^{+}$, then $|\mathscr{C}_{\sss (1)} ^{-}|\leq |\mathscr{C}_{\sss (j)} ^{+}|$ for some $j\geq 2$, which implies that $|\mathscr{C}_{\sss (1)} ^{-}|\leq |\mathscr{C}_{\sss (2)} ^{+}|$. 
Suppose that there is a subsequence $(n_{0k})_{k\geq 1} \subset \N$ along which 
\begin{eq}\label{liminf-prob-pos}
\lim_{n_{0k}\to\infty}   \PR(|\mathscr{C}_{\sss (1)}^{-}| \leq |\mathscr{C}_{\sss (2)}^{+}|) >0.
\end{eq}
If \eqref{liminf-prob-pos} yields a contradiction, then \eqref{eq:sandwich-perc-12} is proved for $K=1$. 
To this end, first note that $(n^{-\rho}(|\mathscr{C}_{\sss (i)}^{-}|, |\mathscr{C}_{\sss (i)}^{+}|)_{i\geq 1})_{n\geq 1}$ is tight in $(\ell^2_{\shortarrow})^2$.
Thus taking a subsequence $(n_k)_{k\geq 1} \subset (n_{0k})_{k\geq 1}$ along which the random vector converges, 
it follows that 
\begin{eq}
n_k^{-\rho}(|\mathscr{C}_{\sss (i)}^{-}|, |\mathscr{C}_{\sss (i)}^{+}|)_{i\geq 1} \dto (\gamma_i,\bar{\gamma}_i)_{i\geq 1} \quad \text{ in }(\ell^2_{\shortarrow})^2,
\end{eq}
 where $(\gamma_i)_{i\geq 1} \stackrel{\sss d}{=}(\bar{\gamma}_i)_{i\geq 1}$.
Thus, along the subsequence $(n_k)_{k\geq 1}$, 
\begin{eq} \label{eq:prob-second-larger-largest-2}
\lim_{n_k\to\infty} \PR(|\mathscr{C}_{\sss (1)}^{-}| \leq |\mathscr{C}_{\sss (2)}^{+}|) = \PR(\gamma_1\leq \bar{\gamma}_2).
\end{eq}
\begin{fact}\label{fact:same-dist-coupling-2}
For all $i\geq 1$, $\gamma_i = \bar{\gamma}_i$ almost surely.
\end{fact}
\begin{proof}
Under the coupling in Proposition~\ref{prop:coupling-lemma}, $\sum_{j\leq i}|\mathscr{C}_{\sss (j)}^{-}|\leq \sum_{j\leq i}|\mathscr{C}^{+}_{\sss (j)}|$ and therefore $\PR(\sum_{j\leq i}\gamma_j\leq \sum_{j\leq i} \bar{\gamma}_j) =1$, for each fixed $i\geq 1$. 
In particular, $\gamma_1 \leq \bar{\gamma}_1$ almost surely. But, since $\gamma_1,\bar{\gamma}_1$ have the same distribution, it must be the case that $\gamma_1 = \bar{\gamma}_1$ almost surely. 
Inductively, we can prove that $\gamma_i = \bar{\gamma}_i$ almost surely.
\end{proof}
\noindent Thus, using Fact~\ref{fact:same-dist-coupling-2}, \eqref{eq:prob-second-larger-largest-2} reduces to
\begin{eq} \label{eq:prob-second-larger-largest-3}
\lim_{n_k\to\infty} \PR(|\mathscr{C}_{\sss (1)}^{-}| \leq |\mathscr{C}^{+}_{\sss (2)}|) = \PR(\gamma_1\leq \gamma_2) = \PR(\gamma_1= \gamma_2) = 0,
\end{eq}
where the last equality follows from Definition~\ref{defn::good_function}(f) and Proposition~\ref{prop:limit-properties}. Note that  \eqref{eq:prob-second-larger-largest-3} contradicts \eqref{liminf-prob-pos}, and thus \eqref{eq:sandwich-perc-12} follows for $K=1$.
For $K\geq 2$, we can use a similar argument to show that, with high probability, $\cup_{i\leq K}\mathscr{C}_{\sss (i)}^{-}\subset \cup_{i\leq K}\mathscr{C}_{\sss (i)}^{+}$. 
If both $\mathscr{C}_{\sss (1)}^{-}$ and  $\mathscr{C}_{\sss (2)}^{-}$ are  contained in $\mathscr{C}^{+}_{\sss (1)}$, then $|\mathscr{C}^{+}_{\sss (1)}| \geq |\mathscr{C}_{\sss (1)}^{-}|+|\mathscr{C}_{\sss (2)}^{-}|$, which occurs with probability tending to zero.
This follows using Fact~\ref{fact:same-dist-coupling-2} and $\PR(\bar{\gamma}_1\geq \gamma_1+\gamma_2)=\PR(\gamma_1\geq \gamma_1+\gamma_2) = 0$.
Thus, $\mathscr{C}_{\sss (2)}^{-} \subset \mathscr{C}^{+}_{\sss (2)}$ with high probability and we can use similar arguments to conclude \eqref{eq:sandwich-perc-12} for $i\leq K$.

Next, we show that, for any $K\geq 1$, 
 \begin{eq}\label{eq:sandwich-perc-123}
 \lim_{n\to\infty} \PR\big(\mathscr{C}_{\sss (i)} ^{-} \subset \mathscr{C}_{\sss (i)}  \subset \mathscr{C}_{\sss (i)} ^{+}, \ \forall i\leq K\big) = 1.
 \end{eq}
If $\mathscr{C}_{\sss (1)}  $ is not contained in $ \mathscr{C}_{\sss (1)} ^{+}$, then $|\mathscr{C}_{\sss (1)}|\leq |\mathscr{C}_{\sss (2)} ^{+}|$. 
However, since $|\mathscr{C}_{\sss (1)}^{-}|\leq |\mathscr{C}_{\sss (1)}|$, it follows that $|\mathscr{C}_{\sss (1)}^{-}|\leq |\mathscr{C}_{\sss (2)} ^{+}|$.
Now, one can repeat identical argument as in   \eqref{eq:sandwich-perc-12} to prove that $\mathscr{C}_{\sss (i)}  \subset \mathscr{C}_{\sss (i)} ^{+}$ for all $i\leq K$ with high probability.
Moreover, since $\mathcal{G}_n(p_c(1-\varepsilon_n)) \subset \mathrm{CM}_n(\bld{d}) $ and $\mathscr{C}_{\sss (i)} ^{-} \subset \mathscr{C}_{\sss (i)} ^{+}$ for all $i\leq K$ with high probability, it must also be the case that $\mathscr{C}_{\sss (i)} ^{-} \subset \mathscr{C}_{\sss (i)} \subset \mathscr{C}_{\sss (i)} ^{+}$ for all $i\leq K$ with high probability. Thus we conclude \eqref{eq:sandwich-perc-123}.
Finally, since $\mathbf{Z}_n^{-}$ and $\mathbf{Z}_n^{+}$ have the same distributional limit, it follows using \eqref{eq:sandwich-perc-12} that, for all $i\leq K$,
\begin{eq}
|\mathscr{C}_{\sss (i)} ^{+}| - |\mathscr{C}_{\sss (i)} ^{-}| = \oP(n^\rho) \quad\text{and}\quad  \mathrm{SP}(\mathscr{C}_{\sss (i)} ^{+})-\mathrm{SP}(\mathscr{C}_{\sss (i)} ^{-}) \pto 0.
\end{eq}
Thus, \eqref{eq:sandwich-perc-123} yields
 \begin{eq}\label{eq:bounds-comp-surp}
\big| |\mathscr{C}_{\sss (i)} ^{+}| - |\mathscr{C}_{\sss (i)}| \big|= \oP(n^{\rho}) \quad\text{and}\quad  \big|\mathrm{SP}(\mathscr{C}_{\sss (i)} ^+)-\mathrm{SP}(\mathscr{C}_{\sss (i)} ) \big|\pto 0.
\end{eq} 
Moreover, 
since both $(\mathbf{Z}_n^{-})_{n\geq 1}$ and $(\mathbf{Z}_n^{+})_{n\geq 1}$ are tight in~$\mathbb{U}^0_{\shortarrow}$, it also follows that $(\mathbf{Z}_n)_{n\geq 1}$ is tight in~$\mathbb{U}^0_{\shortarrow}$.
Thus \eqref{eq:dist-mod} follows and the proof of Theorem~\ref{thm:main} is now complete. 
\end{proof}

\subsection{Analysis of the diameter}\label{sec:diameter}
In this section, we investigate the asymptotics of the diameter of $\cG_n(p_c(\lambda))$.
As in the proof of Theorem~\ref{thm:main}, an application of Proposition~\ref{prop:coupling-lemma} yields the diameter of $\mathrm{CM}_n(\bld{d},p_c(\lambda))$ and completes the proof.

\begin{proof}[Proof of Theorem~\ref{thm:diameter-large-comp}] 
First let us fix $\lambda<1$ and use path counting.
Let $P_l$ denote the number of paths of length $l$ in $\cG_n(p_c(\lambda))$. 
Since $\lambda<1$, we have that $\tilde{\nu}_n = \lambda (1+\oP(1)) < 1$ with high probability. 
Now, an application of \cite[Lemma 5.1]{J09b} yields that for all $l\geq 1$, $\tilde{\E}[P_l] \leq \tell_n (\tilde{\nu}_n)^{l-1}$. 
Thus, on the event $\{\tilde{\nu}_n<1\}$, for any $K\geq 1$, 
\begin{eq}\label{diam-comput}
\tilde{\PR}(\diam(\cG_n(p_c(\lambda))) >K) \leq \sum_{l>K} \tilde{\E}[P_l]\leq \frac{\tell_n(\tilde{\nu}_n)^K}{1-\tilde{\nu}_n}
\end{eq}
Now, taking $K= C\log n$ for some large constant $C>0$ gives the desired $\log n$ bound on the diameter of $\cG_n(p_c(\lambda))$ with high probability for $\lambda<1$.

To extend to the case $\lambda \geq 1$, we delete $R$ highest-degree vertices to obtain a new graph $\mathcal{G}_n^{\sss R}$. 
Using \eqref{eqn:nu-K}, $\mathcal{G}_n^{\sss R}$ is a configuration model with the criticality parameter $\nu_n^{\sss R} <1$ with high probability. Thus the above result applies for $\mathcal{G}_n^{\sss R}$. 
However, after putting back the $R$ deleted vertices, the diameter of $\mathcal{G}^{\sss >R}$ can increase by at most a factor of $R$.
 This implies the $\log n$ bound on the diameter of $\cG_n(p_c(\lambda))$ with high probability for $\lambda \geq 1$.
 Finally, as remarked in the beginning of this section, the proof of Theorem~\ref{thm:diameter-large-comp} follows by invoking Proposition~\ref{prop:coupling-lemma}.
\end{proof}

\section{Near-critical behavior} \label{sec:near-critical-proofs}
Finally we consider the near-critical behavior for $\mathrm{CM}_n(\bld{d},p)$ in this section. 
The analysis for the barely subcritical and supercritical regimes are given separately  in Sections~\ref{sec:subcrit}~and~\ref{sec:supercrit} respectively.
\subsection{Barely-subcritical regime}\label{sec:subcrit}
In this section, we analyze the barely-subcritical regime ($p_n\ll p_c$) for percolation and complete the proof of Theorem~\ref{thm:barely-subcrit}. 
Recall the exploration process from Algorithm~\ref{algo-expl} on the graph $\cG_n(p_n)$, starting with vertex $j$.
Let $\sC(j,p_n)$ denote the connected component in~$\cG_n(p_n)$ containing vertex $j$.
 We will use the same notation for the quantities defined in Section~\ref{sec:expl-process}, but the reader should keep in mind that we now deal with different $p_n$ values. 
 We avoid augmenting $p_n$ in the notation for the sake of simplicity.
Consider exploring the graph using Algorithm~\ref{algo-expl} but starting from vertex $j$. 
The exploration process $\mathbf{S}_n^j$ is given by
   \begin{equation}
   \begin{split}
    S_n^j(0) = \tilde{d}_j, \quad S_n^j(l)= \tilde{d}_j+\sum_{i: i\neq j} \tilde{d}_i \mathcal{I}_i^n(l)-2l.
    \end{split}
   \end{equation}
  Thus the exploration process starts from $\tilde{d}_j$ now. 
   Now, for any $u>0$, as $n\to\infty$, 
   \begin{eq}\label{asympt-barely-subcrit-expl}
   \sup_{u\leq t}(n^{\alpha}p_n)^{-1} \Big|\sum_{i:i\neq j}\mathcal{I}_i^n(un^{\alpha}p_n) -un^{\alpha}p_n\Big| \pto 0.
   \end{eq}
This follows using identical arguments as in Lemma~\ref{lem:asymp-second-term}, and thus is  skipped here.            
Consider the re-scaled process $\bar{\mathbf{S}}^j_n$ defined as $\bar{S}^j_n(t)= (n^{\alpha}p_n)^{-1}S_n^j(\floor{tn^{\alpha}p_n })$. Then, 
   \begin{eq} \label{eqn::scaled_process_j}
    \bar{S}_n^j(t)&= (n^{\alpha}p_n)^{-1}\tilde{d}_j+(n^{\alpha}p_n)^{-1} \sum_{i:i\neq j}\tilde{d}_i \mathcal{I}_i^n(tn^{\alpha}p_n) - 2 t +\oP(1) \\
    & = \theta_j +(n^{\alpha}p_n)^{-1} \sum_{i:i\neq j}(\tilde{d}_i-1) \mathcal{I}_i^n(tn^{\alpha}p_n) -  t +\oP(1).
   \end{eq}
Recall that $\tilde{\E}$ is the conditional expectation conditionally on $(\tilde{d}_i)_{i\in [n]}$.
 Now, since the vertices are explored in a size-biased manner with the sizes being $(\tilde{d}_i/\tilde{\ell}_n)_{i\in [n]}$, for any $t\geq 0$,
\begin{eq} \label{expt-barely-sub-diff}
\tilde{\E}\bigg[\frac{1}{n^{\alpha}p_n}\sum_{i:i\neq j} (\tilde{d_i}-1)\cI_i^n\big(\lfloor tn^{\alpha}p_n\rfloor\big)\bigg] \leq \frac{tn^{\alpha}p_n}{n^{\alpha}p_n \tilde{\ell}_n}\sum_{i\in [n]}\tilde{d}_i(\tilde{d}_i-1) = \oP(1),
\end{eq}where the first inequality uses \eqref{prob-ind-lb}, and the final step follows from Lemma~\ref{lem:perc-degrees}. 
Consequently, $\bar{\mathbf{S}}_n^j$ converges in probability to the deterministic process $(\theta_j-t)_{t\in [0,\theta_j]}$. 
Thus 
\begin{eq}\label{edge-converge-subcrit}
\# \text{ edges in }\sC(j,p_n) \pto \theta_j.
\end{eq}
Next, the proof above shows that $\max_{l\leq \theta_j n^{\alpha}p_n} S_n^j(l) \leq 2\theta_j n^{\alpha}p_n$ with high probability.
Thus, the probability of creating a surplus edge at each step is at most $2\theta_j n^{\alpha}p_n/\tilde{\ell}_n$.
This implies that the probability of creating at least one surplus edge before $\theta_j n^{\alpha} p_n$ is at most $2\theta_j^2n^{2\alpha}p_n^2/\tilde{\ell}_n = \OP(n^{2\alpha-1}p_n)= \oP(1)$. 
Together with \eqref{edge-converge-subcrit}
yields
\begin{equation}
(n^{\alpha}p_n)^{-1}|\sC(j,p_n)| \pto \theta_j, \quad \text{and} \quad \PR(\surp{\sC(j,p_n)} =0) \to 1.
\end{equation}
From \eqref{edge-converge-subcrit}, we can also conclude that $\lim_{n\to\infty}\PR(i\in \mathscr{C}(j)) =0$ for all $i,j\geq 1$ and $i\neq j$, since, if $i\in \mathscr{C}(j)$, then the number of edges in $\sC(j,p_n)$ is atleast $\td_{i}+\td_j = n^{\alpha}p_n(\theta_i+\theta_j)$. Thus, $\sC(i,p_n)$ and $\sC(j,p_n)$ are disjoint with high probability. 

To conclude Theorem~\ref{thm:barely-subcrit}, 
we show that the rescaled vector of ordered component sizes is tight in $\ell^2_{\shortarrow}$. 
This tightness also yields that, for each fixed $j\geq 1$, 
\begin{equation}\label{eq:whp-Cj-barely-sub}
|\mathscr{C}(j,p_n)| = |\sC_{\sss (j)}(p_n)|, \text{ with high probability.} 
\end{equation}
To show $\ell^2_{\shortarrow}$-tightness, it is enough to show that, for any $\varepsilon >0$, 
\begin{equation}
\lim_{K\to\infty}\limsup_{n\to\infty}\PR\bigg(\sum_{i>K}|\sCi(p_n)|^2>\varepsilon n^{2\alpha}p_n^2\bigg) = 0.
\end{equation}
This can be concluded using identical arguments as in the proof of Proposition~\ref{prop:l2-tight} above.
The proof of Theorem~\ref{thm:barely-subcrit} is now complete.

\subsection{Barely-supercritical regime} \label{sec:supercrit}
In this section, we provide the proof of Theorem~\ref{thm:barely-supercrit}. Let $p_n = \lambda_n n^{-\eta}$, where $\lambda_n\to\infty$ since $p_n \gg p_c(\lambda)$. 
Our main tool here is a general result \cite[Theorem 5.4]{HJL16}, that provides asymptotics of the component sizes, if one can verify certain properties of an associated exploration process. 
Using Proposition~\ref{prop:coupling-lemma}, it is enough to prove Theorem~\ref{thm:barely-supercrit} for the graph $\cG_n(p_n)$ generated by Algorithm~\ref{algo-alt-cons-perc}. 
Let $\Mtilde{\bld{d}}$ denote the degree sequence obtained after performing Algorithm~\ref{algo-alt-cons-perc}~(S1). 
Thus, $\cG_n(p_n)$ is distributed as $\rCM_n(\Mtilde{\bld{d}})$.
We will verify Assumptions (B1)--(B8) from \cite{HJL16} on the graph $\cG_n(p_n)$,
which allows us to conclude Theorem~\ref{thm:barely-supercrit} from \cite[Theorem 5.4]{HJL16}.
We start by describing the following exploration process on $\Gp$ from \cite[Section 5.1]{HJL16}:

\begin{algo}\label{algo:expl-janson}
\normalfont
\begin{itemize}
\item[(S0)] Associate an independent $\mathrm{Exponential}(1)$ clock $\xi_e$ to each half-edge $e$. 
Any half-edge can be in one of the \emph{states} among sleeping, active, and dead. 
Initially at time $0$, all the half-edges are sleeping. 
Whenever the set of active half-edges is empty, select a sleeping half-edge $e$ uniformly at random among all sleeping half-edges and declare it to be active. 
If $e$ is incident to $v$, then declare all the other half-edges of $v$ to be active as well.
The process stops when there is no sleeping half-edge left; the
remaining sleeping vertices are all isolated and we have explored all other components.

\item[(S1)]  Pick an active half-edge (which one does not matter) and kill it, i.e., change its status to dead.
\item[(S2)] Wait until the next half-edge dies (spontaneously). This half-edge is paired to the one killed in the previous step (S1) to form an edge of the graph. 
If the vertex it belongs to is sleeping, then we declare this vertex awake and all of its other half-edges active.
Repeat from (S1) if there is any active half-edge; otherwise from (S0).
\end{itemize}
\end{algo}
Denote the number of living half-edges upto time $t$ by $L_n(t)$.
Let $\tV_{n,k}(t)$ denote the number of sleeping vertices of degree $k$ such that all the $k$ associated exponential clocks ring after time $t$.
Define
\begin{equation}\label{identities-A-V-S}
\tV_n(t) = \sum_{k=1}^\infty \tV_{n,k}(t), \quad \tS_n(t) = \sum_{k=1}^\infty k\tV_{n,k}(t),\quad \tA_n(t) = L_n(t)-\tS_n(t).
\end{equation}
We show that Assumptions (B1)--(B8) from \cite{HJL16} hold with 
\begin{eq}\label{eq:choice-parameters-barely-supercrit}
\zeta = \kappa^{\frac{1}{3-\tau}}, \quad \gamma_n = \beta_n = p_n^{\frac{\tau-2}{3-\tau}}, \quad \psi(t) =  \kappa t^{\tau-2} -t, \quad  \hat{g}(t) = t, \quad \hat{h}(t) = \kappa t^{\tau-2}+t,
\end{eq}where we recall the definition of $\kappa$ from \eqref{eq:asymp-laplace}.
The $\zeta$ in our notation corresponds to $\tau$ in the notation of \cite[Theorem 5.4]{HJL16}. We have used $\zeta$ instead of $\tau$, since in our paper $\tau$ denotes the power-law exponent.

We first find the number of vertices in $\mathcal{G}_n(p_n)$.
Let $\tilde{n}:= \#\{i: \tilde{d}_i \geq 1 \}$. 
Recall that $V_n$ is a vertex chosen uniformly at random from $[n]$ and let $D_n = d_{V_n}$ be the degree of $V_n$ in $\CM$.
Note that 
\begin{eq}\label{eq:expt-n-tilde-expr}
\E[\tilde{n}] &= \E\bigg[\sum_{i\in [n]} \ind{\tilde{d}_i \geq 1} \bigg] = \sum_{i\in [n]} \big(1- (1-p_n)^{d_i}\big) = n \E[1-(1-p_n)^{D_n}].
\\
\end{eq}
Using that $1-(1-x)^k \leq kx$ for any $k\geq 1$ and $x\in (0,1)$, we have $\E[\tilde{n}] \leq n \E[D_n]$. 
Also, using $1-(1-x)^k \geq kx-k^2x^2/2$ for any $kx< 1$, $k\geq 1$ and  $x\in (0,1)$, 
\begin{eq}
 \E[1-(1-p_n)^{D_n}] &\geq \E[1-(1-p_n)^{D_n} \ind{p_nD_n <1}] \\
 &\geq p_n \E[D_n \ind{p_nD_n <1}] - \frac{p_n^2}{2}  \E[D_n^2 \ind{p_nD_n <1 }]\\
 & = p_n\E[D_n] -p_n \E[D_n \ind{p_nD_n \geq 1}] - \frac{p_n^2}{2}  \E[D_n^2 \ind{p_nD_n <1}].
\end{eq}
Using Assumption~\ref{assumption-supercrit}~(ii), $(D_n)_{n\geq 1}$ is uniformly integrable and thus $\E[D_n \ind{p_nD_n \geq 1}] = o(1)$, where in the last step we have used that $p_n \ll 1$.
For the third term, since $(D_n)_{n\geq 1}$ is uniformly integrable, we have that $(D_n)_{n\geq 1}$ is also tight. Thus, $p_nD_n \xrightarrow{\sss \PR} 0$.  Using the uniform integrability of $(D_n)_{n\geq 1}$ again together with $p_nD_n \ind{p_nD_n<1} \leq 1$ and $p_nD_n \xrightarrow{\sss \PR} 0$, we conclude that $\E[D_n\times (p_nD_n \ind{p_nD_n<1})] \to 0$. 
From \eqref{eq:expt-n-tilde-expr}, and Assumption~\ref{assumption-supercrit}~(ii), we now conclude that 
\begin{eq}\label{asymp-expt-n-tilde}
\E[\tilde{n}] = np_n(\mu+o(1)).
\end{eq}
Further, 
using standard concentration inequalities for sums of independent Bernoulli random variables \cite[(2.9), Theorem 2.8]{JLR00}, it follows that 
\begin{eq}\label{eq:asymp:tilde-n}
\PR(|\tilde{n} - \E[\tilde{n}] |> \log n \sqrt{np_n} )\leq 2\e^{-C(\log n)^2},
\end{eq}
for some constant $C>0$. 
In what follows, we will often use \eqref{asymp-expt-n-tilde} and \eqref{eq:asymp:tilde-n} to replace $\tilde{n}$ by $np_n\mu$.

Conditions (B1)--(B4) \cite{HJL16} are straightforward.
(B8) follows using $\max_{i\in [n]} \tilde{d}_i = \OP(n^{\alpha}p_n) = \oP(\tilde{n}\gamma_n)$.
To verify Conditions (B5)--(B7), we first obtain below the asymptotics of the mean-curve and then show that the processes $\Mtilde{\bld{S}}_n$, $\tilde{\bld{V}}_n$, $\Mtilde{\bld{A}}_n$ remain uniformly close to their expected curves. 
These are summarized in the following two propositions:
\begin{proposition}\label{prop:expt-supcrit} For any fixed $u> 0$, as $n\to\infty$, 
\begin{gather}
\sup_{t\leq u}\bigg|\frac{1}{np_n\mu\beta_n}\big(\E[\tS_n(0)]-\E[\tS_n(\beta_n t)]\big) - \hat{h}(t)\bigg|\to 0, \label{eq:expt-S}\\
\sup_{t\leq u}\bigg|\frac{1}{np_n\mu\beta_n}\big(\E[\tV_n(0)]-\E[\tV_n(\beta_n t)]\big) - \hat{g}(t)\bigg|\to 0, \label{eq:expt-V}\\
\sup_{t\leq u}\bigg|\frac{1}{np_n\mu\gamma_n}\E[\tA_n(\beta_n t)] - \psi(t)\bigg| \to 0.\label{eq:expt-A}
\end{gather}
\end{proposition}
\begin{proposition}\label{prop:supcrit-expt-concen}
For any fixed $u> 0$, as $n\to\infty$, all the terms
$\sup_{t\leq u}|\tS_n(\beta_n t)-\E[\tS_n(\beta_n t)] |$,  $\sup_{t\leq u}|\tV_n(\beta_n t)-\E[\tV_n(\beta_n t)]|$, and $\sup_{t\leq u}|\tA_n(\beta_n t) - \E[\tA_n(\beta_n t)]| $ are $ \oP(n p_n\beta_n) $ (and thus  $\oP(n p_n\gamma_n)$).
\end{proposition}
To prove Propositions~\ref{prop:expt-supcrit}~and~\ref{prop:supcrit-expt-concen}, we make crucial use of the following lemma: 
\begin{lemma} \label{lem:laplace-transform-estimate} For any $t>0$, as $n\to\infty$, 
\begin{align}
\E\bigg[\sum_{i\in [n]} \tilde{d_i} \e^{-t\beta_n\tilde{d}_i}\bigg] &= (1+o(1)) p_n \e^{-t\beta_n} \sum_{i\in [n]} d_i \e^{-t\beta_np_n d_i}, \\ \E\bigg[\sum_{i\in [n]} \e^{-t\beta_n\tilde{d}_i} \ind{\tilde{d}_i\geq 1}\bigg] &= (1+o(1)) \sum_{i\in [n]} \big( \e^{-t\beta_np_n d_i} - (1-p_n)^{d_i} \big). \label{eq:second-asymp}
\end{align}
\end{lemma}
\begin{proof}
Note that if $X\sim \mathrm{Bin}(m,p)$, then 
\begin{eq}
\E\big[X\e^{-sX}\big] = mp\e^{-s}(1-p+p\e^{-s})^{m-1}.
\end{eq}
Putting $m=d_i$, $p=p_n$, and $s = t\beta_n$, it follows that
\begin{eq} \label{super-crit-d-til-e}
\E\big[\tilde{d}_i \e^{-t\beta_n\tilde{d}_i}\big] &= d_ip_n \e^{-t\beta_n} \Big(1-p_n\big(1-\e^{-t\beta_n}\big)\Big)^{d_i -1} = (1+o(1)) d_ip_n \e^{-t\beta_n} (1-p_nt\beta_n)^{d_i}\\
& = (1+o(1)) d_ip_n \e^{-t\beta_n} \e^{-t\beta_np_n d_i}.
\end{eq}
To prove \eqref{eq:second-asymp}, note that $\E[\e^{-sX} \ind{X\geq 1}] = \E[\e^{-sX}] - \PR(X=0)$.
The proof of \eqref{eq:second-asymp} now follows similarly.
\end{proof}

\begin{proof}[Proof of Proposition~\ref{prop:expt-supcrit}]
Note that, by Lemma~\ref{lem:laplace-transform-estimate},
\begin{eq}\label{eq:S_n-supercrit-expt}
\E\big[\tilde{S}_n( \beta_n t)\big] &= \E\bigg[\sum_{i\in [n]} \Mtilde{d}_i\e^{-t\beta_n \Mtilde{d}_i} \bigg]= (1+o(1))\ell_n p_n \e^{-t\beta_n} \E\big[\e^{-t\beta_n p_n D_n^{\star}}\big], \\ 
\E\big[\tilde{V}_n( \beta_n t)\big] &= \E\bigg[\sum_{i\in [n]} \e^{-t\beta_n \Mtilde{d}_i}\ind{\tilde{d}_i\geq 1} \bigg] =  (1+o(1))n\big( \E\big[\e^{-t\beta_n p_n D_n} - (1-p_n)^{D_n} \big]\big),
\end{eq}where $D_n^{\star}$ has a size-biased distribution with the sizes being $(d_i/\ell_n)_{i\in [n]}$, and $D_n$ is the degree of a vertex chosen uniformly at random from $[n]$. 
By the convergence of $\E[D_n]$ in Assumption~\ref{assumption1}, 
\begin{eq}
\E[\tV_n(0)]-\E[\tV_n(\beta_n t)] =(1+o(1))n\E\big[1 - \e^{-t\beta_n p_n D_n}\big] = (1+o(1)) tn\beta_np_n \mu,
\end{eq}
where the asymptotics of $n\E\big[1 - \e^{-t\beta_n p_n D_n}\big]$ follows using identical arguments as \eqref{eq:expt-n-tilde-expr}.
Further, by using \eqref{eq:asymp-laplace},
\begin{eq}\label{laplace-exp-alpha-estimate}
\E[\tS_n(0)]-\E[\tS_n(\beta_n t)] &= (1+o(1))\ell_np_n\E\big[1- \e^{-t\beta_n}\e^{-t\beta_n p_n D_n^{\star}}\big] \\
&= (1+o(1))\ell_np_n\E[1- (1-t\beta_n+o(\beta_n))\e^{-t\beta_n p_n D_n^{\star}}]\\
&=  (1+o(1))\ell_np_n\big(\E[1- \e^{-t\beta_n p_n D_n^{\star}}]+t\beta_n +o(\beta_n) \big)\\
&= (1+o(1))n\mu p_n \beta_n (\kappa t^{\tau-2}+t+o(1)).
\end{eq}
Thus, \eqref{eq:expt-S} and \eqref{eq:expt-V} follows. 
Moreover, $L_n(t)$ is a pure death process, where $L_n(0) = \sum_{i\in [n]} \Mtilde{d}_i$,  and the jumps occur at rate $L_n(t)$, and at each jump $L_n(t)$ decreases by $2$. 
Therefore, $\E[L_n(t)] = \E[L_n(0)]\e^{-2t}$ and consequently, by \eqref{identities-A-V-S} and \eqref{super-crit-d-til-e},
\begin{eq}
\E[\tA(\beta_n t)] &= \ell_np_n\big(\e^{-2\beta_nt} - \e^{-\beta_n t}\E\big[\e^{-tp_n\beta_nD_n^{\star}}\big]\big)+o(n\beta_np_n)\\
& = n\mu p_n \gamma_n(\kappa t^{\tau-2}-t)+o(n\beta_np_n).
\end{eq}
Thus the proof follows.
\end{proof}
\begin{proof}[Proof of Proposition~\ref{prop:supcrit-expt-concen}]
Let us consider $\tS_n$ only; the other inequalities follow using identical arguments. 
We will show that 
\begin{eq}\label{expt-square-s-n-tilde}
\E\Big[\sup_{t\leq u\beta_n} |\tS_n(t) - \E[\tS_n(t)]|^2 \Big] = o((np_n\beta_n)^2),
\end{eq}then an application of Markov's inequality completes the proof.
To prove \eqref{expt-square-s-n-tilde}, we will use \cite[Lemma 5.15]{HJL16}, which says that 
\begin{eq}\label{eq:expectation-sup-martingale}
\E\Big[\sup_{t\leq u\beta_n} |\tS_n(t) - \E[\tS_n(t)]| ^2 \Big] \leq C \E \bigg[ \sum_{i\in [n]} \Mtilde{d}_i^2 \min\{\tilde{d}_iu\beta_n,1\}\big)\bigg].
\end{eq}
Although, \cite[Lemma 5.15]{HJL16} was stated under Assumptions (A1)-(A4) of this paper, this particular proof does not use this assumption. The proof only uses \cite[Lemma 4.2]{HJL16}. 
Indeed, the deductions in (5.62)--(5.65) of \cite{HJL16} does not require any assumption on the degrees.
We skip redoing the proof of \eqref{eq:expectation-sup-martingale} here.
Using the fact that $1-\e^{-x} \geq (1\wedge x)/3$ in \eqref{eq:expectation-sup-martingale}, it follows that 
\begin{eq}\label{bound-expt-sup-second-moment}
\E\Big[\sup_{t\leq u\beta_n} |\tS_n(t) - \E[\tS_n(t)]| ^2 \Big] \leq C \E \bigg[ \sum_{i\in [n]} \Mtilde{d}_i^2 \big(1-\e^{-u\beta_n \Mtilde{d}_i} \big)\bigg]. 
\end{eq}

\noindent Now, using standard concentration inequalities for tails of binomial distributions \cite[Theorem~2.1]{JLR00}, for any $i\in [n]$,
\begin{eq}
\PR(\Mtilde{d}_i > 2d_1 p_n ) \leq C \e^{-C d_1 p_n} = C\e^{-C n^{\rho} \lambda_n},
\end{eq}
where $\lambda_n = p_nn^{\eta} \to\infty$. 
Therefore $\max_{i\in [n]}\Mtilde{d}_i \leq 2d_1 p_n$, almost surely.
Thus, 
\begin{eq}\label{supercrit-expt-bound-simple}
\frac{1}{(\ell_np_n\beta_n)^2}\E\Big[\sup_{t\leq u\beta_n} |\tS_n(t) - \E[\tS_n(t)]|^2 \Big] &\leq \frac{C 2d_1p_n }{(\ell_np_n\beta_n)^2} \E \bigg[ \sum_{i\in [n]} \Mtilde{d}_i \big(1-\e^{-u\beta_n \Mtilde{d}_i} \big)\bigg]\\
&\leq \frac{C 2d_1p_n \ell_np_n}{(\ell_np_n\beta_n)^2} \E \big[ 1-  \e^{-u\beta_np_nD_n^\star}\big],
\end{eq}
where the last step follows using \eqref{eq:S_n-supercrit-expt}.
The final term in \eqref{supercrit-expt-bound-simple} can be shown to be $O(\beta_n)$ using identical computations as  \eqref{laplace-exp-alpha-estimate}.
Thus,
\begin{eq}
\frac{1}{(\ell_np_n\beta_n)^2}\E\Big[\sup_{t\leq u\beta_n} |\tS_n(t) - \E[\tS_n(t)]|^2 \Big] \leq \frac{C 2d_1p_n \ell_np_n \beta_n}{(\ell_np_n\beta_n)^2} = O(d_1/n\beta_n) = O\big(\lambda_n^{-\frac{\tau-2}{3-\tau}}\big) = o(1),
\end{eq}
since $\lambda_n\to \infty$, as $n\to\infty$. Thus the proof follows.
\end{proof}
\begin{proof}[Proof of Theorem~\ref{thm:barely-supercrit}] 
The proof follows by applying \cite[Theorem 5.4]{HJL16}.
Propositions~\ref{prop:expt-supcrit},~\ref{prop:supcrit-expt-concen} verify conditions (B5)--(B7) in \cite{HJL16}, and the rest of the conditions are straightforward to verify.
\end{proof}

\small{\bibliographystyle{apa}
\bibliography{ultradensebib}}

\appendix 

\section{Appendix}
\subsection{Path counting}\label{sec:appendix-path-counting}
Recall the notation from in Section~\ref{sec:large-early}. We complete the proof of \eqref{path-count-ub-main} using path-counting techniques for configuration models from \cite[Lemma 5.1]{J09b}. 
  Let $\mathcal{A}_l(v,k)$ denote the event that there exists a path of length $l$ from $v$ to $k$ in the graph $\cG^{\sss K}_n$. Also, let $P_l$ denote the number of paths of length $l$. 
  Notice that 
\begin{equation}\label{path-count-split}
 \begin{split}
  & \tilde{\E}\bigg[\sum_{k\in [n]}(\tilde{d}_k-1)\ind{V_n^{*,{\sss K}}\leadsto k}\Big\vert V_n^{*,{\sss K}} = v\bigg]\\
  &\leq \td_{v}-1+  \sum_{l=1}^{(\log n)^2}\sum_{k\in [n]}(\tilde{d}_{k}-1)\probc{\mathcal{A}_l(v,k)} + \max_{k\in [n]} (\td_k-1)\times n \sum_{l\geq (\log n)^2}\tilde{\E}[P_l]. 
 \end{split}
 \end{equation}
Let $\fI_l(v,k)$ denote the collection of $\xx = (x_i)_{0\leq i\leq l}$ such that $x_0 = v$, $x_l = k$ and the $x_i$'s are distinct.
Then, an identical argument to the proof of \cite[Lemma 5.1]{J09b} shows that, for $l= o(n^{2\rho})$, the expected number of paths of length exactly $l$ starting from vertex $v$ and ending at $k$ is given by
\begin{equation}\label{eq:ineq-path}
 \sum_{\xx \in \fI_l(v,k)}\frac{d'_{x_0}d'_{x_l} \prod_{i=1}^{l-1} d_{x_i}'(d_{x_i}'-1) }{(\ell_n'-1)\cdots (\ell_n'-2l+1)}\leq \frac{d'_v \ell'_n}{\ell'_n-2l+3}(\nu_n^{\sss K})^{l-1} = \Big(1+\OP\Big(\frac{l}{\tell_n}\Big)\Big)d'_v(\nu_n^{\sss K})^{l-1},
\end{equation}where $\ell_n'=\sum_{i\in [n]} d_i'$. Recall that $\ell_n' = \tell_n(1+\oP(1))$. 
  Thus, the second term in \eqref{path-count-split} is at most 
\begin{equation}\label{sus-simpl-1}
 \begin{split}
  & \sum_{l = 1}^{(\log n)^2}\sum_{k\in [n]}(\tilde{d}_k-1)\sum_{x_i\neq x_j, \forall i\neq j} \frac{d_v'd_{k}'\prod_{i=1}^{l-1}d_{x_i}'(d_{x_i}'-1)}{(\ell_n'-1)\cdots (\ell_n'-2l+1)} \\
  &\leq  (1+\oP(1))\td_v \bigg(\frac{1}{\tell_n} \sum_{k\in [n]}d_k'(\td_k-1)\bigg)\sum_{l=1}^{\infty}(\nu_n^{\sss K})^{l-1}  \\
  &\leq  (1+\oP(1))\td_v \bigg(\frac{1}{\tell_n} \sum_{k> K}\td_k(\td_k-1)\bigg)\sum_{l=1}^{\infty}(\nu_n^{\sss K})^{l-1}
  \leq (1+\oP(1)) \frac{\td_v \E[\td_{V_n^{*,{\sss K}}}-1]}{1-\nu_n^{\sss K}},
 \end{split}
 \end{equation}where in the one-but-last step we have used $d_i'=0$ for $i\leq K$,  $d_i'\leq \td_i$ for $i>K$ and $\nu_n^{\sss K}<1$. The third term in \eqref{path-count-split} is $\oP(1)$ uniformly over $v$ by \eqref{diam-comput}. 
 Thus the proof of  \eqref{path-count-ub-main} follows. 
 \qed 
 \subsection{Convergence of process tracking surplus}\label{appendix-surplus} 
 In this section, we complete the proof of Lemma~\ref{lem:surp:poisson-conv}. 
 We first argue that, for any fixed $u>0$, 
\begin{equation}\label{1d-tight-surplus}
\big(\bar{N}_n^\lambda (u)\big)_{n\geq 1} \text{ is tight in }\R_+.
\end{equation}
Fix $\varepsilon>0$. 
Recall the asymptotics from Lemma~\ref{lem:perc-degrees} which will be used throughout the proof. 
Also, recall that $\tPR$ and $\tE$ respectively denote the conditional probability and expectation conditionally on $(\tilde{d}_i)_{i\in [n]}$. 
To simplify writing, when we write bounds on the conditionals probabilities $\tPR$ and $\tE$, we always implicitly assume that the bounds hold with high probability.
Recall from \eqref{eqn:intensity} that the compensator of $\bar{\mathbf{N}}_n^\lambda$ is approximately proportional to $\mathrm{refl}(\bar{\mathbf{S}}_n)  \xrightarrow{\sss d} \mathrm{refl} (\mathbf{S}^\lambda_\infty )$, where the distributional convergence follows using Theorem~\ref{thm::convegence::exploration_process} and the continuity of the reflection map (see \cite[Lemma 13.5.1]{W02}).
We write $A_i$ denote the number of active half-edges after stage $i$ while implementing Algorithm~\ref{algo-expl}.
Thus $n^{-\rho} A_{\floor{tn^{\rho}}} =\mathrm{refl}(\bar{S}_n(t))$.
Using the fact that the supremum of a process is continuous with respect to the Skorohod $J_1$-topology \cite[Theorem 13.4.1]{W02}, we can choose $K\geq 1$ large enough so that for all sufficiently large $n$
\begin{eq}\label{eq:tightness-supremum}
 \tPR\Big(\sup_{i\leq \floor{un^{\rho}}} A_{i}>K n^{\rho}\Big) <\varepsilon.
\end{eq}
Fix times $0<l_1<\dots<l_m\leq \floor{un^\rho}$, and let $\cA(l_1,\dots,l_m)$
denote the event that the surplus edges appear at times $l_1,\dots,l_m$ and $A_{l_{j}-1}\leq K n^\rho$ for all $j \in [m]$. Then,  
\begin{eq}\label{eq:ub-surplus-restricted-suprema}
&\tPR\bigg(\sum_{i=2}^{\floor{un^\rho}}\xi_i \geq m, \text{ and }\sup_{i\leq \floor{un^\rho}} A_{i}\leq K n^\rho\bigg) \leq \sum_{0<l_1<\dots<l_m\leq \floor{un^\rho}} \tPR(\cA(l_1,\dots,l_m)) \\
&\leq \sum_{0<l_1<\dots<l_m\leq \floor{un^\rho}} \tE \big[\tPR(\text{surplus created at }l_m\vert \sF_{l_m-1})\ind{A_{l_m - 1} \leq K n^\rho} \mathbbm{1}_{\cA(l_1,\dots,l_{m-1})}\big] 
\\
&\leq \frac{Kn^\rho}{\tell_n -2\floor{un^\rho} +1} \sum_{0<l_1<\dots<l_m\leq \floor{un^{\rho}}} \tPR(\cA(l_1,\dots,l_{m-1})).
\end{eq}
Continuing the iteration in the last step, it follows that with high probability 
\begin{eq}\label{ub-tightness-sup-bounded}
\tPR\bigg(\sum_{i=2}^{\floor{un^\rho}}\xi_i \geq m, \text{ and }\sup_{i\leq \floor{un^\rho}} A_{i}\leq K n^\rho\bigg)  \leq (1+o(1))\Big(\frac{K n^\rho}{\tell_n}\Big)^m \frac{(\floor{un^\rho})_m}{m!},
\end{eq}where $(n)_m = n(n-1)\dots (n-m+1)$. The last term in \eqref{ub-tightness-sup-bounded} tends to zero in the iterated limit $\lim_{m\to\infty}\limsup_{n\to\infty}$. 
An application of \eqref{eq:tightness-supremum} now yields~\eqref{1d-tight-surplus}.

Next, let $\mathbf{S}_n'$ be the process obtained by discarding the points where a surplus edge was added. More precisely, if $\zeta_ l = S_n(l) - S_n(l-1)$, then  we can define $S_n'(l) = S_n'(l-1) +\zeta'_{l}$, where 
\begin{equation}\label{eq:defn-zetal}
\zeta'_l = \zeta_{k_l}, \quad \text{with} \quad k_l = \inf\{j>k_{l-1}: \zeta_j \neq -2\}, \ k_0 =0.
\end{equation}
Let $\bar{S}_n'(t) = n^{-\rho} S_n'(\floor{tn^\rho})$.
Also, let $\mathrm{d}_{J_1, T}$ denote the metric for the  Skorohod $J_1$-topology on $\mathbb{D}([0,T], \R)$.
We claim that, for any $T>0$ and $\varepsilon>0$,
\begin{equation}\label{eq:Sn-Snprime-close}
\lim_{n\to\infty}\PR\big( \mathrm{d}_{J_1,T} (\bar{\mathbf{S}}_n',\bar{\mathbf{S}}_n) > \varepsilon\big) = 0.
\end{equation}
First, let $1\leq l_1<\dots<l_K\leq \floor{Tn^\rho}$ denote the times where the  surplus edges have occurred. 
Also, let $\cA$ be the good event that $l_{j}+1<l_{j+1}$ for all $j\leq K$, i.e., none of the surplus edges occur in consecutive steps.
Note that
\begin{eq}\label{eq:good-event-surplus}
&\tPR\Big(\cA^c\bigcap \Big\{\sup_{i\leq \floor{Tn^\rho}} A_{i}\leq K n^\rho\Big\} \Big) \leq Tn^\rho \Big(\frac{Kn^\rho}{\tell_n}\Big)^2 = O(n^{-\rho}),
\end{eq}and thus using \eqref{eq:tightness-supremum}, $\PR(\cA^c) \to 0$. 
We now restrict ourselves on $\cA$. 
Putting $l_0 =0$ and $l_{K+1} = \floor{Tn^{\rho}}+1$, let 
\begin{eq}\label{eq:time-change-definition}
\Lambda_n(l) = 
\begin{cases}
l+j-1 \quad &\text{for } l_{j-1}<l<l_j,\\
l_j+j-1 \quad &\text{for } l = l_j - 0.5,\\
l_j+j \quad &\text{for } l = l_j.
\end{cases}
\end{eq}
$\Lambda_n(t)$ is obtained by linearly interpolating between the values specified by \eqref{eq:time-change-definition}. 
Also, note that the definition of $\Lambda_n$ works well on $\cA$, and on $\cA^c$ we define $\Lambda_n(t) =t$.
Using \eqref{1d-tight-surplus} and  \eqref{eq:good-event-surplus}, it immediately follows that 
\begin{eq}\label{eq:modified-J1-1}
\sup_{l\leq Tn^\rho} |\Lambda_n(l) - l | = \oP(n^\rho). 
\end{eq}
Moreover, the occurrence of each surplus edge causes $|S_n'(l) - S_n(\Lambda_n(l))|$ to increase by at most 2, so that  
\begin{eq}\label{eq:modified-J1-2}
\sup_{l\leq Tn^\rho} 
|S_n'(l) - S_n(\Lambda_n(l))| = \oP(n^\rho). 
\end{eq}
Now, \eqref{eq:Sn-Snprime-close} follows by combining \eqref{eq:modified-J1-1} and \eqref{eq:modified-J1-2}. 
We now proceed to complete the proof of Lemma~\ref{lem:surp:poisson-conv}.
Let set up some notation for the rest of the proof. Fix $T>0$, $k\geq 0$ and let $\mathrm{Surp}_T = \{l_1,\dots ,l_k\}$, where $1\leq l_1 <l_2<\dots<l_k \leq \floor{Tn^\rho}+k$. Let $(z_l)_{l\leq \floor{Tn^\rho}+k}$ be a sequence of integers such that $z_{l_i} = -2$ and  $z_l\geq -1$ for $l\notin \{l_1,\dots,l_k\}$. 
Thus $(z_l)_{l\leq \floor{Tn^\rho}+k }$ represents a sample path of $S_n$ which has explored $k$ surplus edges, and $\mathrm{Surp}_T $ is the set of times when surplus edges are found. 
Next, $(z_l')_{l\leq \floor{Tn^\rho}}$ denote the sequence obtained from $(z_l)_{l\leq \floor{Tn^\rho}+k}$ by deleting the $-2$'s. 
Thus, $(z_l')_{l\leq \floor{Tn^\rho}}$ corresponds to a sample path of $S_n'$. 
Recall that $\zeta_l = S_n(l) - S_n(l-1)$. Let $\omega_n\to\infty$ sufficiently slowly. 
Thus, 
\begin{eq}\label{sample-path-prob-conditional}
&\tPR(N_n^{\lambda}(\floor{Tn^\rho}+k) = k \vert (S_n'(l))_{l\leq \floor{Tn^\rho}}= (z_l')_{l\leq \floor{Tn^\rho}}, N_n^{\lambda}(\floor{Tn^\rho}+k)  \leq \omega_n) \\
& = \sum_{1\leq l_1<\dots<l_k\leq Tn^\rho}\PR\bigg(\text{surplus occurs only at times }l_1, \dots,l_k \bigg\vert \substack{\big(S_n'(l)\big)_{l\leq Tn^\rho} = (z_l')_{l\leq Tn^\rho},\\ N_n^\lambda(\floor{T n^\rho}+k)\leq \omega_n}\bigg)\\
& = \sum_{1\leq l_1<\dots<l_k\leq Tn^\rho} \frac{\tPR(\zeta_l = z_{l}, \ \text{ for all } 1\leq l\leq \floor{Tn^\rho}+k)}{\tPR(\big(S_n'(l)\big)_{l\leq Tn^\rho} = (z_l')_{l\leq Tn^\rho}, N_n^\lambda(\floor{T n^\rho}+k)\leq \omega_n)}.
\end{eq}
Define $m_1= \{i\in [n]:d_i = z_1+2\}$, and for $l\notin\mathrm{Surp}_T$, we denote 
$m_l = \#\{i\in [n]:d_i = z_l+2\} - \#\{j< l:z_j = z_l\}.$
Next, let $a_l$ denote the number of active half-edges at time $l$ when the exploration process takes the path $(z_l)_{l\leq \floor{Tn^\rho}+k}$, and $a_l' = S_n'(l)-\min_{j<l} S_n'(j)$. 
Now, 
\begin{eq}\label{exact-prob-sample-path}
&\tPR(\zeta_l = z_l, \ \forall l\leq \floor{Tn^\rho}+k) = \frac{\prod_{l\notin \mathrm{Surp}_T} m_l \times \prod_{j=1}^k(a_{l_j-1}-1)}{(\tell_n - 1)(\tell_n-3)\dots (\tell_n - 2\floor{Tn^\rho}-2k+1)} \\
& = \frac{\prod_{l\notin \mathrm{Surp}_T} m_l \times \prod_{j=1}^k(a_{l_j-1}-1)}{(\tell_n - 1)\dots (\tell_n - 2\floor{Tn^\rho}+1)} \times  (1+\oP(1))\prod_{j=1}^k\frac{a_{l_j-1}'}{\tell_n^k},
\end{eq}where the $\oP(1)$ term above is uniform over $k\leq \omega_n = \log n$. 
Thus, 
\begin{eq}\label{sample-path-prob-conditional-2}
\eqref{sample-path-prob-conditional}& =(1+o(1))\frac{\sum_{1\leq l_1<\dots<l_k\leq \floor{Tn^\rho}+k}\prod_{j=1}^k\frac{a_{l_j-1}'}{\tell_n^k} }{\sum_{r=0}^{\omega_n}\sum_{1\leq l_1<\dots<l_r\leq  \floor{Tn^\rho}+r}\prod_{j=1}^r\frac{a_{l_j-1}'}{\tell_n^r}} =: (1+o(1)) \frac{\beta_{n,k}}{\sum_{r=0}^\infty \beta_{n,r}},
\end{eq}where $\beta_{n,r} = 0$ for $r>\omega_n$. 
We write $\tilde{\mu} = \lambda\mu^2/ \sum_i \theta_i^2$, so that 
$\tilde{\ell}_n= \tilde{\mu}n^{2\rho}(1+\oP(1))$.
Now, using $\mathrm{refl}(\bar{\mathbf{S}}_n')  \xrightarrow{\sss d} \mathrm{refl} (\mathbf{S}^\lambda_\infty )$, it follows that 
\begin{eq}\label{eq:beta-S-joint-convergence}
\Big((\beta_{n,r})_{r\geq 0}, (\bar{S}_n'(u))_{u\leq T} \Big) \dto \bigg(\Big(\frac{1}{r!}\Big(\frac{1}{\tilde{\mu}} \int_0^T \refl{\iS(u)}\dif u\Big)^r\Big)_{r\geq 0}, (\iS(u))_{u\leq T} \bigg),
\end{eq}
where the convergence of $(\beta_{n,r})_{r\geq 0}$ holds with respect to the product topology on $\R^\infty$.
Next, let us ensure that $\sum_{r=0}^\infty \beta_{n,r}$ in \eqref{sample-path-prob-conditional} converges to the desired quantity. 
To this end, consider a probability space where the convergence of \eqref{eq:beta-S-joint-convergence} holds almost surely. 
On this space, $\sup_{l\leq Tn^\rho+k} \refl{S_n'(l)}\leq 2 (\sup_{l\leq Tn^\rho+k} S_n'(l) +\omega_n)=: X_n(T)$, and thus
\begin{eq}
\beta_{n,r} \leq \frac{(Tn^\rho+\omega_n)^r}{r!} \frac{X_n(T)^r}{\tell_n^r}. 
\end{eq}
Since $n^{-\rho}\sup_{l\leq Tn^{\rho}+k} S_n'(l)$ converges, an application of Dominated Convergence Theorem yields that 
\begin{eq}\label{eq:normalization-surplus-limit}
\sum_{r\geq 0} \beta_{n,r} \asto \sum_{r\geq 0}\frac{1}{r!}\Big(\frac{1}{\tilde{\mu}} \int_0^T \refl{\iS(u)}\dif u\Big)^r = \exp \bigg(\frac{1}{\tilde{\mu}} \int_0^T \refl{\iS(u)}\dif u\bigg).
\end{eq}
Next, for bounded continuous functions $\phi_1: \mathbb{D}([0,T], \R) \to \R$ and $\phi_2: \N \to \R$, 
\begin{eq}
&\E \big[\phi_1\big(\big(\bar{S}_n'(u)\big)_{u\leq T}\big)\phi_2(\bar{N}_n^{\lambda}(T)) \big] \\
&= \E \big[\phi_1\big(\big(\bar{S}_n'(u)\big)_{u\leq T}\big)\phi_2(\bar{N}_n^{\lambda}(T)) \ind{N_n^{\lambda}(\floor{Tn^{\rho}} + k)  \leq \omega_n}\big] +o(1)\\
& = o(1) + \E \bigg[\phi_1\big(\big(\bar{S}_n'(u)\big)_{u\leq T}\big)\ind{N_n^{\lambda}(\floor{Tn^{\rho}} + k)   \leq \omega_n}\times (1+o(1)) \frac{\sum_{k\geq 0}\phi_2(k) \beta_{n,k}}{\sum_{r\geq 0} \beta_{n,r}}\bigg] \\ 
& = o(1) + \E \bigg[\phi_1\big(\big(\bar{S}_n'(u)\big)_{u\leq T}\big)\times \frac{\sum_{k\geq 0}\phi_2(k) \beta_{n,k}}{\sum_{r\geq 0} \beta_{n,r}}\bigg] \to \E \big[\phi_1\big(\big(\iS(u)\big)_{u\leq T}\big)\phi_2(N^\lambda(T)) \big],
\end{eq}where $N^\lambda(T)$, conditionally on $(\iS(u))_{u\leq T}$, is distributed as Poisson$(\frac{1}{\tilde{\mu}}\int_0^T\refl{\iS(u)} \dif u)$.
We have used \eqref{1d-tight-surplus} in the third step, and the final step follows by combining \eqref{eq:beta-S-joint-convergence} and \eqref{eq:normalization-surplus-limit}.
Hence, we have shown that, for any $T>0$, 
\begin{eq}\label{total-limit-law}
\Big(\big(\bar{S}_n'(u)\big)_{u\leq T},\bar{N}_n^{\lambda}(T)\Big)\dto \Big(\big(\iS(u)\big)_{u\leq T},N^\lambda(T)\Big).
\end{eq}
Next, let $U_1^n<U_2^n<...$ denote the location of surplus edges in the process $S_n$. 
Then, using \eqref{exact-prob-sample-path} yields
\begin{eq}\label{location-law}
&\tPR\Big(U_j^n = l_j, \text{ for all }j\in [k]\Big\vert \big(\bar{S}_n'(u)\big)_{u\leq T},\bar{N}_n^{\lambda}(T) = k\Big)\\
& = (1+o(1)) \frac{\frac{1}{\tell_n^k}\prod_{j=1}^k (A_{l_j}-1)}{\sum_{1\leq l_1'<\dots < l_k'\leq \floor{Tn^{\rho}}+k}\frac{1}{\tell_n^k}\prod_{j=1}^k (A_{l_j'}-1)}.
\end{eq}
From this, it can be seen that the law of $n^{-\rho} (U_j^n)_{j\in [k]}$, conditionally on $(\bar{S}_n'(u))_{u\leq T}$, and $\bar{N}_n^{\lambda}(T) = k$, converges to the order-statistics of $k$ i.i.d random variables with density $\frac{\ind{u\in [0,T]} \refl{\iS (u)}}{\int_0^T\refl{\iS (u)}\dif u}$. This shows that the location of the occurrence of surplus edges, conditionally on $(\bar{S}_n'(u))_{u\leq T}$, converges in distribution to the location of the points of the Poisson process \eqref{defn::counting-process} on $[0,T]$ conditionally on $\big(\iS(u)\big)_{u\leq T}$. 
Convergence of the total number of surplus edges created, conditionally on $(\bar{S}_n'(u))_{u\leq T}$, is given by \eqref{total-limit-law}.
Thus combining \eqref{total-limit-law} and \eqref{location-law}, it follows that 
\begin{eq}\label{process-counting-limit-law}
\Big(\big(\bar{S}_n'(u)\big)_{u\leq T},\big(\bar{N}_n^{\lambda}(u)\big)_{u\leq T}\Big)\dto \Big(\big(\iS(u)\big)_{u\leq T},\big(N^\lambda(u)\big)_{u\leq T}\Big).
\end{eq}
Now, an application of \eqref{eq:Sn-Snprime-close} completes the proof of Lemma~\ref{lem:surp:poisson-conv}.
\qed

\subsection{Tightness of component sizes and surplus} \label{appendix-surp-unot}
In this section, we prove Lemma~\ref{lem-surp-u-0}. 
Let $V_n^*$ denote a vertex chosen in a size-biased manner with sizes being $(\tilde{d}_i)_{i\in [n]}$, independently of the graph $\mathrm{CM}_n(\boldsymbol{d})$.
 Let $\mathscr{C}(V_n^*)$ denote the component containing $V_n^*$,  $D(V_n^*) = \sum_{k\in \mathscr{C}(V_n^*)} \tilde{d}_k$, and $D_i = \sum_{k\in \sCi} \tilde{d}_k$.
 Since component sizes corresponding to the components having one vertex and no edges is zero by our convention, $|\sCi| \leq D_i$ for all $i$. 
Thus, it is enough to show that, for any $\varepsilon>0$,  
\begin{equation}
 \tilde{\PR}\bigg(\sum_{i: D_i\leq \delta n^{\rho} }D_i\times \surp{\mathscr{C}_{\sss(i)}}> \varepsilon n^{\rho}\bigg)\pto 0,
 \end{equation}in the iterated limit $\lim_{\delta\to 0}\limsup_{n\to\infty}$.
The following estimate will be our crucial ingredient. 
We first prove Lemma~\ref{lem-surp-u-0} using Lemma~\ref{lem:sp-cv-n}, and the proof of Lemma~\ref{lem:sp-cv-n} will come after that.
\begin{lemma} \label{lem:sp-cv-n}
Assume that $\lambda<1$.
 Let $\delta_k=\delta k^{-0.12}$. Then, for $\delta > 0$ sufficiently small, with high probability,
\begin{equation}
 \probc{\surp{\mathscr{C}(V_n^*)}\geq K,D(V_n^*)\in (\delta_K n^\rho,2\delta_K n^{\rho})}\leq \frac{C\sqrt{\delta}}{n^{\rho}K^{1.1}},
\end{equation}
 where $C$ is a fixed constant independent of $n,\delta, K$. 
 \end{lemma}
\begin{proof}[Proof of Lemma~\ref{lem-surp-u-0} using Lemma~\ref{lem:sp-cv-n}]
First, let us consider
the case $\lambda<1$. 
Fix any $\varepsilon, \delta >0$. Note that
\begin{equation}\label{eq:u-0-calc}
 \begin{split}
  &\tilde{\PR}\bigg( \sum_{D_i\leq \delta n^{\rho}} D_i\surp{\mathscr{C}_{\sss (i)}}> \varepsilon n^{\rho} \bigg)\leq \frac{1}{\varepsilon n^{\rho}}\tilde{\E} \bigg[\sum_{i=1}^{\infty}D_i \surp{\mathscr{C}_{\sss (i)}} \1_{\{ |D_i\leq \delta n^{\rho}\}} \bigg]
  \\&= \frac{\tilde{\ell}_n}{\varepsilon n^{\rho}}\exptc{\mathrm{SP}(\mathscr{C}(V_n^*))\1_{\{ |\mathscr{C}(V_n^*)|\leq \delta n^{\rho}\}}}\\
  &= \frac{\tilde{\ell}_n}{\varepsilon n^{\rho}}\sum_{k=1}^{\infty}\sum_{i\geq \log_2(1/(k^{0.12}\delta))}\tilde{\PR}\big(\mathrm{SP}(\mathscr{C}(V_n^*))\geq k, |\mathscr{C}(V_n^*)|\in (2^{-(i+1)}k^{-0.12}n^\rho, 2^{-i}k^{-0.12}n^\rho] \big)\\
  &\leq \frac{C}{\varepsilon} \sum_{k=1}^{\infty}\frac{1}{k^{1.1}}\sum_{i\geq \log_2(1/(k^{0.12}\delta))} 2^{-i/2} \leq \frac{C}{\varepsilon}\sum_{k=1}^{\infty}\frac{\sqrt{\delta}}{k^{1.04}}  =O(\sqrt{\delta}/\varepsilon),
 \end{split}
\end{equation} where the last-but-second step follows from Lemma~\ref{lem:sp-cv-n}, and the inequality holds with high probability.
 The proof of Lemma~\ref{lem-surp-u-0} now follows for the $\lambda <1$ case. 

 Now consider the case $\lambda > 1$. 
 Fix a large integer $R\geq 1$ such that $\lambda \sum_{i>R} \theta_i^2 <1$. This can be done because $\boldsymbol{\theta}\in \ell^{2}_{\shortarrow}$. Using \eqref{eq:CgeqT2}, for any $\delta_0>0 $, it is possible to choose $T>0$ such that 
\begin{equation}\label{eq:early-expl}
 \limsup_{n\to\infty}\prob{\text{all the vertices }1,\dots,R \text{ are explored within time }Tn^{\rho} }> 1-\delta_0.
\end{equation} 
Let $T_e$ denote the first time after $Tn^{\rho}$ when we finish exploring a component. By Theorem~\ref{thm::convegence::exploration_process}, $(n^{-\rho}T_e)_{n\geq 1}$ is a tight sequence. Let $\mathcal{G}^*_T$ denote the graph obtained by removing the components explored up to time $T_e$. Then, $\mathcal{G}^*_T$ is again a configuration model conditioned on its degrees. Let $\nu_n^*$ denote the value of the criticality parameter for $\mathcal{G}^*$. 
Then using \eqref{eqn:nu-K} and the fact that $\lambda \sum_{i>R} \theta_i^2 <1$, $\nu_n^*< 1-\varepsilon_0$ with high probability for some $\varepsilon_0>0$.
Thus, if $\mathscr{C}_{\sss(i)}^*$ denotes the $i\text{-th}$ largest component of $\mathcal{G}_T^*$, then the argument for $\lambda<1$ yields
\begin{equation}\label{surplus-positivelamba1} \lim_{T\to\infty}\lim_{\delta\to 0}\limsup_{n\to\infty}\PR\bigg(\sum_{i: |\mathscr{C}_{(i)}^*|\leq \delta n^\rho }|\mathscr{C}_{\sss(i)}^*|\times \surp{\mathscr{C}_{\sss(i)}^*}> \varepsilon n^{\rho}\bigg)=0.
\end{equation} To conclude the proof for the whole graph (with $\lambda >1$), let $$\mathcal{K}_n^T:=\{i:|\mathscr{C}_{\sss(i)}|\leq \delta n^{\rho}, |\mathscr{C}_{\sss(i)}| \text{ is explored before the time }T_e  \}.$$ Note that
\begin{equation}
 \begin{split}
  \sum_{i \in \mathcal{K}_n^T}|\mathscr{C}_{\sss (i)}|\times \mathrm{SP}(\mathscr{C}_{\sss (i)})&\leq \Big( \sum_{i\in \mathcal{K}_n^T}|\mathscr{C}_{\sss (i)}|^2\Big)^{1/2}\times \Big(\sum_{i\in \mathcal{K}_n}\mathrm{SP}(\mathscr{C}_{\sss (i)})^2 \Big)^{1/2}\\
  &\leq  \bigg( \sum_{|\mathscr{C}_{\sss (i)}|\leq \delta n^\rho}|\mathscr{C}_{\sss (i)}|^2\bigg)^{1/2}\times \mathrm{SP}(T_e),
 \end{split}
\end{equation}where $\mathrm{SP}(t)$ is the number of surplus edges explored up to time $tn^\rho$ and we have used the fact that $\sum_{i\in\mathcal{K}_n^T}\mathrm{SP}(\mathscr{C}_{\sss (i)})^2\leq (\sum_{i\in\mathcal{K}_n^T}\mathrm{SP}(\mathscr{C}_{\sss (i)}))^2\leq \mathrm{SP}(T_e)^2$. From Lemma~\ref{lem:surp:poisson-conv} and Proposition~\ref{prop:l2-tight} we can conclude that for any $T>0$,
\begin{equation}\label{surplus-positivelamba2}
 \lim_{\delta\to 0}\limsup_{n\to\infty}\PR\bigg(\sum_{i \in \mathcal{K}_n^T}|\mathscr{C}_{\sss (i)}|\times \mathrm{SP}(\mathscr{C}_{\sss (i)})>\varepsilon n^{\rho}\bigg)=0.
\end{equation}
  The proof is now complete for the case $\lambda > 1$ by combining \eqref{surplus-positivelamba1} and \eqref{surplus-positivelamba2}.
\end{proof}

\begin{proof}[Proof of Lemma~\ref{lem:sp-cv-n}]
We use a generic constant $C$ to denote a positive constant independent of $n,\delta,K$. 
Consider the graph exploration described in Algorithm~\ref{algo-expl}, but now we start by choosing vertex $V_n^*$ at Stage 0 and declaring all its half-edges active. 
The exploration process is still given by \eqref{defn:exploration:process} with $S_n(0)=\tilde{d}_{V_n^*}$.
Note that $\mathscr{C}(V_n^*)$ is explored when $\mathbf{S}_n$ hits zero, and the hitting time at zero gives $D(V_n^*)/2$.
For $H>0$,  let 
\begin{equation} \label{defn:gamma}
\gamma := \inf \{ l\geq 1: S_n(l)\geq H \text{ or }  S_n(l)= 0 \}\wedge 2\delta_K n^\rho.
\end{equation} 
Here, we let $\mathcal{A}$ be the intersection of all the events described in Lemma~\ref{lem:perc-degrees}, which are shown to hold with high probability.
Recall that we write $\mathscr{F}_l = \sigma (\mathcal{I}_i(l): i\in [n]) \cap \mathcal{A}$.  
Note that
\begin{equation}\label{exploration:super_martingale}
\begin{split}
 \exptc{S_n(l+1)-S_n(l)\mid \mathscr{F}_l}&= \sum_{i\in [n]}\tilde{d}_i\probc{i\notin \mathscr{V}_l, i\in \mathscr{V}_{l+1}\mid \left( \mathcal{I}_i^n(l)\right)_{i=1}^n} -2\\
 &= \frac{ \sum_{i\notin \mathscr{V}_l}\tilde{d}_i^2}{\tilde{\ell}_n-2l-1}-2\leq \frac{ \sum_{i\in [n]}\tilde{d}_i^2}{\tilde{\ell}_n-2l-1}-2\\:
 & =(\lambda-1) +\frac{2l+1}{\tilde{\ell}_n-2l-1}\times \frac{\sum_{i\in [n]}\tilde{d}_i^2}{\tilde{\ell}_n}   \leq 0,
\end{split}
\end{equation} uniformly over $l\leq 2\delta_K n^{\rho}$ for all small $\delta >0$ and large $n$, where the last step uses that $\lambda<1$. Therefore, $\{S_n(l)\}_{l= 1}^{2\delta_Kn^{\rho}}$ is a super-martingale. The optional stopping theorem now implies
  \begin{equation}
   \tilde{\mathbbm{E}}\left[\td_{V_n^*}\right] \geq \tilde{\mathbbm{E}}\left[S_n(\gamma)\right] \geq H \tilde{\mathbbm{P}}\left( S_n(\gamma) \geq H \right).
  \end{equation} Thus,
  \begin{equation} \label{eqn::bound_geq_H_at_stopping_time}
    \tilde{\mathbbm{P}}\left( S_n(\gamma) \geq H \right) \leq \frac{\tilde{\mathbbm{E}}[d_{V_n^*}]}{H}.
  \end{equation}
Put $H= n^{\rho}K^{1.1}/\sqrt{\delta}$. To simplify the writing, we  write $S_n[0,t]\in A$ to denote that $S_n(l)\in A,$ for all $ l\in [0,t]$.  Notice that
 \begin{equation}\label{surp:sup:less}\begin{split}
  &\probc{\surp{\mathscr{C}(V_n^*)}\geq K,D(V_n^*)\in (\delta_K n^{\rho},2\delta_Kn^{\rho})}\\
  &\leq \probc{S_n(\gamma)\geq H}+\probc{\surp{\mathscr{C}(V_n^*)}\geq K, S_n[0,2\delta_K n^{\rho}]< H, S_n[0,\delta_K n^{\rho}]>0}.
  \end{split}
 \end{equation}Now,
 \begin{equation}
  \begin{split}
   &\probc{\surp{\mathscr{C}(V_n^*)}\geq K, S_n[0,2\delta_K n^{\rho}]< H, S_n[0,\delta_K n^{\rho}]>0}\\
  &\leq \sum_{1\leq l_1<\dots< l_K\leq 2\delta_K n^{\rho}} \probc{\text{surpluses occur at times } l_1,\dots,l_K,  S_n[0,2\delta_K n^{\rho}]< H, S_n[0,\delta_K n^{\rho}]>0}\\
  &=\sum_{1\leq l_1<\dots<l_K\leq 2\delta_K n^{\rho}}\exptc{\ind{0<S_n[0,l_K-1]<H, \mathrm{SP}(l_K-1)=K-1}Y},
  \end{split}
 \end{equation}
 where
 \begin{equation}
 \begin{split}
  Y&=\probc{K^{th}\text{ surplus occurs at time }l_K,  S_n[l_K,2\delta_Kn^{\rho}]< H, S_n[l_K,\gamma]>0\mid \mathscr{F}_{l_K-1} }\\
  &\leq \frac{CK^{1.1}n^\rho}{\tilde{\ell}_n\sqrt{\delta}}\leq \frac{CK^{1.1}}{n^{\rho}\sqrt{\delta}}.
 \end{split}
 \end{equation}Therefore, using induction, \eqref{surp:sup:less} yields
 \begin{equation}\label{exploration:bounded:surplus}
 \begin{split}
  &\probc{\surp{\mathscr{C}(V_n^*)}\geq K, S_n[0,2\delta_K n^{\rho}]< H, S_n[0,\delta_K n^{\rho}]>0}\\
  &\leq C\bigg( \frac{K^{1.1}}{\sqrt{\delta}n^{\rho}}\bigg)^K\frac{(2\delta n^{\rho})^{K-1}}{K^{0.12(K-1)}(K-1)!}\sum_{l_1=1}^{2\delta_K n^{\rho}}\probc{D(V_n^*)|\geq l_1}\leq C \frac{\delta^{K/2}}{K^{1.1}n^{\rho}}  \exptc{D(V_n^*)},
  \end{split}
 \end{equation}where we have used the fact that $\#\{1\leq l_2,\dots,l_K\leq 2\delta n^{\rho}\}=(2\delta n^{\rho})^{K-1}/(K-1)!$ and Stirling's approximation for $(K-1)!$ in the last step. Since $\lambda <1$, we can use \eqref{path-count-ub-main} to conclude that, for all sufficiently large $n$,
 \begin{equation} \label{expectation:random:vert:comp}
  \exptc{D(V_n^*)-1}\leq C,
 \end{equation} with high probability for some constant $C>0$.
 Thus, we get the desired bound for \eqref{surp:sup:less}.
  The proof of Lemma~\ref{lem:sp-cv-n} is now complete.
\end{proof}

\end{document}